\documentclass[11pt,leqno]{amsart}
\usepackage[top=0.8in, bottom=0.7in, left=0.8in, right=0.8in]{geometry}
\usepackage{amscd}
\usepackage{epsf}
\usepackage{color}
\usepackage[normalem]{ulem}
\usepackage[symbol]{footmisc}
\usepackage{tikz}
\usepackage{soul}
\usetikzlibrary{decorations.markings,arrows.meta}
\usepackage{amssymb}
\usepackage{amsfonts}
\usepackage{enumerate}
\usepackage{latexsym}
\usepackage{verbatim}
\usepackage{mathrsfs,mathtools}
\usepackage{bbm}
\usepackage[shortlabels]{enumitem}
\usepackage[backref=page]{hyperref}
\renewcommand*{\backref}[1]{}
\renewcommand*{\backrefalt}[4]{%
	\ifcase #1 (Not cited).%
	\or        (Cited on page~#2).%
	\else      (Cited on pages~#2).%
	\fi}
\hypersetup{colorlinks=true,linkcolor=blue,citecolor=blue,urlcolor=black}
\newcommand{\iprod}{\mathbin{\lrcorner}}
\newcommand{\n}{\mathfrak{n}}
\renewcommand{\r}{\mathfrak{r}}
\renewcommand{\v}{\mathfrak{v}}
\def\zz{z}
\def\ee{e}

\newcommand{\Z}{\mathbb{Z}}

\newcommand{\R}{\mathbb{R}}
\newcommand{\C}{\mathbb{C}}
\newcommand{\w}{\wedge}
\newcommand{\G}{\mathrm{G}}
\newcommand{\f}{\varphi}

\renewcommand{\epsilon}{\varepsilon}

\newtheorem{thm}{Theorem}[section]
\newtheorem{cor}[thm]{Corollary}
\newtheorem{prop}[thm]{Proposition}
\newtheorem{lemma}[thm]{Lemma}

\theoremstyle{definition}
\newtheorem{defn}[thm]{Definition}
\newtheorem{ex}[thm]{Example}

\theoremstyle{remark}
\newtheorem{rmk}[thm]{Remark}

\numberwithin{equation}{section}

\begin{document}

\title[{The heterotic G$_2$-system on 2-step nilmanifolds endowed with principal torus bundles}]{The heterotic G$_\mathbf{2}$-system on 2-step nilmanifolds endowed with principal torus bundles}

\author{Andrei Moroianu}
\address{Université Paris-Saclay, CNRS,  Laboratoire de mathématiques d'Orsay, 91405, Orsay, France, 
and Institute of Mathematics “Simion Stoilow” of the Romanian Academy, 21 Calea Grivitei, 010702 Bucharest, Romania}
\email{andrei.moroianu@math.cnrs.fr}

\author{Alberto Raffero}
\address{Dipartimento di Matematica ``G. Peano'' \\ Universit\`a degli Studi di Torino\\ Via Carlo Alberto 10\\10123 Torino\\ Italy}
\email{alberto.raffero@unito.it}

\author{Luigi Vezzoni}
\address{Dipartimento di Matematica ``G. Peano'' \\ Universit\`a degli Studi di Torino\\ Via Carlo Alberto 10\\10123 Torino\\ Italy}
\email{luigi.vezzoni@unito.it}
\begin{abstract}

We study the geometric heterotic G$_2$-system on 7-dimensional 2-step nilmanifolds $M=\Gamma\backslash N$ endowed with principal torus bundles and with a prescribed flat tangent bundle instanton. 
We first prove that every invariant G$_2$-structure solving the system must be coclosed when the dimension of the commutator of $N$ is $1$ or $2$, 
and under an additional calibration assumption when the dimension is $3$.
Then, we discuss the existence of solutions for all possible isomorphism classes of 7-dimensional 2-step nilpotent Lie algebras, 
and we provide examples with constant dilaton function.

\end{abstract}
\subjclass[2020]{53C10, 53C30} 
\keywords{Heterotic system, $\G_2$-structure with torsion, nilpotent Lie algebra, torus bundle} 
\maketitle

\section{Introduction}
 In this paper, we study the heterotic $\G_2$-system \cite{CFT,dSGFLSE, dG, X,X1,DLS,X2, FIUVa,FIUV,fist, FI,FI2,GS,GMPW,GMPW2,REF1,LS,mac1,mac2,Nolle} on $2$-step nilmanifolds 
endowed with principal torus bundles. 

Consider a 7-dimensional manifold $M$ endowed with a $\G_2$-structure $\f\in\Omega^3(M)$ and a $K$-principal bundle $P\to M$, 
where $K$ is a $k$-dimensional real Lie group. 
Let $\theta\in\Omega^1(P;\mathfrak{k})$ be a connection on $P$ with curvature $F_\theta\,\in \Omega^2(M,\mathrm{ad}(P))$, 
and let $\nabla$ be an affine connection on $M$ with curvature $R_\nabla\,\in \Omega^2(M,\mathrm{End}(TM))$.
 The {\em heterotic $\G_2$-system} is given by the following set of equations \cite{X,DLS}:
\begin{equation}\label{hetsysIntro0}
\begin{aligned}
d*\f&=4\tau_1\wedge *\f, \\
F_{\theta}\wedge *\f&=0,\\
R_{\nabla}\wedge *\f&=0,\\
dH_{\varphi}&=\frac{\alpha'}{4}\left(\mathrm{Tr}( F_{\theta}\wedge F_{\theta}) - {\rm Tr}(R_{\nabla}\wedge R_{\nabla})\right) \,,
\end{aligned}
\end{equation}
where $*$ is the Hodge operator determined by the $\G_2$-structure $\f$, 
$\tau_1 \coloneqq  \tfrac1{12}*(\f\w* d\f) \in\Omega^1(M)$  is the intrinsic torsion $1$-form of $\f$,
$H_\f$ is the following 3-form
\[
H_{\varphi}\coloneqq   \frac{1}{6} {*(\varphi\w d\varphi)}  \varphi - * d\varphi + *(4\tau_1 \wedge \varphi),
\]
and $\alpha'$ is a positive constant. 
Note that the trace of $F_\theta\wedge F_\theta$ is computed using a non-degenerate $\mathrm{Ad}(K)$-invariant bilinear form $\langle\cdot,\cdot\rangle_{\mathfrak{k}}$ on the Lie algebra $\mathfrak{k}$ of $K$. 

\smallskip

In the mathematical literature on the heterotic $\G_2$-system (see e.g.~\cite{dSGFLSE, X,DLS}), the affine connection $\nabla$ entering \eqref{hetsysIntro0} is regarded as part of the geometric data of the system and is required to satisfy 
the $\G_2$-instanton condition. 
From the perspective of heterotic supergravity, the choice of tangent-bundle connection is part of the formulation of the effective theory.  
In the present work we will adopt the geometric viewpoint and study solutions of \eqref{hetsysIntro0} for a prescribed instanton connection $\nabla$.

\smallskip
The heterotic $\G_2$-system arises in the study of $\mathcal{N} = 1$ supersymmetric compactifications of the heterotic string to three dimensions in the supergravity limit, 
corresponding to a leading-order approximation in the  $\alpha'$-expansion.  The $\mathcal{N} = 1$ supersymmetry imposes that the dilaton field $\phi$ is a potential for the intrinsic torsion 
1-form of the $\G_2$-structure, $\tau_1=\tfrac12d\phi$, 
 and that the torsion form $\tau_0 \coloneqq \tfrac17*(\f\wedge d\f)\in C^\infty(M)$ is constant.  
Moreover, $-\tau_0^2$ is a non-zero multiple of the {\em cosmological constant} 
of the $3$-dimensional spacetime, which is constrained to be Minkowski (if $\tau_0=0$) 
or anti-de Sitter (if $\tau_0\neq0$) by supersymmetry and maximal symmetry \cite{DLS,X}.  
The system \eqref{hetsysIntro0} admits an alternative description in terms of Killing spinor equations, see e.g.~\cite{CFT,dSGFLSE,X}.  
The complete bosonic equations of heterotic supergravity were presented in \cite{MMS22} and \cite{MMS24}, where the first compactification backgrounds not locally isomorphic to a supersymmetric compactification background were obtained.

\smallskip

The first two equations of the heterotic $\G_2$-system \eqref{hetsysIntro0} are equivalent to the {\em gravitino equation} \cite[Lemma 2.12]{dSGFLSE}. 
The first one states that one of the components of the intrinsic torsion of $\varphi$, namely the intrinsic torsion 2-form $\tau_2$, is zero. 
This condition characterizes the existence of a connection on $TM$ that preserves the 
$\G_2$-structure and has totally skew-symmetric torsion \cite[Thm.~4.7]{FI}; the latter corresponds to the $3$-form $H_\f$ under the musical isomorphism.   
$\G_2$-structures satisfying the condition $\tau_2=0$ are known as {\em integrable $\rm{G}_2$-structures} \cite{FI,FI2}
or {\em $\rm{G}_2$-structures with torsion} (shortly {\em $\G_2T$-structures}) \cite{ChSw} in the literature. 
The  second and third equations in \eqref{hetsysIntro0} state that the connections $\theta$ and $\nabla$ are ${\rm G}_2$-instantons \cite{DT}.
The last equation is known as the {\em heterotic Bianchi identity}. 

\medskip 

Following the recent work \cite{dSGFLSE}, the system \eqref{hetsysIntro0} can be seen as a special instance of a more general one, defined as follows for a $7$-manifold $M$ endowed with a $\G_2$-structure $\varphi$ and a 
$\hat K$-principal bundle $\hat P\to M$:
\begin{equation}\label{hetsysIntro2}
\begin{aligned}
d*\f&=4\tau_1\wedge *\f, \\
F_{\hat\theta}\wedge *\f&=0,\\
dH_{\varphi}&=\langle F_{\hat\theta}\wedge F_{\hat\theta}\rangle_{\hat{\mathfrak{k}}}\,,
\end{aligned}
\end{equation} 
where $\hat\theta\in\Omega^1(\hat P;\hat{\mathfrak{k}})$ is a connection on $\hat P$ with curvature $F_{\hat\theta}$, 
and $\langle \cdot,\cdot\rangle_{\hat{\mathfrak{k}}}$ is an $\mathrm{Ad}(\hat K)$-invariant non-degenerate symmetric bilinear form  
on the Lie algebra $\hat{\mathfrak{k}}$ of $\hat K$. 
Indeed, every solution of the heterotic $\G_2$-system \eqref{hetsysIntro0} determines a solution of \eqref{hetsysIntro2} by choosing the Lie group $\hat K \coloneqq K\times \mathrm{GL}(7,\R)$, the principal $\hat K$-bundle $\hat P \coloneqq P \times_M F(M)$, where $F(M)$ is the frame bundle of $M$, the connection form
$\hat\theta \coloneqq (\theta,\theta_\nabla)$, where $\theta_\nabla\in\Omega^1(F(M),\mathfrak{gl}(7,\R))$ is the connection form on $F(M)$ 
corresponding to the covariant derivative $\nabla$ on the associated vector bundle $TM$, \
and $\langle\cdot,\cdot\rangle_{\hat{\mathfrak{k}}} =\frac{\alpha'}{4}\left( \langle\cdot,\cdot\rangle_{\mathfrak{k}}\oplus \langle\cdot,\cdot\rangle_{\mathfrak{gl}(7,\R)} \right)$, with $\langle A,B\rangle_{\mathfrak{gl}(7,\R)}:=-\mathrm{Tr}(AB)$.

By \cite[Thm.~4.9]{dSGFLSE}, for every solution of \eqref{hetsysIntro2} on a compact 7-manifold $M$, 
the function $\lambda \coloneqq \tfrac{7}{12}\tau_0$ is constant. 
This holds, in particular, for solutions of \eqref{hetsysIntro0}.

Special solutions to system \eqref{hetsysIntro2} are provided by {\em torsion-free} ${\rm G}_2$-structures, 
which are characterized by the conditions $d\f=0$ and $d*\f=0$. 
In this case, $H_{\varphi}$ vanishes and any flat bundle gives a solution to the system. 
More generally, ${\rm G}_2T$-structures satisfying the condition
$dH_{\varphi}=0$ 
provide solutions to the system when the connection $\theta$ is flat.     
These $\G_2$-structures are known as {\em strong} ${\rm G}_2$-structures with torsion \cite{ChSw,FMR},  
in analogy to the terminology used in the Hermitian case.  
In this context, they play the role that Calabi-Yau structures play in the Hull-Strominger system. 

\smallskip

In the present paper, we study the heterotic $\G_2$-system \eqref{hetsysIntro0} in the particular case where $M=\Gamma\backslash N$ is a $2$-step nilmanifold, 
$N$ being a simply connected $2$-step nilpotent 7-dimensional Lie group and $\Gamma$ a cocompact lattice, 
the Lie group $K=\mathbb{T}^k$ is a $k$-dimensional torus, and the affine connection $\nabla$ is flat.  
The choice of a flat tangent bundle instanton is motivated by the parallelizability of nilmanifolds and allows for an explicit analysis of the system. 
We emphasize, however, that in this paper we do not investigate the relation between this flat connection and the Hull connection appearing in the standard first-order formulation of heterotic supergravity. 
Consequently, while our constructions provide solutions of the heterotic $\G_2$-system \eqref{hetsysIntro0} in the above geometric sense, we do not claim that they determine heterotic compactification backgrounds in a specific supergravity scheme.

Under these assumptions and after absorbing the $\frac{\alpha'}{4}$ factor in the scalar product $\langle\cdot,\cdot\rangle_{{\mathfrak{t}^k}}$, the heterotic $\G_2$-system \eqref{hetsysIntro0} becomes:
\begin{equation}\label{hetsysIntro}
\begin{aligned}
d*\f&=4\tau_1\wedge *\f, \\
F_{\theta}\wedge *\f&=0,\\
dH_{\varphi}&=\langle F_{\theta}\wedge F_{\theta}\rangle_{{\mathfrak{t}^k}}\,.
\end{aligned}
\end{equation}
In this geometric setting, we refer to an {\em invariant} solution to the system as a solution where $\varphi$ and the curvature form $F_{\theta}$ (which can be thought of as a 2-form on $M$ with values in $\mathfrak{t}^k\cong\R^k$) 
are induced by left-invariant forms on $N$.
In this way, the system can be regarded as a system for exterior forms on the Lie algebra $\n$ of $N,$ 
and the analysis can be performed in an algebraic fashion. 

Let $(k_+,k_{-})$, with $k_+ + k_- = k,$ be the signature of the  
non-degenerate symmetric bilinear form $\langle \cdot,\cdot\rangle_{\mathfrak{t}^k}$ on the Lie algebra of $\mathbb{T}^k$. 
Note that if a solution of $\eqref{hetsysIntro}$ exists for some signature $(k_+,k_{-})$, then one can obtain solutions for any signature $(k'_+,k'_{-})$ with $k'_+\ge k_+$ and $k'_-\ge k_-$ simply by multiplying $P$ with a trivial bundle with flat connection. We will thus be interested in constructing solutions with smallest possible $k_+$ and $k_-$. 
Note also that any solution must have $k_+\geq1$. This follows from \cite[Thm.~3.9]{dSGFLSE} and will be discussed in Section \ref{sec:hetG2}. 

\medskip
Nilpotent Lie algebras $\mathfrak n$ of dimension $7$ are classified \cite{Gon}. In the $2$-step case, the possible dimensions of the commutator $\mathfrak n'\coloneqq[\mathfrak n,\mathfrak n]$ of $\mathfrak n$ are 1, 2 or $3$. In  the case $\dim(\n')=3$ we focus on ${\rm G}_2$-structures $\varphi$ which calibrate  $\n'$, i.e., $\varphi$ restricts to a volume form on $\n'$. This is a technical assumption, already considered in \cite{DMR}, which allows us to have a global description of the $3$-form $\varphi$.  

\medskip 
As a preliminary result of independent interest, we prove the following:

\begin{prop}\label{coclosed}
Let $\varphi$ be an invariant ${\rm G}_2T$-structure  
on a $2$-step nilpotent Lie algebra $\n$. Then $\varphi$ is coclosed, i.e., $d*\f=0$, 
if either $\dim(\n')\leq2$, or $\dim(\n')=3$ and $\varphi$ calibrates $\n'$.    
\end{prop}
Note that the result does not hold without the calibration assumption when $\dim(\n')=3$:
on 2-step nilpotent Lie algebras $\n$ with $\dim(\n')=3$, it is possible to construct invariant ${\rm G}_2T$-structures $\varphi$ (not calibrating $\n'$) which are not coclosed, see Example \ref{ex:notcal}.

This result allows us to restrict our analysis to $2$-step nilpotent Lie algebras admitting coclosed $\G_2$-structures; 
such Lie algebras are classified \cite{BFF,DMR}. 
In detail: if $\dim(\n')=1$, the only three isomorphism classes of $7$-dimensional $2$-step nilpotent Lie algebras are 
$\mathfrak{h}_3\oplus\R^4$, $\mathfrak{h}_5\oplus\R^2$ and 	${\mathfrak h_{7}}$ (here $\mathfrak h_{n}$ denotes the $n$-dimensional Heisenberg Lie algebra) and each class admits coclosed ${\rm G}_2$-structures; if $\dim(\n')=2$ there are four isomorphism classes of $2$-step nilpotent Lie algebras admitting 
a coclosed ${\rm G}_2$-structure ($\n_{5,2}\oplus\R^2$,  $\mathfrak{h}_3\oplus\mathfrak{h}_3\oplus\R$, $\mathfrak{h}_3^{\C}\oplus\R$, $\n_{6,2}\oplus\R$); if $\dim(\n')=3$ all seven isomorphism classes of $2$-step nilpotent Lie algebras admit  coclosed ${\rm G}_2$-structures (see Appendix \ref{2stepnilclass} for a detailed description of these Lie algebras).  

When $\dim(\n')=3$ and the $\G_2$-structure is coclosed, 
our technical assumption is not restrictive from the metric point of view:
any metric induced by a coclosed $\G_2$-structure on $\n$ is also induced by a coclosed  $\G_2$-structure calibrating $\n'$ \cite[Lemma 4.9]{DMR}.

Since a coclosed $\G_2$-structure has $\tau_1=0$, any solution to \eqref{hetsysIntro} in this 
setting provides a solution to the heterotic $\G_2$-system \eqref{hetsysIntro0} with flat affine connection $\nabla$ and constant dilaton function.

\smallskip 

Our first main result focuses on invariant solutions in the case where $\dim(\n')=1$.

\begin{thm}\label{main1}
Let $M=\Gamma\backslash N$ be a $7$-dimensional $2$-step nilmanifold such that $\dim(\n')=1$.
\begin{itemize}
\item[$(i)$] If $\n\cong{\mathfrak h_{3}\oplus \R^{4}}$, the ${\rm G}_2$-system \eqref{hetsysIntro} on $M$ has no invariant solutions with $\lambda=0$;
\vspace{0.1cm}
\item[$(ii)$] If $\n\cong{\mathfrak h_{5}\oplus \R^{2}}$, there are solutions with $\lambda=0$, non-trivial torus bundle $P$, and with signature $(k_+,k_-)=(1,0)$;
\vspace{0.1cm}
\item[$(iii)$] If $\n\cong{\mathfrak h_{7}}$ there are solutions with $\lambda=0$, signature $(k_+,k_-)=(1,0)$ and non-zero curvature. Each solution with $\lambda=0$ and either $k_+=1$ or $k_-=0$ has trivial torus bundle, and there are solutions with $\lambda=0$, non-trivial torus bundle, and $(k_+,k_-)=(2,1)$;
\vspace{0.1cm}
\item[$(iv)$] The ${\rm G}_2$-system \eqref{hetsysIntro} on $M$ has no invariant solutions with $\lambda\neq 0$ and $k_+=1$, while there are solutions to the system with $\lambda\neq 0$, $(k_+,k_-)=(2,0)$ 
and non-trivial torus bundle for each of the three isomorphism classes of $\n$. 
  
\end{itemize}
\end{thm}

Our second main result focuses on the existence of invariant solutions in the case where $\dim(\n')\geq 2$.   

\begin{thm}\label{main2}
Let $M=\Gamma\backslash N$ be a $7$-dimensional $2$-step nilmanifold such that $\dim(\n')\geq2$. 

\begin{itemize}
\item[$(i)$] The ${\rm G}_2$-system \eqref{hetsysIntro} has solutions with $\lambda=0$ and $k_{+}=1$ for every isomorphism class of $\n$ admitting a coclosed ${\rm G}_2$-structure.  

\vspace{0.1cm}
\item[$(ii)$] If $\dim(\n')=2$, there are no invariant solutions to the ${\rm G}_2$-system \eqref{hetsysIntro} on $M$  
with $\lambda\neq 0$ and $k = k_++k_-\leq2$.
Moreover, if $\n\cong\n_{5,2}\oplus\R^2 $ then the ${\rm G}_2$-system on $M$ admits invariant solutions with $\lambda \neq 0$ and $k\geq3$.

\vspace{0.1cm}
\item[$(iii)$] If $\dim(\n')=3$, there are no invariant solutions to the ${\rm G}_2$-system \eqref{hetsysIntro} on $M$ with $\lambda\neq 0$, $k=1$, 
and $\varphi$ calibrating $\n'$. 
Moreover, the ${\rm G}_2$-system on $M$ admits invariant solutions with $\lambda \neq 0$, $k=2$ and $\varphi$ calibrating $\n'$ if and only if $\n\cong\n_{6,3}\oplus\R$. 
\end{itemize}
\end{thm}

The proof of the results is obtained as follows. 

\smallskip 
If $\dim(\n')=1$, then every $\G_2$-structure $\varphi$ on $\n$ induces an ${\rm SU}(3)$-structure $(\omega,\Omega_+)$ on the orthogonal complement $\v$ of $\n'$.
The isomorphism class of $\n$ is determined by the rank of $\alpha\coloneq dz^{\flat}$, where $z$ is a unit-norm generator of 
$\n'$. The $\G_2$-system \eqref{hetsysIntro} can then be regarded as a system of forms on $\v$ and the decomposition of $2$- and $3$-forms in ${\rm SU}(3)$-modules can be used to simplify the equations. In particular, the system can be reduced to the equation
$$
\alpha_0\w\alpha_0-4\lambda^2  \omega \w \omega =\sum_{r=1}^k \epsilon_r \sigma^r \w \sigma^r\,, 
$$
on  $\R^6$ with the standard ${\rm SU}(3)$-structure, where $\alpha_{0}\in \Lambda^{(1,1)}(\R^6)^*$, $\lambda\in \R$, 
$\epsilon_1\dots,\epsilon_k\in \R\smallsetminus\{0\}$, $\sigma_1,\dots,\sigma_k\in \Lambda^{(1,1)}_0(\R^6)^*$. 

\smallskip
If $\dim(\n')=2$ or if $\dim(\n')=3$ and $\varphi$ calibrates $\n'$, one can split $\n$ as $\v\oplus \langle z_1,z_2,z_3\rangle$ with $\n'\subseteq \langle z_1,z_2,z_3\rangle$, so that the $4$-dimensional space $\v$ inherits an ${\rm SU}(2)$-structure $\omega_1,\omega_2,\omega_3$. 
In this case the system can be written in terms of forms on $\v$ and the decompositions of forms in ${\rm SU}(2)$-modules allows us to simplify the equations.

\smallskip

The paper is organized as follows. Section \ref{pre} contains preliminary material about $\G_2$-structures and the $\G_2$-system \eqref{hetsysIntro}. 
Section \ref{G2system} shows how to reformulate the $\G_2$-system in the invariant case in terms of invariant forms (see system \ref{invsystem}). 
Section \ref{subsec1} focuses on the case $\dim(\n') = 1$ and provides a proof of Theorem \ref{main1}, while Section \ref{2or3} studies the cases $\dim(\n') = 2$ and $\dim(\n') = 3$ and contains the proof of Theorem \ref{main2}. 
Section \ref{open} collects some problems arising from the results of the present paper which remain open. 
The paper also contains two appendices: in the first one we recall how to construct principal torus bundles on $2$-step nilmanifolds from integral forms, 
and in the second one we recall the classification of $7$-dimensional $2$-step nilpotent Lie algebras.

\smallskip
Table \ref{tab1} summarizes the results of Theorem \ref{main1} ($\dim(\n')=1$). 
The results of Theorem \ref{main2} ($\dim(\n')= 2$ or $3$) are summarized in Table \ref{tab2} only for solutions with  $\lambda\neq0$, 
as in this case solutions with $\lambda=0$ always exist for every isomorphism class of 2-step nilpotent Lie algebras admitting coclosed $\G_2$-structures.

\begin{table}[ht]
\begin{center}
\renewcommand{\arraystretch}{1.4}
\begin{tabular}{|c|c|c|c|c|}
\hline
${\mathfrak n}$ & $\lambda$ & $\exists\, \mathrm{sol}$ & $(k_+, k_-)$ & \mbox{Solutions with non-trivial Torus Bundle} \\
\hline

$\mathfrak{h}_3 \oplus \mathbb{R}^4$ & $\lambda = 0$ & $\times$ & -- & -- \\
\cline{2-5}
 & $\lambda \neq 0$ & $\checkmark$ & $k_+ \geq 2,\; k_- \geq 0$ & {$\checkmark$} \\
\hline

$\mathfrak{h}_5 \oplus \mathbb{R}^2$ & $\lambda = 0$ & $\checkmark$ & $k_+ \geq 1,\; k_- \geq 0$ & $\checkmark$\\
\cline{2-5}
 & $\lambda \neq 0$ & $\checkmark$ & $k_+ \geq 2,\; k_- \geq 0$ & {$\checkmark$}  \\
\hline

$\mathfrak{h}_7$ & $\lambda = 0$ & $\checkmark$ & $k_+ \geq 1,\; k_- \geq 0$ & 
\begin{tabular}{c}
only if $k_+\geq 2$ and  $k_-\geq1$ 
\end{tabular} \\
\cline{2-5}
 & $\lambda \neq 0$ & $\checkmark$ & $k_+ \geq 2,\; k_- \geq 0$ & {$\checkmark$} \\
\hline
\end{tabular}
\vspace{0.1cm}
\caption{Summary of the results for $\dim(\n')=1$}\label{tab1}
\end{center}
\end{table}

\vspace{-0.5cm}

\begin{table}[h]
\centering
\renewcommand{\arraystretch}{1.4}
\begin{tabular}{|c|c|c|}
\hline
$(\mathfrak n,\varphi)$ & $\nexists$ sol & $\exists$ sol \\
\hline

$\dim(\mathfrak n') = 2$ 
& if $k_+ + k_- \leq 2$ 
& for $\mathfrak n \cong \mathfrak n_{5,2} \oplus \mathbb{R}^2$ \\
& 
& with $k_+ \geq 3$, $k_-\geq 0$, or $k_+ \geq 2$, $k_- \geq 1$\\
\hline

$\dim(\mathfrak n') = 3$ 
and $\varphi|_{\mathfrak n'}=\mathrm{vol}_{\n'}$
& if $(k_+,k_-) = (1,0)$ 
& if and only if $\mathfrak n \cong \mathfrak n_{6,3} \oplus \mathbb{R}$ when $(k_+,k_-) = (2,0)$ \\
&  
&  open when $(k_+,k_-)=(1,1)$ or  $k_+ + k_- \geq 3$ \\
\hline

\end{tabular}

\vspace{0.1cm}
\caption{Summary of the results for $\dim(\mathfrak n') = 2,3$ and $\lambda\neq0$}\label{tab2}
\end{table}

Finally, we stress that the purpose of the present work is the construction and classification of solutions to the system \eqref{hetsysIntro} on 2-step nilmanifolds endowed with principal torus bundles. 
The relation of these solutions to particular formulations of heterotic supergravity, including the identification of the tangent-bundle connection with distinguished physical choices and the analysis of 
higher-order $\alpha'$-corrections, lies beyond the scope of the present work.

\bigskip

\noindent {\sc Acknowledgments.} 
This work was supported by the ``National Group for Algebraic and Geometric Structures, and their Application'' (GNSAGA - INDAM). 
A.M. was partially supported by the PNRR-III-C9-2023-I8 grant CF 149/31.07.2023 {\em Conformal Aspects of Geometry and Dynamics}.
A.R.~was also supported by the project PRIN 2022AP8HZ9 ``Real and Complex Manifolds: Geometry and Holomorphic Dynamics". 
L.V. was also supported by the project PRIN 20225J97H5 ``Differential-geometric aspects of manifolds via Global Analysis''. The authors would like to thank  Mario Garcia-Fern\'andez, Jason Lotay, and Carlos Shahbazi for useful comments. 
They are also grateful to the anonymous referee for their suggestions and remarks.

\section{Preliminaries}\label{pre}
\subsection{Basic properties of $\G_2$-structures} 
Let $M$ be a $7$-dimensional manifold with a $\G_2$-structure $\varphi \in \Omega^3(M)$. 
Recall that $\varphi$ determines a Riemannian metric $g_\varphi$ and an orientation on $M$. We denote by $*$ the corresponding Hodge operator and let
$\psi\coloneqq *\varphi$. 

At each point $x$ of $M$ there exists an oriented $g_\varphi$-orthonormal basis $\{e_1,\ldots,e_7\}$ of $T_xM$ with dual basis $\{e^1,\ldots,e^7\}$ such that
\begin{equation}\label{standard}
\begin{split}
    \varphi_x &= e^{127}+e^{347}+e^{567} + e^{135} -e^{245} -e^{146}-e^{236},\\
\psi_x &= e^{1234}+e^{1256}+e^{3456} +e^{1367}+e^{1457}+e^{2357}-e^{2467}. 
\end{split}
\end{equation}
These bases are said to be {\em adapted} to the G$_2$-structure $\varphi$. 

The spaces of $4$- and $5$-forms on $M$ decompose as follows into $\G_2$-irreducible summands
$$
\Omega^4(M)=\,\Omega^4_1\oplus \Omega^4_7\oplus \Omega^4_{27}\,,\quad 
\Omega^5(M)=\Omega^5_7\oplus \Omega^5_{14}\,,
$$
where 
\[
\begin{split}
\Omega^4_1&=\{f\,\psi\,\,:\,\, f\in C^{\infty}(M)\}\,,\qquad
\Omega^4_7=\{\alpha\wedge \varphi\,\,:\,\,\alpha \in \Omega^1(M) \}\,,\\ 
\Omega^4_{27}&=\{*\gamma\in\Omega^4(M) \,\,:\,\, \gamma\wedge\varphi=0,~\gamma\wedge\psi=0  \},\\
\Omega^5_{7}&=\{\alpha \wedge \psi\,\,:\,\,\alpha \in \Omega^1(M)\}\,,\qquad
\Omega^5_{14}=\{*\beta\in \Omega^5(M)\,\,:\,\, \beta\wedge\varphi = -*\beta\}\,.
\end{split}
\]
The exterior differentials of $\varphi$ and $\psi$ decompose accordingly as 
\begin{eqnarray}\label{g2torsion}
    \begin{aligned}
d\varphi=&\,\tau_0\,\psi+3\tau_1\w\varphi + *\tau_3,\\
d\psi=&\,4\tau_1\w \psi-*\tau_2\,,
\end{aligned}
\end{eqnarray}
where $\tau_0\in\ C^\infty(M)$, $\tau_1\in\Omega^1(M)$, $\tau_2\in\Omega^2_{14} = *\Omega^5_{14}$, and $\tau_3\in\Omega^3_{27}=*\Omega^4_{27}$ 
are called the {\em intrinsic torsion forms} of $\f$, 
and the presence of $\tau_1$ in both expressions is shown in \cite[Prop.~1]{Bry}. 

A G$_2$-structure $\varphi$ is said to be
\begin{itemize}
    \item {\em torsion-free} if all intrinsic torsion forms $\tau_0,\tau_1,\tau_2,\tau_3$ vanish, 
            namely if $\f$ is both closed and coclosed;
    \item {\em coclosed} if $d*\f=0$, namely if both $\tau_1$ and $\tau_2$ are zero;
    \item $\G_2T$ ($\G_2$ {\em with torsion}) if $\tau_2=0$. 
\end{itemize}

\medskip
From  \cite[Thm. 4.7]{FI}, a 7-manifold $M$ endowed with a ${\rm G}_2$-structure $\varphi$ admits a G$_2$-connection $\nabla$ with totally skew-symmetric torsion if and only if $\tau_2=0$. 
In such a case, $\nabla$ is unique and its torsion $3$-form $H_\varphi$ is given by 
\begin{equation}\label{hphi}
    H_{\varphi}=   \frac{1}{6}\tau_0\, \varphi -\tau_1\iprod\psi -\tau_3\,.
\end{equation}

For later use, we make the following observation:
\begin{rmk}\label{tau2tau1}
    If $\tau_1=0$, then the first equation in \eqref{g2torsion} implies $H_\f= \frac{7}{6}\tau_0\, \varphi-*d\f$.
\end{rmk}

\subsection{The $\G_2$-system \eqref{hetsysIntro2}} \label{sec:hetG2}

Let $M$ be a compact connected $7$-dimensional manifold. Consider a $\G_2$-structure $\f\in\Omega^3(M)$ with dual 4-form $\psi\coloneqq *\f$ and intrinsic torsion forms $\tau_0,\tau_1,\tau_2,\tau_3$ defined in \eqref{g2torsion}, a $\hat{K}$-principal bundle $\hat{P}\to M$, where $\hat{K}$ is a $k$-dimensional real Lie group  with Lie algebra $\hat{\mathfrak k}$, a connection 1-form $\hat{\theta}$ on $\hat{P}$ with curvature $F_{\hat{\theta}}$, and an $\mathrm{Ad}(\hat{K})$-invariant non-degenerate symmetric bilinear form
$$
\langle \cdot,\cdot\rangle_{\hat{\mathfrak{k}}} \colon  \hat{\mathfrak k}\times \hat{\mathfrak k}\to \R. 
$$

Recall from \cite{dSGFLSE} that the quadruple $(\varphi,\hat{P},\hat{\theta},\langle \cdot,\cdot\rangle_{\hat{\mathfrak{k}}})$ 
is a solution to the {\em$\G_2$-system} \eqref{hetsysIntro2} if 
the following equations are satisfied:
\begin{equation}\label{eq:HetG2sys0}
\begin{aligned}
\tau_2&=0,\\
F_{\hat{\theta}}\wedge \psi&=0,\\
dH_{\varphi}&=\langle F_{\hat{\theta}}\wedge F_{\hat{\theta}}\rangle_{\hat{\mathfrak{k}}} \,,
\end{aligned}
\end{equation}
where the 3-form $H_{\varphi}$ is defined by \eqref{hphi}, and the extension of the bracket $\langle\cdot,\cdot\rangle_{\hat{\mathfrak{k}}}$ 
to $\hat{\mathfrak{k}}$-valued differential forms on $M$ is defined as follows. 
Let $\{t_1,\ldots,t_k\}$ be a basis of $\hat{\mathfrak{k}}$. 
For every $\hat{\mathfrak{k}}$-valued forms $\alpha = \sum \alpha^i\otimes t_i$ and $\beta = \sum \beta^j\otimes t_j$, where $\alpha^i\in\Omega^r(M)$ and 
$\beta^j\in\Omega^s(M)$, we set
\[
\langle \alpha \w \beta \rangle_{\hat{\mathfrak{k}}}\coloneqq \sum_{1\leq i,j\leq k} \alpha^i \w \beta^j \langle t_i, t_j \rangle_{\hat{\mathfrak{k}}}. 
\]
In particular, if we consider an orthogonal basis $\{t_1,\ldots,t_k\}$ of $\hat{\mathfrak{k}}$ and let $\epsilon_i\coloneqq \langle t_i,t_i\rangle_{\hat{\mathfrak{k}}}$, 
then  
\[
\langle \alpha \w \beta \rangle_{\hat{\mathfrak{k}}} = \sum_{i=1}^k \epsilon_i\, \alpha^i \w \beta^i.  
\]

As recalled in the introduction, the intrinsic torsion form $\tau_0$ of the $\G_2$-structure $\f$ is constant for every solution 
of \eqref{eq:HetG2sys0}, and we henceforth introduce the notation $\lambda\coloneqq\frac{7}{12}\tau_0$
as in \cite{dSGFLSE}.

The next result shows how $\lambda$ is related to other quantities, such as the scalar curvature of $g_\f$. 
\begin{thm}{\cite[Thm.~3.9]{dSGFLSE}}\label{thm:dSGFLSE}
    Let $(\varphi, \hat{P},\hat{\theta},\langle\cdot,\cdot\rangle_{\hat{\mathfrak{k}}})$ be a solution to the $\mathrm{G}_2$-system \eqref{eq:HetG2sys0} on a $7$-manifold $M.$ Then 
    \[
        \mathrm{Scal}_{g_\varphi} - \frac12|H_\varphi|^2 + |F_{\hat{\theta}}|_{\hat{\mathfrak{k}}}^2-8d^*\tau_1-16|\tau_1|^2 = 4\lambda^2,
    \]
    where $|F_{\hat{\theta}}|_{\hat{\mathfrak{k}}}^2 \coloneqq * \langle F_{\hat{\theta}}\w *F_{\hat{\theta}}\rangle_{\hat{\mathfrak{k}}}$. 
\end{thm}
We observe that this theorem readily implies the following. 
\begin{cor}\label{cor:bnegdef}
    Let $(\varphi, \hat{P},\hat{\theta},\langle\cdot,\cdot\rangle_{\hat{\mathfrak{k}}})$ be a solution to the $\mathrm{G}_2$-system \eqref{eq:HetG2sys0} with negative definite 
    $\langle\cdot,\cdot\rangle_{\hat{\mathfrak{k}}}$, $d^*\tau_1\geq0$  and $\mathrm{Scal}_{g_\varphi}\leq0$. Then, the $\mathrm{G}_2$-structure is torsion-free and the connection $\hat{\theta}$ is flat. 
\end{cor}
\begin{proof}
    Using Theorem \ref{thm:dSGFLSE} and the hypothesis, we have 
    \[
        0 = \mathrm{Scal}_{g_\varphi} - \frac12|H_\varphi|^2 + |F_{\hat{\theta}}|_{\hat{\mathfrak{k}}}^2-8d^*\tau_1-16|\tau_1|^2 - 4\lambda^2 \leq 0. 
    \]
    Since all summands are non-positive,  
    this implies that the intrinsic torsion forms $\tau_0 = \frac{12}{7}\lambda$, $\tau_1$ and $\tau_3$ are zero and that the connection $\hat{\theta}$ has vanishing curvature $F_{\hat{\theta}}$. 
\end{proof}

\begin{cor}\label{cor:solvunimodnd}
    Let $(\varphi, \hat{P},\hat{\theta},\langle\cdot,\cdot\rangle_{\hat{\mathfrak{k}}})$ be a solution to the $\mathrm{G}_2$-system \eqref{eq:HetG2sys0} on a $7$-dimensional 
    unimodular solvable Lie group $\mathrm{G}$.
    If the $\mathrm{G}_2$-structure $\varphi$ is left-invariant and the bilinear form $\langle\cdot,\cdot\rangle_{\hat{\mathfrak{k}}}$ is negative definite, then the 
    $\mathrm{G}_2$-structure is torsion-free, the induced metric $g_\varphi$ is flat, and the connection $\hat{\theta}$ is flat. 
\end{cor}
\begin{proof}
    Since $\varphi$ is left-invariant, it induces a left-invariant metric $g_\varphi$ and the corresponding intrinsic torsion forms $\tau_0$, $\tau_1$ and 
    $\tau_3$ are left-invariant, too. 
    Since $\mathrm{G}$ is unimodular, every left-invariant form of degree $6=\dim(\mathrm{G})-1$ on it is closed, thus $d^*\tau_1=0$. 
    Since $\mathrm{G}$ is solvable, then either $g_\varphi$ is flat or $\mathrm{Scal}_{g_\varphi}<0$ \cite{Mil}. 
    The thesis then follows from Corollary \ref{cor:bnegdef}. 
\end{proof}

In the following, we shall look for invariant solutions to the heterotic G$_2$-system \eqref{hetsysIntro0} 
on $7$-dimensional $2$-step nilmanifolds $M=\Gamma\backslash N$ endowed with a principal torus bundle $\mathbb{T}^k\to P\to M$ 
and a flat affine connection $\nabla$, so we will focus on the system \eqref{hetsysIntro}, which is equivalent to
\begin{equation}\label{eq:HetG2sys}
\begin{aligned}
\tau_2&=0,\\
F_{{\theta}}\wedge \psi&=0,\\
dH_{\varphi}&=\langle F_{{\theta}}\wedge F_{{\theta}}\rangle_{{\mathfrak{t}}^k} \,.
\end{aligned}
\end{equation}
Since every solution of \eqref{eq:HetG2sys} is also a solution of \eqref{eq:HetG2sys0}, 
the last corollary ensures that there are no solutions admitting a negative definite bilinear form 
$\langle\cdot,\cdot\rangle_{\mathfrak{t}^k}$. In other words, if we denote by $(k_+,k_{-})$ the signature of $\langle \cdot,\cdot\rangle_{\mathfrak{t}^k}$, then $k_+\ge 1$ whenever a solution of the $\G_2$-system \eqref{eq:HetG2sys} exists.

\subsection{2-step nilpotent metric Lie algebras} 
A real $n$-dimensional Lie algebra $\n$ with derived algebra $\n'\coloneqq [\n,\n]$ is {\em $2$-step nilpotent} if 
\[
\{0\}\neq\n'\subseteq \mathfrak{z},
\]
where $\mathfrak{z}$ is the center of $\n$. 

Given a metric $g$ on $\n$, one can consider the orthogonal splitting $\n = \r \oplus \n'$, where $\r \coloneqq (\n')^\perp$. 
Choosing a basis $\{z_1,\ldots,z_{n'}\}$ of $\n'$ and a basis $\{\ee_1,\ldots,\ee_{n-n'}\}$ of $\r$,  
it is possible to describe the structure of $\n$ in terms of the differentials 
of the dual basis $\{\ee^1,\ldots,\ee^{n-n'}, z^1,\ldots,z^{n'}\}$ of $\n^*$:
\[
\begin{cases}
d\ee^i = 0,&\quad 1\leq i\leq n-n',\\
dz^r = :\alpha_j,&\quad 1\leq j\leq n',
\end{cases}
\]
where $\alpha_j\in\Lambda^2\r^*$ for $1\leq j\leq n'$.  
Typically, the properties of the forms $\alpha_j$ allow one to distinguish between isomorphic and non-isomorphic Lie algebras 
(see, e.g., Remark \ref{rem:heisenberg} below).

\section{The $\G_2$-system \eqref{eq:HetG2sys}  on $2$-step nilmanifolds}\label{G2system}

We now consider the G$_2$-system \eqref{eq:HetG2sys}  on $2$-step nilmanifolds endowed with principal torus bundles. 
In this case, $M=\Gamma\backslash N$ is the quotient of a $7$-dimensional, simply connected, 2-step nilpotent Lie group $N$ 
by a cocompact discrete subgroup $\Gamma$, 
and the Lie group $K=\mathbb{T}^k$ is a $k$-dimensional torus, so $\mathfrak{k} = \mathfrak{t}^k\cong\R^k$ and $F_\theta=d\theta$.   

Given a $7$-dimensional nilmanifold $M=\Gamma\backslash N$, we look for solutions $(\varphi,P,\theta,\langle \cdot,\cdot\rangle_{\mathfrak{k}})$ of \eqref{eq:HetG2sys} such that both the 3-form $\f\in\Omega^3(M)$ 
and the curvature $F_\theta$, thought as a $\mathfrak{k}$-valued 2-form on $M,$ are induced by left-invariant forms on $N.$ 
 This allows us to study the system \eqref{eq:HetG2sys} for a $\G_2$-structure $\f\in\Lambda^3\n^*$ and a closed $\mathfrak{k}$-valued 2-form $F_\theta\in\Lambda^2\n^*\otimes\mathfrak{k}$ on the Lie algebra $\n$ of $N.$ 
We refer to such solutions as {\em invariant solutions} to the $\G_2$-system \eqref{eq:HetG2sys}. 

Note that every bilinear form on $\mathfrak{k}$ is  $\mathrm{Ad}(\mathbb{T}^k)$-invariant. 
It is convenient to fix an $\langle\cdot,\cdot\rangle_{\mathfrak{k}}$-orthogonal basis $\{t_1,\dots,t_k\}$ of $\mathfrak{k}$ and set 
\[
\epsilon_r \coloneqq \langle t_r,t_r\rangle_{\mathfrak k},\,\ 1\leq r \leq k. 
\]
We then have 
\[
\quad F_\theta=\sum_{r=1}^k F^r\,t_r, 
\]
where $F^r\in\Lambda^2\n^*$ are closed $2$-forms on $\n$, for $1\leq r\leq k$, and 
\[
\langle F_\theta \w F_\theta\rangle_{\mathfrak{k}} = \sum_{r=1}^k\epsilon_{r}\, F^{r}\wedge F^{r}. 
\]

\medskip 
If $\f$ is a left-invariant $\G_2$-structure on $N$, then its intrinsic torsion forms are left-invariant, too. 
In particular, $\tau_0$ is constant independently of $\f$ being a solution to the $\G_2$-system \eqref{eq:HetG2sys}. 
The study of invariant solutions to the $\G_2$-system on $M=\Gamma\backslash N$ 
can then be reduced to the study of the following set of equations on the Lie algebra $\n$ of $N$ for 
a ${\rm G}_2$-structure $\f\in\Lambda^3\n^*$,  2-forms $F^1,\ldots,F^r\in\Lambda^2\n^*$ 
and non-zero constants $\epsilon_1,\ldots,\epsilon_r\in \mathbb R\smallsetminus\{0\}$: 
\begin{equation}
\label{invsystem}
\begin{cases}
\tau_2=0\,,\\
dF^r=0,\qquad\ 1\leq r\leq k\,,\\
F^r\wedge \psi=0,\quad 1\leq r\leq k\,,\\
dH_{\varphi}=\sum_{r=1}^k\epsilon_{r}\, F^{r}\wedge F^{r}\,. 
\end{cases}
\end{equation}

\medskip 
On the other hand, a suitable solution $(\f,F^1,\ldots,F^k,\epsilon_1,\ldots,\epsilon_k)$ to system \eqref{invsystem} on a $7$-dimensional 
$2$-step nilpotent Lie algebra $\n$ gives rise to an invariant solution to 
the $\G_2$-system \eqref{eq:HetG2sys} on a certain $2$-step nilmanifold endowed with a principal torus bundle.  
In detail, consider the $g_\f$-orthogonal splitting $\n=(\n')^\perp\oplus \n'$ and a basis $\mathcal B=\{\ee_1,\ldots,\ee_{7-n'},z_1,\ldots,z_{n'}\}$ of $\n$ 
for which the structure constants are integers and such that $\{z_1,\ldots,z_{n'}\}$ is a basis of $\n'$ ($1\leq n'\leq 3$).  
Then 
\[
\Gamma\coloneqq\exp\left(\mathrm{span}_\Z(6\ee_1,\ldots,6\ee_{7-n'},z_1,\ldots,z_{n'})\right)
\]
is a cocompact lattice of the simply connected nilpotent Lie group $N = \exp(\n)$ and thus $M=\Gamma\backslash N$ is a 2-step nilmanifold.  
The choice of this particular lattice is explained in Remark \ref{rem:6}.

The $\G_2$-structure $\f\in\Lambda^3\n^*$ solves the first equation of \eqref{invsystem}, so it gives rise to a left-invariant $\G_2T$-structure 
on $N$ with constant intrinsic torsion form $\tau_0$.  
This in turn defines a $\G_2$-structure of the same type on the quotient $M=\Gamma\backslash N$. 
If the $2$-forms $F^1,\ldots,F^k$ solving \eqref{invsystem} are {\em integral} 
with respect to $\mathcal{B}$, namely $F^r(x,y)\in \mathbb Z$ for every $x,y\in \mathcal B$ and $1\leq r\leq k$ (cf.~Definition \ref{integralForms}), 
then there exists a principal $\mathbb{T}^k$-bundle $P\to M$ endowed with a connection $\theta$ whose curvature is $F_\theta = \sum_{r=1}^k F^rt_r$ 
(see Theorem \ref{Theoapp1} of Appendix \ref{sect:Tkbdl}), and we can define 
\[
\langle \cdot,\cdot\rangle_{\mathfrak{k}} \coloneqq \sum_{r=1}^k\epsilon_r t^{r}\otimes t^{r}\,,
\]
where $\{t_1,\ldots,t_k\}$ is a basis of the Lie algebra of $\mathbb{T}^k$ and $\{t^1,\dots,t^k\}$ is its dual basis. 
The data $(\f,P,\theta,\langle\cdot,\cdot\rangle_\mathfrak{k})$ is then an invariant solution to the $\G_2$-system \eqref{eq:HetG2sys} on $M=\Gamma\backslash N.$

\medskip 
We can then focus on the system \eqref{invsystem}. 
We study it by considering the cases $\dim(\mathfrak{n}')=1$ and $\dim(\mathfrak{n}')=2$ or $3$ separately. 

\section{The case $\dim(\n')=1$} \label{subsec1}
Assume that $\n'\coloneqq [\n,\n]$ is $1$-dimensional. 
Let $\varphi$ be a ${\rm G}_2$-structure on $\n$ and let $g_\varphi$ be the associated metric.  
Denote by $\v$ the $g_\varphi$-orthogonal complement to $\n'$. 
Then, once a unit vector $z\in\n'$ is fixed, $\varphi$ induces an ${\rm SU}(3)$-structure $(\omega,\Omega_+)\in\Lambda^2\v^* \times \Lambda^3\v^*$ on $\v$ via the relation 
\begin{equation}\label{varphi}
\varphi=\omega \wedge  z^\flat+\Omega_+\,,
\end{equation}
where $z^{\flat}$ is the dual covector of $z$ with respect to $g_\varphi$. 
Using the splitting $\n=\v\oplus\n' =  \v \oplus \langle z\rangle$, we can write any $2$-form $F$ on $\n$ as 
$$
F=F_{\v}+ \eta\wedge z^{\flat}, 
$$
for some $\eta\in \Lambda^1\v^*$, where $F_{\v}$ is the projection of $F$ onto $\Lambda^2\v^*$. By using this approach we can rewrite system \eqref{invsystem} as a system of forms in $\v$. This is the goal of the present subsection. 

\medskip
We adopt the following notation:
\begin{itemize}
\item $J$ is the complex structure on $\v$ induced by $(\omega,\Omega_+)$;

\vspace{0.1cm}
\item $*_\v$ is the Hodge star operator on $\v$ determined by the metric $\omega(\cdot,J\cdot)$ and the orientation $\omega^3$;

\vspace{0.1cm}
\item $\Omega_-\coloneqq *_\v \Omega_+$.
\end{itemize}

From \eqref{varphi}, we obtain 
\begin{equation}\label{eq:psi6}
\psi=*\varphi=\frac12 \omega\wedge\omega+\Omega_-\wedge z^\flat.    
\end{equation}

\smallskip

Since $(\omega,\Omega_+)$ is an {\rm SU}(3)-structure on  $\v$, $\Lambda ^2\v^*$ splits into orthogonal SU(3)-irreducible summands as follows: 
\begin{equation}\label{lambda2v}
\Lambda^{2}\v^*=\langle\omega\rangle \oplus \Lambda^{(2,0)+(0,2)}\v^*\oplus \Lambda^{(1,1)}_0\v^*\,,
\end{equation}
where 
$$
\langle\omega\rangle = \R\omega\,,\quad
\Lambda^{(2,0)+(0,2)}\v^*\coloneqq\{*_\v (\eta\wedge \Omega_+)\,\,:\,\,\eta\in \Lambda^1\v^*\}\,,\quad 
\Lambda^{(1,1)}_0\v^*\coloneqq\{\sigma\in \Lambda^2\v^*\,\,:\,\, \sigma \wedge \omega = - *_\v \sigma \}\,.
$$

\begin{rmk} The space $\Lambda^{(1,1)}_0\v^*$ can be described as the space of real $2$-forms that are of type $(1,1)$ with respect to $J$ (i.e. commuting with $J$ when viewed as skew-symmetric endomorphisms) and primitive with respect to $\omega$, i.e., their wedge product with $\omega\wedge\omega$ is zero. Similarly, the space $\Lambda^{(2,0)+(0,2)}\v^*$ is the space of $J$ anti-invariant real $2$-forms (identified with skew-symmetric endomorphisms anti-commuting with $J$). 
\end{rmk}

\begin{lemma}
    The following formulas hold:
\begin{eqnarray}
\label{|sigma|}
 *_\v(\sigma\wedge\sigma\w\omega) &=&-|\sigma|^2\,,\quad \mbox{for all $\sigma\in \Lambda^{(1,1)}_{0}\v^*$}\,,\\
\label{eta}
(*_\v(\eta\w\Omega_+))\wedge(*_\v(\eta\w\Omega_+))  &=& 2*_\v(\eta \w J\eta)\,, \\
 \label{|eta|}
|*_\v (\eta\w\Omega_+)|^2  &=&  2|\eta|^2 \,,
\end{eqnarray}
for every $\eta\in \Lambda^1\v^*$, where by definition $(J\eta)(\cdot)\coloneqq-\eta(J\cdot)$.
\end{lemma}
\begin{proof}
   For every {\rm SU}(3)-structure $(g,J,\omega,\Omega_+)$ on a $6$-dimensional vector space $\v$, there exists an oriented orthonormal basis $\{e_1,\ldots,e_6\}$ of $\v$ with dual basis $\{e^1,\ldots,e^6\}$ such that 
   \begin{equation}\label{su3-st}
       J(e_{2i-1})=e_{2i},\qquad\omega=e^{12}+e^{34}+e^{56},\qquad\Omega_+=e^{135} -e^{245} -e^{146}-e^{236}\,.\end{equation}  
   In order to check \eqref{|sigma|}, we notice that the left hand side is proportional to $|\sigma|^2$ by the Schur Lemma, 
   since $\Lambda^{(1,1)}_0\v^*$ is an irreducible ${\rm SU}(3)$ representation. 
   In order to find the proportionality factor, we pick an arbitrary $\sigma\in \Lambda^{(1,1)}_0\v^*$, e.g. $\sigma=e^{12}-e^{34}$.
   
   In order to check the other two relations, we use the fact that ${\rm SU}(3)$ acts transitively on the spheres of $\v$, so for every $\eta\in\v^*$ one can choose the above basis so that $\eta=ce^1$, for some $c\in\R$. Then both formulas can be checked directly from \eqref{su3-st}.
\end{proof}
\medskip

Notice that any $r$-form $\gamma\in\Lambda^r\n^*$ can be written as $\gamma=\gamma_1+\gamma_2\wedge z^\flat$,  for some $\gamma_1\in \Lambda^{r}\v^*$ and $\gamma_2\in \Lambda^{r-1}\v^*$. 
Since $z\in \n'$ and $\n$ is $2$-step nilpotent, we have that $\alpha\coloneqq dz^\flat$ lies in $\Lambda^2\v^*$, so it decomposes according to the splitting $\Lambda^2\v^*=  \langle\omega\rangle \oplus \langle\omega\rangle^\perp$ as follows
\begin{equation}\label{alpha}
\alpha \coloneqq dz^\flat=  b\omega + \alpha_0. 
\end{equation}
In particular, $\alpha_0\in \Lambda^{(2,0)+(0,2)}_0\v^*\oplus \Lambda^{(1,1)}_0\v^*$. 

\begin{rmk}\label{rem:heisenberg}
The  2-form $\alpha\in\Lambda^2\mathfrak{v}^*$ is nonzero, since $\n$ is not abelian, and its rank  determines the isomorphism class of $\n$: $\n\cong \mathfrak{h}_3\oplus \R^4$ if $\mathrm{rank}(\alpha)=2$, 
$\n\cong \mathfrak{h}_5\oplus \R^2$ if $\mathrm{rank}(\alpha)=4$ and $\n\cong \mathfrak{h}_7$ if $\mathrm{rank}(\alpha)=6$. 
\end{rmk}

\begin{lemma}\label{lemma1}
Consider the $\mathrm{G}_2$-structure $\varphi$ given by \eqref{varphi}.
Then, denoting as before $\lambda\coloneqq \frac{7}{12}\tau_0$, one has  $b=2\lambda$ in the expression \eqref{alpha} of $\alpha$. 
In addition, the condition $\tau_2=0$ implies $\tau_1=0$ and is equivalent to 
$$
 \alpha	\in \Lambda^{(1,1)}\v^*= \R\omega\oplus \Lambda_{0}^{(1,1)}\v^*\,. 
$$
\end{lemma}
\begin{proof}
Using the expressions \eqref{varphi} and \eqref{eq:psi6}, we obtain
$$
d\varphi=\omega \wedge dz^\flat = \omega\wedge\alpha\,,\qquad d \psi=-\Omega_-\wedge dz^\flat= -\Omega_-\wedge \alpha\,.
$$
Hence 
$$
7\tau_0 *1 = d\varphi\wedge \varphi =\omega\wedge\omega\wedge dz^\flat\wedge z^\flat = 6b*1. 
$$
From this, the first assertion follows.

Now, the first equation in \eqref{g2torsion} gives $12\tau_1=*(\varphi\wedge *d\varphi)$, which in the present situation implies 
$$
\tau_1=\frac{1}{12}*\left((\omega\wedge z^\flat+\Omega_+)\wedge *(\omega\wedge dz^\flat)\right).
$$
This shows that $\tau_1 \in \Lambda^1\v^*$. The condition $\tau_2=0$ is equivalent to  $d\psi=4\tau_1\wedge\psi$, i.e., to
\begin{equation}\label{d*f}
d\psi=4\tau_1\wedge*\varphi=4\tau_1\wedge\left( \frac12 \omega\wedge\omega+\Omega_-\wedge z^\flat\right).
\end{equation}
Since $d\psi = -\Omega_-\wedge \alpha \in \Lambda^{4}\v^*$, \eqref{d*f} gives 
$$
\tau_1\wedge\Omega_-=0\,,
$$
which in turn implies that 
$
\tau_1=0,
$
so that $\varphi$ is coclosed. Now, from 
$$
0= d \psi= -\Omega_-\wedge \alpha, 
$$
we deduce that $\alpha\in \Lambda^{(1,1)}\v^*$. 
\end{proof}

Lemma \ref{lemma1} proves Proposition \ref{coclosed} when $\dim(\n')=1$: every $\G_2T$-structure  $\f\in\Lambda^3\n^*$ is coclosed.  
Moreover, it shows that 
\[
\alpha = dz^\flat =  2\lambda\omega + \alpha_0,  
\]
with $\alpha_0\in \Lambda^{(1,1)}_0\v^*$.

\medskip   
We next focus on the torsion form $H_{\varphi}$ of $\varphi$ and write its differential in terms of $\alpha_0$, $\omega$ and $\lambda$. 
\begin{lemma}\label{l34}
Assume that the $\mathrm{G}_2$-structure $\varphi$ given by \eqref{varphi} has $\tau_2=0$. Then  
\[
H_\f = 2\lambda\varphi -(4\lambda \omega-\alpha_0)\wedge z^{\flat},
\]
and
$$
dH_{\varphi}=\alpha_0\wedge\alpha_0-4 \lambda^2\omega\wedge\omega\,.
$$
\end{lemma}
\begin{proof}
From Lemma \ref{lemma1} we have $\tau_1=0$, so from \eqref{g2torsion} we get:
$$
*\tau_3=d	\varphi-\frac{12}{7}\lambda *\varphi=\omega\wedge \alpha-\frac{12}{7}\lambda *\varphi,
$$
and thus
$$
\begin{aligned}
\tau_3=&\,*(\omega \wedge \alpha)-\frac{12}{7}\lambda \varphi=(-\alpha_0+4\lambda \omega)\wedge z^{\flat} -\frac{12}{7}\lambda \varphi.
\end{aligned}
$$
We then obtain 
$$
\begin{aligned}
H_\varphi&=\frac{2}{7}\lambda\varphi -\tau_3=\frac{2}{7}\lambda\varphi -(-\alpha_0+4\lambda \omega)\wedge z^{\flat} +\frac{12}{7}\lambda \varphi=2\lambda\varphi -(-\alpha_0+4\lambda \omega)\wedge z^{\flat}\,,
\end{aligned}
$$
and therefore
$$
\begin{aligned}
d H_\varphi	&=2\lambda\omega \wedge \alpha -(-\alpha_0+4\lambda \omega)\wedge\alpha 
		=-2\lambda\omega \wedge (\alpha_0+2\lambda\omega) + \alpha_0\wedge (\alpha_0+2\lambda\omega)\\
		&= -4\lambda^2\omega\wedge\omega+\alpha_0\wedge\alpha_0\,,
\end{aligned}
$$
as required. 
\end{proof}

In order to study the instanton equation, we prove the following. 
\begin{lemma}\label{instanton}
Consider the $\mathrm{G}_2$-structure $\varphi$ given by \eqref{varphi}. 
A $2$-form $F=F_{\v}+\eta\w z^{\flat}\in \Lambda^2\n^*=\Lambda^2\v^*\oplus \v^*\wedge z^{\flat}$ satisfies  
$$
\begin{cases}
dF=0 \,,\\
F\wedge \psi=0\,, 
\end{cases}
$$ 
if and only if  
$$
\eta\w\alpha=0\,,\quad {\mbox and }\quad F_\v={\frac12} *_\v (\eta\wedge \Omega_+)+\sigma\,,\quad \sigma \in 	\Lambda^{(1,1)}_0\v^*\,. 
$$
In particular, if $\mathrm{rank}(\alpha)\in\{4,6\}$, then $\eta=0$ and $F = F_\v \in \Lambda^{(1,1)}_0\v^*$. 
\end{lemma}
\begin{proof}
Since both $F_\v$ and $\eta$ are closed, the condition $dF=0$ is equivalent to
$$
\eta\wedge \alpha =0\,, 
$$
which implies $\eta=0$ if $\mathrm{rank}(\alpha)\in\{4,6\}$. 
Next we impose $F\wedge \psi=0$. Since 
$$
F\wedge \psi =\frac12 F_{\v}\wedge \omega\wedge\omega+\frac12 \eta\wedge z^{\flat}\wedge \omega \wedge\omega +F_\v\wedge \Omega_-\wedge z^{\flat}\,,
$$
we get 
$$
F\wedge \psi=0 \iff 
\begin{cases}
F_\v \wedge \omega\wedge\omega=0\,,\\
F_\v\wedge \Omega_-=-\frac12\eta\wedge\omega\wedge\omega\,. 
\end{cases}
$$
The first equation shows that one can write $F_\v=*_\v (x\wedge \Omega_+)+\sigma$ for some $\sigma \in\Lambda^{(1,1)}_0\v^*$ and $x\in \v^*$. Using the general formula $*_\v (x\wedge \Omega_+)\wedge\Omega_-=-x\wedge\omega\wedge\omega$, we get $x=\frac12\eta$.
\end{proof}

\begin{lemma}\label{Flemma}
Assume that the $\mathrm{G}_2$-structure $\varphi$ given by \eqref{varphi} has $\tau_2=0$, 
and let $H_\f$ be the corresponding torsion $3$-form. 
Let $\epsilon_1,\dots,\epsilon_k \in \R\smallsetminus\{0\}$. Then the $2$-forms $F^{r}=F^r_{\v}+\eta^r\w z^{\flat}\in \Lambda^2\n^*$, $r=1,\dots, k$, 
satisfy    
\begin{equation}
\label{Fsystem0}
\begin{cases}
dF^r=0\quad \mbox{ for every }r\,,\\
F^r\wedge \psi=0\quad \mbox{ for every }r\,,\\
dH_{\varphi}=\sum_{r=1}^k\epsilon_{r}\, F^{r}\wedge F^{r}\,,
\end{cases}
\end{equation}
if and only if 
\begin{equation}\label{s4}
\begin{cases}
\eta^r\w \alpha=0,\\
\sum_{r=1}^k \epsilon_r*_\v (\eta^r\wedge \Omega_+)\w\sigma^r=0=\sum\epsilon_r(*_\v (\eta^r\w\Omega_+))\wedge(*_\v (\eta^r\w\Omega_+)),\\
 \alpha_0\wedge\alpha_0-4\lambda^2  \omega\wedge\omega = 	\sum_{r=1}^k \epsilon_r \sigma^r\wedge\sigma^r, \\
\sum_{r=1}^k  \epsilon_r\left( *_\v (\eta^r\w\Omega_+)\w \eta^r+2\sigma^r \wedge \eta^r\right)=0\,,
\end{cases}
\end{equation}
where $\sigma^r$ is the component of $F^r$ in $\Lambda^{(1,1)}_0\v^*$ according to the decomposition \eqref{lambda2v}. 
\end{lemma}

\begin{proof}
Lemma \ref{instanton} implies that each $F^{r}$ can be written as  
$$
F^r=\frac12*_\v (\eta^r\wedge \Omega_+)+\sigma^r +\eta^r\wedge z^{\flat}\,, 
$$
with $\eta^r\w\alpha=0$.
It follows 
$$
F^r\w F^r		=	*_\v (\eta^r\w\Omega_+)\w \eta^r\w z^{\flat}+2\sigma^r \w\eta^r\w z^{\flat} 
				+*_\v (\eta^r\wedge \Omega_+)\w\sigma^r + \frac14(*_\v (\eta^r\w\Omega_+))\wedge(*_\v (\eta^r\w\Omega_+))+\sigma^r\wedge\sigma^r. 
$$
By Lemma \ref{l34}, the last equation of system \eqref{invsystem} is equivalent to the system 
$$
\begin{cases}
 \alpha_{0}\wedge\alpha_0-4\lambda^2   \omega\wedge\omega  = 	\sum_{r=1}^k \epsilon_r\left(\frac14(*_\v (\eta^r\w\Omega_+))\wedge(*_\v (\eta^r\w\Omega_+))+\sigma^r\wedge\sigma^r 
 									+*_\v (\eta^r\wedge \Omega_+)\w\sigma^r \right),\\
0=\sum_{r=1}^k  \epsilon_r\left( *_\v (\eta^r\w\Omega_+)\w \eta^r+2\sigma^r \wedge \eta^r\right)\,.
\end{cases}
$$
Since the forms 
$$
\alpha_{0}\wedge\alpha_0\,,\quad  \omega\wedge\omega\,,\quad (*_\v (\eta^r\w\Omega_+))\wedge(*_\v (\eta^r\w\Omega_+))\,,\quad \sigma^r\wedge\sigma^r 
$$
belong to $\Lambda^{(2,2)}\v^*$, while 
$$
*_\v (\eta^r\wedge \Omega_+)\w\sigma^r
$$
belongs to $\Lambda^{(3,1)+(1,3)}\v^*$, the above system splits into 
\begin{equation}\label{s3}
\begin{cases}
0=\sum_{r=1}^k \epsilon_r*_\v (\eta^r\wedge \Omega_+)\w\sigma^r,\\
 \alpha_{0}\wedge\alpha_0-4\lambda^2   \omega\wedge\omega = 	\sum_{r=1}^k \epsilon_r\left(\frac14(*_\v (\eta^r\w\Omega_+))\wedge(*_\v (\eta^r\w\Omega_+))+\sigma^r\wedge\sigma^r \right),\\
0=\sum_{r=1}^k  \epsilon_r\left(*_\v (\eta^r\w\Omega_+)\w \eta^r+2\sigma^r \wedge \eta^r\right)\,.
\end{cases}
\end{equation}
Taking the wedge product of the last equation with $\Omega_+$ and using $\sigma^r\w\Omega_+=0$ and  \eqref{|eta|}  we get   
\begin{equation}\label{normeta}
\sum_{r=1}^k\epsilon_r|\eta^r|^2=0\,.
\end{equation}

Next we show that $\sum_{r=1}^k \epsilon_r(*_\v (\eta^r\w\Omega_+))\wedge(*_\v (\eta^r\w\Omega_+))=0$. We can assume that $\mathrm{rank}(\alpha)=2$ since otherwise $\eta^r=0$ for all $r$. From the fact that $\alpha\wedge \eta^r=0$ and $\alpha\in\Lambda^{(1,1)}\v^*$, we also get $\alpha\wedge J\eta^r=0$ for every $r$, thus showing that for every $r$ there exists $c_r$ with $c_r\alpha=\eta^r\wedge J\eta^r$. Taking the scalar product with $\omega$ yields 
\begin{equation}\label{cr}|\eta^r|^2=6\lambda c_r\,,\end{equation} so from \eqref{normeta} we obtain $\lambda\sum\epsilon_r c_r=0$.  Note that we can assume $\lambda\ne 0$ since otherwise \eqref{cr} yields $\eta^r=0$ for all $r$. Therefore, we get
\begin{equation}\label{e4}
\sum_{r=1}^k\epsilon_r \eta^r\wedge J\eta^r=0.
\end{equation}
On the other hand, using \eqref{eta}, we have that  \eqref{e4} yields 
\begin{equation}\label{e5}
    \sum\epsilon_r(*_\v (\eta^r\w\Omega_+))\wedge(*_\v (\eta^r\w\Omega_+))=0, 
\end{equation}
showing that the middle equation in the system \eqref{s3} decouples into \eqref{e5} and 
\begin{equation}\label{equation}
\alpha_{0}\wedge\alpha_0-4\lambda^2   \omega\wedge\omega = 	\sum_{r=1}^k \epsilon_r\sigma^r\wedge\sigma^r. 
\end{equation}
Hence the claim follows.
\end{proof}

\begin{rmk}
We observe that Equation \eqref{equation} has no solution if all $\epsilon_r$'s are negative. Indeed, 
wedging both sides of equation \eqref{equation} by $\omega$, we obtain 
$$
-|\alpha_0|^2-24\lambda^2=-\sum_{r=1}^k \epsilon_r|\sigma^r|^2\,.
$$
If all $\epsilon_r$'s are negative, we then get 
$|\alpha_0|=\lambda=0$, whence $\alpha=0$ by \eqref{alpha} and Lemma \ref{lemma1}, contradicting the assumption that $\n$ is not abelian. 
This is consistent with Corollary \ref{cor:solvunimodnd}.
\end{rmk}

\smallskip 

If $\eta^k=0$, for all $1\leq k\leq r$, then the system \eqref{s4} reduces to Equation \eqref{equation}. 
This happens, for instance, when $\mathrm{rank}(\alpha)\in\{4,6\}$ by Lemma \ref{instanton}. 
In this case,  in order to find solutions to the ${\rm G}_2$-system \eqref{eq:HetG2sys}, we can focus on Equation \eqref{equation} on $\R^6$ equipped with 
the standard basis $\mathcal{B}=\{e_1,\ldots,e_6\}$ and the standard ${\rm SU}(3)$-structure 
\begin{equation}\label{stdSU(3)}
\omega=e^{12}+e^{34}+e^{56}\,,\quad 
\Omega_+=e^{135}-e^{146}-e^{236}-e^{245}\,,
\end{equation}
and use the next result. 

\begin{prop}\label{R6} 
Consider the vector space $\R^6$ endowed with the standard ${\rm SU}(3)$-structure \eqref{stdSU(3)}. 
Assume that $\alpha_{0},\sigma^1,\dots,\sigma^k\in \Lambda^{(1,1)}_0\,(\R^6)^*$, $\lambda\in \R$, $\epsilon_1\dots,\epsilon_k\in \R\smallsetminus\{0\}$, 
is a solution to \eqref{equation}, that $2\lambda\omega + \alpha_0\neq0$, 
and that the forms $2\lambda\omega + \alpha_0$ and $\sigma^1,\ldots,\sigma^k$ are integral with respect to the standard basis $\mathcal{B}$ of $\R^6$. 
Then, this solution gives rise to a $2$-step nilmanifold endowed with a $k$-torus bundle 
and admitting an invariant solution to the ${\G}_2$-system \eqref{eq:HetG2sys}.
\end{prop}
\begin{proof}

Consider the $7$-dimensional vector space $\n\coloneqq \R^6\oplus \langle z\rangle$ and endow it with the product metric $g$ 
for which $z$ is a unit vector.
Then, imposing   
\[
dz^\flat \coloneqq 2\lambda \omega + \alpha_{0}\neq0, \qquad de^i=0,~1\leq i\leq6\,, 
\]
gives $\n$ the structure of a $2$-step nilpotent Lie algebra with $1$-dimensional derived algebra $\n'= \langle z\rangle$. 
Since the $2$-form $2\lambda \omega + \alpha_{0}$ is integral with respect to $\mathcal{B}$, the simply connected $2$-step nilpotent Lie group $N=\exp(\n)$ has a 
cocompact lattice given by $\Gamma = \exp(\mathrm{span}_\Z(6e_1,\ldots,6e_6,z))$, and $M=\Gamma\backslash N$ is a 2-step nilmanifold. 
Moreover, since the $\sigma^r$'s are integral with respect to $\mathcal{B}$, Theorem \ref{Theoapp1} ensures the existence of a principal $\mathbb{T}^k$-bundle 
$P\to M=\Gamma\backslash N$ endowed with a connection $\theta$ whose curvature is $F_\theta= \sum_{r=1}^k \sigma^rt_r$, 
where $\{t_1,\ldots,t_k\}$ is a basis of the Lie algebra of $\mathbb{T}^k$. 
The $3$-form $\varphi\coloneq \omega \wedge  z^\flat+\Omega_+$ defines a $\G_2$-structure on $\n$ inducing the product metric $g$. 
By Lemma \ref{lemma1}, $\f$ is coclosed and has $\tau_0 = \frac{12}{7}\lambda$, 
so it gives rise to a $\G_2$-structure of the same type on $M=\Gamma\backslash N$. 
If we define $\langle \cdot,\cdot\rangle_{\mathfrak{k}}\coloneqq\sum_{r=1}^k\epsilon_r t^{r}\otimes t^{r}$, 
we then obtain an invariant solution $(\varphi, P,\theta,\langle\cdot,\cdot\rangle_{\mathfrak{k}})$ 
to the $\G_2$-system \eqref{eq:HetG2sys} on $M=\Gamma\backslash N$.

\end{proof}

Motivated by the last result, we now consider a $6$-dimensional vector space $\v$ equipped with an ${\rm SU}(3)$-structure $(\omega,\Omega_+)$, 
and we focus on Equation \eqref{equation} 
$$
\alpha_{0}\wedge\alpha_0-4\lambda^2   \omega\wedge\omega = 	\sum_{r=1}^k \epsilon_r\sigma^r\wedge\sigma^r 
$$
for the unknowns $\alpha_0,\sigma^1,\ldots,\sigma^k\in\Lambda^{(1,1)}_0\v^*$, $\lambda\in \R$, and $\epsilon_1,\ldots,\epsilon_k\in\R\smallsetminus\{0\}$. 
Note that $\lambda$ and $\alpha_0$ cannot be both zero, as otherwise the Lie algebra $\n$ would be abelian.

In order to study Equation \eqref{equation}, which involves 4-forms on $\v$, we rewrite it in terms of endomorphisms of $\v$. 
To do this, we use the dual Lefschetz operator, which can we written using an orthonormal basis $\{e_1\dots,e_6\}$ of $\v$ as
$$
\Lambda=\frac12 \sum_{i=1}^6 Je_{i} \lrcorner e_{i} \lrcorner \,.
$$
This operator is an isomorphism from $\Lambda^4\v^*$ to $\Lambda^2\v^*$. 

One can easily check that $\Lambda(\omega\wedge\omega)=4\omega$ and $\Lambda(\sigma\wedge\sigma)(\cdot,\cdot)=2\langle A^2J\cdot,\cdot\rangle$, 
for every $\sigma\in \Lambda^{(1,1)}_0\v^*$, 
where $A$ is the skew-symmetric endomorphism commuting with $J$ defined by $\sigma(\cdot,\cdot)=\langle A\cdot,\cdot\rangle$.  
Applying $\Lambda$, Equation \eqref{equation} can be rewritten in the following, equivalent, form 
\begin{equation}\label{A_0}
L^2_0-\sum_{r=1}^k \epsilon_r L_r^2 =-8 \lambda^2 {\rm Id}\,,
\end{equation}
where the $L_r$'s are the trace-free symmetric endomorphisms commuting with $J$ defined by 
$$
\alpha_0(\cdot,\cdot)=\langle L_0J\cdot,\cdot \rangle\,,\quad \sigma^r(\cdot,\cdot)=\langle L_rJ\cdot,\cdot \rangle\,.
$$

\medskip 

We now discuss some properties of the solutions of \eqref{A_0} that will be useful in the proof of Theorem \ref{main1}.
Assume that one of the $\epsilon_r$'s, say $\epsilon_1$, is positive, and all the others are negative. In this case, by setting 
\begin{equation}\label{eq:ABr}
A= \sqrt{\frac{\epsilon_1}{8}}\, L_1\,,\quad B_1= \sqrt{\frac{1}{8}}\, L_0\,,\quad B_r= \sqrt{\frac{-\epsilon_r}{8}}\, L_r\,,\quad 2\leq r \leq k\,, 
\end{equation}
we can rewrite \eqref{A_0} as  
\begin{equation}\label{A2}
A^2=\lambda^2\,{\rm Id}+\sum_{r=1}^k B_r^2\,.
\end{equation}
\begin{prop}\label{implieslambda=0}
Let $A,B_1,\ldots,B_k$ be trace-free symmetric endomorphisms of a $6$-dimensional Hermitian vector space $(\v,J)$ commuting with $J$  
and satisfying Equation \eqref{A2}. Then $\lambda=0$ and the endomorphisms $A^2,B_1^2,\ldots,B_k^2$ are positively collinear. 
\end{prop}

\begin{proof}
Composing $A^2$ with itself and using \eqref{A2}, we get
\begin{equation}\label{A4}
A^4=\lambda^4 {\rm Id}+2\lambda^2 \sum_{r=1}^k B_r^2+\sum_{r=1}^k B_r^4+\sum_{r\neq s} B_r^2 B_s^2\,.
\end{equation}

Next, we observe that if $X$ is a trace-free symmetric endomorphism, then  
\begin{equation}\label{X}
{\rm tr}\, (X^4)=\frac{({\rm tr}(X^2))^2}{4}.
\end{equation}
Indeed, if $x_1,x_2,-x_1-x_2\in \R$ are the eigenvalues of $X$ (all of algebraic multiplicity $2$), we have 
$$
{\rm tr}\, (X^2)=2(x_1^2+x_2^2+(x_1+x_2)^2)=4(x_1^2+x_2^2+x_1x_2)
$$
and 
$$
\begin{aligned}
{\rm tr}\, (X^4)=&\,2(x_1^4+x_2^4+(x_1+x_2)^4)=4(x_1^4+x_2^4+2x_1^3x_2+2x_1x_2^3+3x_1^2x_2^2)=4(x_1^2+x_2^2+x_{1}x_2)^2\,,
\end{aligned}
$$
from which \eqref{X} follows. Taking the trace of both sides of \eqref{A4} and applying \eqref{X}, we obtain:
 $$
 \frac{({\rm tr}(A^2))^2}{4}={\rm tr}(A^4)=6\lambda^4+2\lambda^2 \sum_{r=1}^k {\rm tr}\,(B_r^2) +\sum_{r=1}^k {\rm tr}\,(B_r^4)+\sum_{r\neq s} {\rm tr}(B_r^2B_{s}^2)\,.
 $$
Moreover, from \eqref{A2} and \eqref{X}, we have 
 $$
  \begin{aligned}
\frac{ ({\rm tr}(A^2))^2}{4}=&\frac14 \left({\rm tr} \left(\lambda^2{\rm Id}+\sum_{r=1}^k  B_r^2\right)\right)^2
 =9\lambda^4+3\lambda^2\sum_{r=1}^k {\rm tr}(B_r^2)+\frac14 \sum_{r=1}^k ({\rm tr} (B_r^2))^2+\frac14\sum_{r\neq s} {\rm tr}(B_r^2){\rm tr}(B_{s}^2)\\
 &= 9\lambda^4+3\lambda^2\sum_{r=1}^k {\rm tr}(B_r^2)+\sum_{r=1}^k {\rm tr} (B_r^4)+\frac{1}{4}\sum_{r\neq s} {\rm tr}(B_r^2){\rm tr}(B_{s}^2)\,.
 \end{aligned}
 $$
 Comparing these two equations yields
\begin{equation}\label{eql}
 \begin{aligned}
3\lambda^4+\lambda^2\sum_{r=1}^k {\rm tr}\,(B_r^2)+\frac{1}{4}\sum_{r\neq s} {\rm tr}(B_r^2){\rm tr}(B_{s}^2)-\sum_{r\neq s} {\rm tr}(B_r^2B_{s}^2)=0\,.
 \end{aligned}
\end{equation}
  Now, we observe that \eqref{X} and the Cauchy-Schwarz inequality imply:  
 $$
 {\rm tr}(B_s^2B_r^2)\leq \sqrt{{\rm tr}(B_r^4)}\,\sqrt{{\rm tr}(B_s^4)}=\frac14 
 {\rm tr}(B_r^2){\rm tr}(B_s^2)\,. 
 $$
so \eqref{eql} yields
$$
3\lambda^4+\lambda^2\sum_{r=1}^k {\rm tr}(B_r^2)\leq 0
 $$
 which forces $\lambda=0$. 

Moreover, the equality case in the Cauchy-Schwarz inequality implies that the symmetric endomorphisms $B_r^2$ are positively collinear for $1\leq r\leq k$.  
Finally, since $A^2=\sum_{r=1}^kB_r^2$, the same conclusion holds for $A^2$ as well. 
\end{proof}

We apply the previous proposition in our case.
Since $\lambda=0$, the 2-form $\alpha_0$ is equal to the 2-form $\alpha=dz^\flat$, which is non-zero. 
Thus the endomorphism $B_1$ defined in \eqref{eq:ABr} must be non-zero. Consequently, there exist $a_2,\dots, a_k\in \R$ such that 
$$
B_r^2=a_r^2B_1^2\,,
$$
for all $2\leq r \leq k$, and therefore
\[
A^2 = \left(1+\sum_{r=2}^ka_r^2\right) B_1^2. 
\]

\begin{lemma}\label{lb}
Let $A$ and $B$ be  trace-free symmetric endomorphisms of $\v$ commuting with the complex structure $J$ and satisfying 
$$
A^2=a^2B^2,
$$
for a some $a\in \R$. If $B$ is invertible, then $A=\pm aB$.
 \end{lemma}
 \begin{proof}
 If $a=0$ the statement is obvious. We can then focus on the case $a\neq0$, where it is not restrictive assuming $a=1$, so that $A^2=B^2$. 
 One can find a basis of $\v$ such that the matrices representing $B$ and $J$ with respect to it are the following:  
 $B=\begin{pmatrix}
     D&0_3\\0_3&D
 \end{pmatrix}$, 
 where 
 $D=\begin{pmatrix}d_1&0&0\\0&d_2&0\\0&0&d_3
 \end{pmatrix}$ 
 with $d_1+d_2+d_3=0$ and $d_1d_2d_3\neq0$,
 and 
 $J=\begin{pmatrix}
     0_3&-I_3\\I_3&0_3
 \end{pmatrix}$.
 Since $A$ is symmetric, trace-free, and commutes with $J$, it is represented by a block-diagonal matrix of the form 
 $A=\begin{pmatrix}
     X&-Y\\Y&X
 \end{pmatrix}$, 
 where $X$ is a symmetric and trace-free $3\times 3$ matrix, and $Y$ is a skew-symmetric $3\times 3$ matrix.
The condition $A^2=B^2$ reads then 
\begin{equation}\label{sy}
    \begin{cases}
   X^2-Y^2=D^2,\\
   XY+YX=0_3.
\end{cases}
\end{equation}
 If $X=0$, the first equation of \eqref{sy} shows that $Y$ is invertible, which is impossible for a skew-symmetric $3\times 3$ matrix. Since $X$ is trace-free, its kernel can be at most 1-dimensional. 
 Assume that $\ker(X)$ is a line. The second equation of the system \eqref{sy} shows that $Y$ preserves the kernel of $X$, and since $Y$ is skew-symmetric, it must vanish on $\ker(X)$. 
 This contradicts the first equation, since $D$ would then also vanish on $\ker(X)$. Therefore, $X$ is invertible. 
 
 Since ${\rm tr}(X)=0$, if $\lambda$ is an eigenvalue of $X$, then $-\lambda$ is not an eigenvalue of $X$. 
The second equation of \eqref{sy} thus shows that $Y$ vanishes on the eigenspaces of $X$, so $Y=0_3$ and $X^2=D^2$. The spectrum of $X$ is therefore $\{\varepsilon_1 d_1,\varepsilon_2 d_2,\varepsilon_3 d_3\}$, where $\epsilon_i\in\{\pm1\}$. Since $X$ and $D$ are trace-free, all signs $\varepsilon_i$ are equal. If $D^2$ has three different eigenvalues, clearly $X$ preserves the three eigenspaces, so it has diagonal form in the given basis, whence $X=\pm D$. If $D^2$ has a double eigenvalue, then one can assume $d_1=d_2$ and $d_3=-2d_1$, so $X$ has block-diagonal form $X=\begin{pmatrix}
     U&0\\0&\varepsilon d_3,
 \end{pmatrix}$, where $\varepsilon\in\{\pm1\}$ and $U$ is a $2\times 2$ symmetric matrix satisfying ${\rm tr}(U)=-\varepsilon d_3=2\varepsilon d_1$ and $U^2=d_1^2I_2$. Then clearly $U=\varepsilon d_1 I_2$, so $X=\varepsilon D$, and thus $A=\varepsilon B$.
\end{proof}

We are now ready to give the proof of Theorem \ref{main1}.

\subsection{Proof of Theorem \ref{main1}-(i)} 

If $M=\Gamma\backslash N$ admits an invariant solution $(\f,P,\theta,\langle\cdot,\cdot\rangle_{\mathfrak{k}})$ to the $\G_2$-system \eqref{eq:HetG2sys} 
and $\dim(\n')=1$, then the above discussion shows that there exists a $g_\f$-orthogonal decomposition $\n=\v\oplus\n'$, with $\n'=\langle z\rangle$ and 
$dz^\flat = \alpha = 2\lambda\omega + \alpha_0$,
where $\alpha_0\in\Lambda^{(1,1)}_0\v^*$ is a primitive form of type $(1,1)$. If $\lambda=0$, then $\alpha=\alpha_0$, so  $\alpha\wedge\alpha\neq 0$ by \eqref{|sigma|}. Consequently, $\n$ cannot be isomorphic to $\mathfrak h_3\oplus \R^4$, and thus $\n\cong{\mathfrak h_{5}\oplus \R^{2}}$ if $\mathrm{rank}(\alpha)=4$ 
or $\n\cong{\mathfrak h_{7}}$ if $\mathrm{rank}(\alpha)=6$. 
Note that, in both cases, the $\G_2$-system \eqref{eq:HetG2sys} reduces to Equation \eqref{equation} by Lemma \ref{instanton}, as 
all of the $\eta^r$'s appearing in \eqref{s4} must be zero.

\medskip 
\subsection{Proof of Theorem \ref{main1}-(ii)}

By Proposition \ref{R6}, if we consider $\R^6$ endowed with the standard SU(3)-structure, then a solution to Equation \eqref{equation} with 
$\lambda=0$, $k=k_+=1$, $\alpha_0,\sigma^1\in\Lambda^{(1,1)}_0(\R^6)^*$ integral $2$-forms with respect to the standard basis of $\R^6$, and $\mathrm{rank}(\alpha_0)=4$ gives rise to a solution to the $\G_2$-system \eqref{eq:HetG2sys} on a 2-step nilmanifold endowed with a principal $S^1$-bundle 
and with corresponding Lie algebra $\n\cong\mathfrak{h}_5\oplus\mathbb{R}^2$.  
The $S^1$-bundle is non-trivial if the curvature form $F=\sigma^1$ is not exact, namely, by Nomizu's theorem \cite{Nomizu}, if $\alpha_0$ and $\sigma^1$ are not proportional. 

Equation \eqref{equation} is equivalent to \eqref{A_0}, which now reads
\begin{equation}\label{eq:L0L1Thm1ii}
   L_0^2 = \epsilon_1 L_1^2, 
\end{equation}
where $L_0$ and $L_1$ are trace-free symmetric endomorphisms of $\R^6$ commuting with $J$ such that $\alpha_0 = \langle L_0J\cdot,\cdot\rangle$ and $\sigma^1 = \langle L_1J\cdot,\cdot\rangle$. 
We can take $\epsilon_1=1$ and search for solutions for which $L_0$ has rank $4$ and $L_1$ is not proportional to $L_0$.

A solution is given, for instance, by the endomorphisms represented by the following integer matrices with respect to the standard basis of $\R^6:$
\[
L_0 = \mathrm{diag}\left(5,5,-5,-5,0,0\right),
\qquad
L_1 =
\begin{pmatrix}
0&0&3&4&0&0\\
0&0&-4&3&0&0\\
3&-4&0&0&0&0\\
4&3&0&0&0&0\\
0&0&0&0&0&0\\
0&0&0&0&0&0
\end{pmatrix}. 
\]
The corresponding 2-forms are 
\[
\alpha_0 = 5\,e^{12}-5\,e^{34},\qquad  F=\sigma^1= -4\,e^{13} + 3\,e^{14} - 3\,e^{23} - 4\,e^{24}.
\]

We can now conclude as in the proof of Proposition \ref{R6}. For the reader's convenience, we sketch the argument here.  
The $7$-dimensional $2$-step nilpotent Lie algebra is obtained considering the vector space 
$\n \coloneqq \R^6\oplus\langle z\rangle$  
endowed with the product metric $g$ for which $z$ has unit length, and setting 
\[
de^i=0,~1\leq i \leq 6, \qquad dz^\flat = \alpha_0 = 5\,e^{12}-5\,e^{34}.   
\]
The coclosed $\G_2$-structure on $\n$ is defined by the $3$-form $\f = \omega\w z^\flat+\Omega_+$, which induces the metric $g$. 
Since both $\alpha_0$ and the closed $2$-form $F$ are integral with respect to the basis $\{e_1,\ldots,e_6,z\}$ of $\n$,
we can apply Theorem \ref{Theoapp1} and obtain the desired solution on the $2$-step nilmanifold $M=\Gamma\backslash N$, where $N=\exp(\n)$ and 
$\Gamma=\exp(\mathrm{span}_\Z(6e_1,6e_2,6e_3,6e_4,6e_5,6e_6,z))$.

\medskip 
\subsection{Proof of Theorem \ref{main1}-(iii)}
We now study the existence of solutions to equation \eqref{equation} satisfying the requirements of Proposition \ref{R6} 
and such that $\mathrm{rank}(\alpha_0)=6$, so that $\n\cong\mathfrak{h}_7$. 
By Proposition \ref{implieslambda=0}, if we take either $k=k_+=1$ or $k>1$ and $k_+=1$, then every solution must have $\lambda=0$. 
Without loss of generality, we may assume  $\epsilon_1>0$ and, if $k>1$, $\epsilon_2,\dots,\epsilon_{k}<0$. 
We can then rewrite equation \eqref{equation} as \eqref{A2} with $\lambda=0$, i.e., 
$$
A^2=\sum_{r=1}^k B_r^2\,,
$$
where $A$ and $B_r$'s are the trace-free symmetric endomorphisms of $\R^6$ commuting with $J$ defined in \eqref{eq:ABr}. 
If $\mathrm{rank}(\alpha_0)=6$, the endomorphism $B_1$ must be invertible. 
In particular, $A$ must be non-zero, so the component $F^1=\sigma^1$ of the curvature is non-zero. 
Combining Proposition \ref{implieslambda=0} and Lemma \ref{lb} (or applying just the latter if $k=k_+=1$), 
we obtain that $A$, and $B_r$, for $2\leq r \leq k$, in \eqref{A2}, 
are proportional to $B_1$, 
showing that all curvature forms $F^r$ are proportional to the exact 2-form $\alpha=dz^\flat$. 
It follows that every solution to the $\mathrm{G}_2$-system \eqref{eq:HetG2sys}  
has trivial torus bundle and non-zero curvature if $k_+=1$.

\smallskip

Assume now that $\lambda=0$ and $k_-=0$. Then, Equation \eqref{A_0} becomes
\[
L_0^2 = \sum_{r=1}^k\epsilon_r L_r^2, 
\]
with $\epsilon_r>0$, for $1\leq r \leq k$. 
Since $L_0$ is invertible, at least one of the endomorphisms $L_r$ is non-zero, so at least one component of the curvature is non-zero. 
Moreover, Proposition \ref{implieslambda=0} implies that the endomorphisms $L_0^2,\epsilon_1 L_1^2,\ldots,\epsilon_k L_k^2$ are positively collinear.   
We then have 
\[
L_r^2 = \frac1{\epsilon_r}\,a_r^2 L_0^2,
\]
for certain $a_r\in\R$, $1\leq r\leq k$, with at least one of them different from zero. 
By Lemma \ref{lb}, the endomorphisms $L_r$ are either zero or non-zero multiples of $L_0$. 
This shows that all non-zero curvature forms $F^r$ are proportional to the exact 2-form $\alpha=dz^\flat$, whence it follows that  
every solution to the $\mathrm{G}_2$-system \eqref{eq:HetG2sys} has trivial torus bundle and non-zero curvature.

\smallskip

To conclude the proof, it is sufficient to construct an example with $k_+=2$, $k_-=1$ and non-trivial torus bundle. 
This is equivalent to finding a suitable solution to Equation \eqref{A_0} with $k=3$, with two positive and one negative $\epsilon_r$, $1\leq r\leq3$, 
and such that at least one of the curvature forms $F^r = \sigma^r=\langle L_rJ\cdot,\cdot \rangle$ is not exact. 
Without loss of generality, we assume $\epsilon_1,\epsilon_2>0$, $\epsilon_3<0$, and look for solutions of the equation
\[
L_0^2-\epsilon_3 L_3^2 = \epsilon_1 L_1^2 + \epsilon_2 L_2^2, 
\]
such that $\mathrm{rank}(L_0)=6$ and at least one of the endomorphisms $L_1,L_2,L_3$ is not proportional to $L_0$. 
We can take, for instance, $\epsilon_1 = \epsilon_2 = - \epsilon_3 = 1$ and the endomorphisms represented by the following matrices with respect to 
the standard basis of $\R^6$
\[
L_0 = L_2 =  \mathrm{diag}\left(1,1,-2,-2,1,1\right),\qquad L_1=L_3 = \mathrm{diag}\left(1,1,-1,-1,0,0\right).
\]
They correspond to the following $2$-forms
\[
\alpha_0 = \sigma^2 = e^{12} -2\,e^{34} + e^{56}, \qquad 
\sigma^1 = \sigma^3 = e^{12} - e^{34}.
\]
In particular, the curvature forms $F^r=\sigma^r$, $1\leq r\leq 3,$ are integral with respect to the basis $\{e_1,\ldots,e_6,z\}$ and $F^1,F^3$ are not exact. 
We can now conclude as in the proof of Proposition \ref{R6}.

\medskip 

\subsection{Proof of Theorem \ref{main1}-(iv)}
Proposition \ref{implieslambda=0} implies that there are no invariant solutions to the ${\rm G}_2$-system \eqref{eq:HetG2sys} 
with $k_+=1$ and $\lambda\neq 0$ on a $2$-step nilmanifold with $\dim(\n')=1$. 

To conclude the proof of Theorem \ref{main1}-(iv) it is then sufficient to show that when $k=k_+=2$, 
the ${\rm G}_2$-system \eqref{eq:HetG2sys} can be solved for each of the three isomorphism classes of $2$-step nilpotent Lie algebras $\n$ with $\dim(\n')=1$. 
In view of Proposition \ref{R6}, the problem reduces to finding a solution $(\lambda,\alpha_0,\sigma^1,\sigma^2,\epsilon_1,\epsilon_2)$ to equation \eqref{equation} on $\R^6$ equipped the standard ${\rm SU}(3)$-structure \eqref{stdSU(3)}, where $\lambda\in\mathbb{R}\smallsetminus\{0\}$ 
and the forms $ \alpha = 2\lambda\omega + \alpha_0,\sigma^1,\sigma^2$ are integral. 
The rank of the form $\alpha = 2\lambda \omega + \alpha_0$ then  determines the isomorphism class of the corresponding Lie algebra.  

\medskip 
A straightforward computation shows that we have the following solutions: 

\begin{itemize}
\item 
$\lambda =\frac12\,,\alpha_0=-e^{12}-e^{34}+2\,e^{56}, \sigma^1 = e^{12}-e^{56},\,  \sigma^2 = e^{34}-e^{56},\,
        \epsilon_1=\epsilon_2=3$. \vspace{0.1cm} \\
        In this case $\alpha=3e^{56}$, so the corresponding Lie algebra is isomorphic to $\mathfrak{h}_3\oplus\R^4$.

\vspace{0.2cm}
\item 
$\lambda =\frac12\,,\alpha_0=-e^{12}+e^{56},\, \sigma^1 = -e^{12} + 2\,e^{34}-e^{56},\,  \sigma^2 = e^{12}-e^{56},\,\epsilon_1=\frac12\,,\epsilon_2=\frac52$. \vspace{0.1cm} \\
In this case $\alpha=e^{34}+2\,e^{56}$, so the corresponding Lie algebra is isomorphic to $\mathfrak{h}_5\oplus\R^2$.

\vspace{0.2cm}
\item 
$\lambda=\tfrac12\,,\alpha_0=0\,, \sigma^1=-2\,e^{12} + e^{34} + e^{56}\,, \sigma^2 = e^{34}-e^{56}, 
\,\epsilon_1=\frac12\,,\epsilon_2=\frac32$. \vspace{0.1cm} \\ 
In this case $\alpha= e^{12}+e^{34}+e^{56}$, so the corresponding Lie algebra is isomorphic to 
$\mathfrak{h}_7$.
\end{itemize} 
Note that in all examples both curvature forms $F^1=\sigma^1$ and $F^2=\sigma^2$ are integral 
with respect to the basis $\{e_1,\ldots,e_6,z\}$ 
and not proportional to $\alpha$, so they are not exact by Nomizu's theorem \cite{Nomizu}. 
Consequently, the solutions to the $\G_2$-system have non-trivial torus bundle.

\section{The case $\dim(\n')=2$ or $3$.}\label{2or3}
We now consider the case $\dim(\n')\geq 2$. 
We use an approach that is similar to one of the previous section, but in this case we will decompose the Lie algebra $\n$ as the orthogonal direct sum 
of a $3$-dimensional subspace calibrated by $\varphi$ and its orthogonal complement $\r$.  
The latter naturally inherits an ${\rm SU}(2)$-structure induced by $\varphi$, which can be used to rewrite \eqref{invsystem} as a system of  exterior forms on $\r$. 

\medskip 
Let $\varphi$ be a ${\rm G}_2$-structure on $\n$. If $\dim(\n')=3$, we assume that $\varphi$ calibrates $\n'$, 
that is, there exists an orthonormal basis $\{z_1,z_2,z_3\}$ of $\n'$ such that $|\varphi(z_1,z_2,z_3)|=1$.   
Recall that $\varphi(z_1,z_2,z_3)=1$ if and only if $z_1 \times_\varphi z_2 = z_3$, where $\times_\varphi$ denotes the vector cross product induced by $\varphi$.

\smallskip 

Let us consider two orthonormal vectors  $z_1,z_2\in\n'$ and let  
$$
z_3\coloneqq\varphi(z_1,z_2,\cdot)^{\sharp} = z_1 \times_\varphi z_2.
$$
The vector $z_3$ has unit length and is orthogonal to $z_1$ and $z_2$. 
The subspace $\langle z_1,z_2,z_3\rangle$ is calibrated by $\varphi$ and it is contained in the center of $\n$, but does not necessarily coincide with it. 
When $\dim(\n')=2$, we have $\n'=\langle z_1,z_2\rangle$ and we will show that $d z_3^\flat=0$.
If $\dim(\n')=3$ and $\varphi$ calibrates $\n'$, we have $\n'=\langle z_1,z_2,z_3\rangle$ and $d z_3^\flat\neq 0$. 

Let $\mathfrak{r}\coloneqq\langle z_1,z_2,z_3\rangle^{\perp}$, so that $\n= \mathfrak{r} \oplus \langle z_1,z_2,z_3\rangle$. 
Since $\G_2$ acts transitively on orthonormal pairs of vectors, there exists an adapted basis $\{e_1,\ldots,e_7\}$ of $\n$ 
(i.e. a basis where $\f$ is expressed by \eqref{standard}), such that $e_5=z_1$ and $e_6=z_2$. 
Then by \eqref{standard} we get $z_3=\varphi(z_1,z_2,\cdot)^{\sharp}=e_7$, and denoting by $\{e^1,e^2,e^3,e^4,z^1,z^2,z^3\}$ the dual basis, we have
\begin{equation}\label{eq:G2n23}
\varphi=\sum_{i=1}^3 \omega_i\w z^i+z^1\w z^2 \w z^3,
\end{equation}
where $\omega_1=e^{13}-e^{24}$, $\omega_2=-e^{14}-e^{23}$ and $\omega_3=e^{12}+e^{34}$ define an {\rm SU}(2)-structure on $\r$. 
The metric corresponding to the SU(2)-structure is the restriction of $g_\f$ to $\r$. Moreover, since $\langle z_1,z_2,z_3\rangle$ 
is calibrated by $\f$, its orthogonal complement $\r$ is calibrated by $\psi=*\f$, so it is oriented by the volume form $\psi|_{\r}$. 
We denote by $*_\r$ the Hodge operator corresponding to this metric and orientation on $\r$.  
The closed 2-forms 
$\omega_1,\omega_2,\omega_3$ are self-dual and pairwise orthogonal: 
\[
*_\r\omega_i = \omega_i,\qquad \omega_i\w\omega_j = 2\delta_{ij}*_\r1\,,\qquad \forall i,j\in\{1,2,3\}\,. 
\]
From this observation and the expression of $\f$, we deduce that
\begin{equation}\label{psi1}
\psi=*\varphi=\omega_1\w z^2\w z^3+\omega_2\w z^{3}\w z^{1}+\omega_3\w z^{1}\w z^2+*_\r1\,. 
\end{equation}

We introduce the following notation 
\[
\alpha_i\coloneqq dz^i\in\Lambda^2\r^*,\quad 1\leq i \leq 3, 
\]
and
\[
a_{ij}\coloneqq \langle \omega_i,\alpha_j\rangle = *_\r(\omega_i\w\alpha_j)\,,\quad 1\leq  i,j\leq 3\,. 
\]
If $\dim(\n')=2$, then $\alpha_1$ and $\alpha_2$ are linearly independent and $\alpha_3=0$, since $z_3$ is orthogonal to $\n'$.  
Consequently, $a_{i3}=0$ for all $1\leq i \leq 3$. 
If $\dim(\n')=3$, then $\alpha_1$, $\alpha_2$ and $\alpha_3$ are linearly independent.

\begin{lemma}\label{coclosed2}
Consider the $\mathrm{G}_2$-structure $\varphi$ given by \eqref{eq:G2n23}. Then  
\begin{equation}\label{lambda3}
\lambda=\frac{1}{6}\left(a_{11}+a_{22}+a_{33}\right).
\end{equation}
Moreover, the condition $\tau_2=0$  is equivalent to 
\begin{equation}\label{a=a}
a_{ij}=a_{ji}\,,\qquad\forall i,j\in\{1,2,3\}\,.
\end{equation}
and implies $\tau_1=0$.
\end{lemma}
\begin{proof}
We have 
\begin{equation}\label{dphi}
d\varphi=\sum_{i=1}^3 \omega_i\w\alpha_i+\alpha_1\w z^2\w z^3+\alpha_2 
\w z^3 \w z^1+ \alpha_3 
\w z^1 \w z^2\,,
\end{equation}
which together with \eqref{eq:G2n23} and the fact that $\omega_i\wedge\alpha_i=a_{ii}*_\r 1$ gives
\begin{equation}\label{fdf}
\varphi\w d\varphi = 2 (a_{11}+a_{22}+a_{33})\,*1\,.
\end{equation}
On the other hand, from the first equation in \eqref{g2torsion} we obtain
$\varphi\w d\varphi =7\tau_0*1=12 \lambda *1$, which together with \eqref{fdf} implies \eqref{lambda3}. 
Moreover, 
$$
\begin{aligned}
d\psi=&\, \omega_1\w \alpha_2\w z^3 -  \omega_1\w \alpha_3\w z^2+ \omega_2\w \alpha_3\w z^1- \omega_2\w \alpha_1\w z^3+ \omega_3\w \alpha_1\w z^2-  \omega_3\w \alpha_2\w z^1\\
=&\, (a_{12}-a_{21}) *_\r1\w z^3 + (a_{31}-a_{13}) *_\r1\w z^2+ (a_{23}-a_{32}) *_\r1\w z^1,
\end{aligned}
$$
and 
\begin{equation}\label{tau1}
    12 \tau_1=\psi  \iprod d\psi = (a_{12}-a_{21})z^3 +(a_{31}-a_{13})z^2+  (a_{23}-a_{32})z^1\,,
\end{equation}
which implies
$$
\begin{aligned}
12 \tau_1\w\psi =((a_{23}-a_{32})\omega_1+(a_{31}-a_{13})\omega_2+ (a_{12}-a_{21})\omega_3)\w z^1\w z^2\w z^3+12\tau_1\wedge *_\r1\,. 
\end{aligned}
$$
Hence, the condition $\tau_2=0$, which by \eqref{g2torsion} is equivalent to $4\tau_1\w \psi =d\psi$,  is satisfied if and only if 
$$\tau_1=0\qquad\mbox{and}\qquad
a_{ij}=a_{ji}\,,\quad\forall i,j\in\{1,2,3\}\,.
$$
These last two conditions are in fact equivalent by \eqref{tau1}. 
\end{proof}

The previous result proves Proposition \ref{coclosed}: if $\dim(\n')=2$ or $\dim(\n')=3$ and $\n'$ is calibrated by $\f$, every $\G_2T$-structure on $\n$ is coclosed.
Note however that, when $\dim(\n')=3$, there exist $\G_2T$-structures with $\n'$ not calibrated by $\f$ that are not coclosed: 
\begin{ex}\label{ex:notcal}
For any real numbers $x,y$ consider the 2-step nilpotent Lie algebra structure on $\R^7$ with $\n'=\langle e_4,e_6,e_7\rangle$ and structure 2-forms 
$$de^4=e^{13}+xe^{12}+ye^{23},\qquad de^6=e^{15}-ye^{25},\qquad de^7=-e^{35}+xe^{25}.$$
Then the $\G_2$ structure defined by \eqref{standard} satisfies
$d\psi=e^2\wedge\psi,$ so it has $\tau_2=0$ and $\tau_1=\frac14 e^2$. Note that the Lie algebra is isomorphic to $\n_{6,3}\oplus\R$ if $x=y=0$, and to $\n_{7,3,B}$ otherwise.
\end{ex}

Combining Lemma \ref{coclosed2} with \cite[Prop.~4.4]{DMR}, we obtain the following. 
\begin{cor}
    If $\dim(\n')=2$ and $\varphi$ is a $\G_2T$-structure on $\n$, then  $\n$ is decomposable. 
    In particular, $\n$ is isomorphic to one of the following Lie algebras: $\n_{5,2}\oplus\R^2$, $\mathfrak{h}_3\oplus\mathfrak{h}_3\oplus\R$, 
    $\mathfrak{h}_3^\C\oplus\R$, $\n_{6,2}\oplus\R$. 
\end{cor}

\begin{lemma}\label{dH2}
Assume that the $\mathrm{G}_2$-structure $\varphi$ given by \eqref{eq:G2n23} has $\tau_2=0$ (and thus $\tau_1=0$ by Lemma \ref{coclosed2}). 
Then  
$$
dH_{\varphi}=-4\lambda\,\left(\alpha_1\w z^2\w z^3+\alpha_2\w z^3\w z^1
+\alpha_3\w z^1\w z^2\right)+ \left(12 \lambda^2-\sum_{i=1}^3|\alpha_i|^2\right)*_\r1\,.
$$
\end{lemma}
\begin{proof}
Using \eqref{dphi} and \eqref{lambda3}, we compute  
$$
\begin{aligned}
*d\varphi=&\,*\left(\sum_{i=1}^3 \omega_i\w\alpha_i+\alpha_1\w z^2\w z^3+\alpha_2 
\w z^3\w z^1+ \alpha_3 
\w z^1 \w z^2\right)\\
=&\,\sum_{i=1}^3a_{ii}\,z^1\w z^2\w z^3+\sum_{i=1}^3(*_\r\alpha_i)\w  z^i\\
=&\, 6\lambda \,z^1\w z^2\w z^3+\sum_{i=1}^3(*_\r\alpha_i)\w  z^i\,.
\end{aligned}
$$
Since $
H_{\varphi}=2\lambda \varphi-*d\varphi$ (cf. Remark \ref{tau2tau1}), it  follows that
$$
H_\varphi=-6\lambda \,z^1\w z^2\w z^3-\sum_{i=1}^3(*_\r\alpha_i)\w  z^i
+2\lambda\sum_{i=1}^3 \omega_i\w z^i+2 \lambda\, z^1\w z^2 \w z^3\,. 
$$
We then compute using again \eqref{lambda3}:
$$
\begin{aligned}
dH_{\varphi}=&\, -6\lambda\,\left(\alpha_1\w z^2\w z^3-\alpha_2\w z^1\w z^3
+\alpha_3\w z^1\w z^2\right)-\sum_{i=1}^3 |\alpha_i|^2*_\r1\\
&\,+2\lambda \sum_{i=1}^3 \omega_i\w\alpha_i+2\lambda \alpha_1\w z^2\w z^3 - 2\lambda \alpha_2 
\w z^1 \w z^3+ 2\lambda \alpha_3 
\w z^1 \w z^2\\
&=\, -4\lambda\left(\alpha_1\w z^2\w z^3+\alpha_2\w z^3\w z^1
+\alpha_3\w z^1\w z^2\right)+ \left(12 \lambda^2-\sum_{i=1}^3|\alpha_i|^2\right)*_\r1\,,
\end{aligned}
$$
\end{proof}

We now focus on the instanton equation. We first note that any $F\in \Lambda^2\n^*$ can be written as 
\begin{equation}\label{eq:genericF}
F=F_0+v_1\w z^1 +v_2\w z^2 + v_3\w z^3 + a_1\,z^2\w z^3 + a_2\,z^3\w z^1 + a_3\,z^1\w z^2\,,
\end{equation}
for some $F_0\in \Lambda^2\r^*$, $v_1,v_2,v_3\in \r^*$, and $a_1,a_2,a_3\in\R$. 
\begin{lemma}\label{instaton2}
Let $F\in \Lambda^2\n^*$ be any $2$-form, expressed as in \eqref{eq:genericF}. Then 
$$
\begin{cases}
dF=0\,,\\
F\w\psi=0,
\end{cases}
\qquad\iff\qquad 
\begin{cases}
F_{0}\w\omega_i=0\,,\qquad \forall i\in\{1,2,3\},\\
\sum_{i=1}^3v_i\w \alpha_i =0,\\
\sum_{i=1}^3v_i\w\omega_i=0,\\
a_1=a_2=a_3=0.
\end{cases}
$$
\end{lemma}
\begin{proof}
Since $F_0,v_1,v_2,v_3$ are closed, we have
\[
dF  =  -v_1\w \alpha_1 - v_2\w \alpha_2- v_3\w \alpha_3 
     +(a_2\,\alpha_3 - a_3 \alpha_2) \w z^1 +(a_3\,\alpha_1 - a_1\,\alpha_3)\w z^2+ (a_1\,\alpha_2-a_2\,\alpha_1)\w z^3, 
\]
and then 
$$
dF=0 \iff 
\begin{cases} 
v_1\w \alpha_1 + v_2\w \alpha_2+v_3\w \alpha_3=0,\\
a_2\,\alpha_3 - a_3\,\alpha_2 = 0,\\
a_3\,\alpha_1 - a_1\,\alpha_3=0,\\
a_1\,\alpha_2-a_2\,\alpha_1=0.
\end{cases}
$$
If $\dim(\n')=2$, then $\alpha_3=0$ and $\alpha_1$ and $\alpha_2$ are linearly independent. 
If $\dim(\n')=3$, then $\alpha_1$, $\alpha_2$ and $\alpha_3$ are linearly independent. 
In both cases, the previous system shows that $a_1=a_2=a_3=0$, hence  
$$
dF=0 \iff 
\begin{cases} 
v_1\w \alpha_1 + v_2\w \alpha_2+v_3\w \alpha_3=0,\\
a_1=a_2=a_3=0.
\end{cases}
$$
Moreover,  
$$
\begin{aligned}
F\wedge \psi= &\,(F_0+v_1\w z^1 +v_2\w z^2 + v_3\w z^3)\w 
(\omega_1\w z^2\w z^3+\omega_2\w z^3\w z^1+\omega_3\w z^1\w z^2+*_\r1)\\
=&\,F_0\w \omega_1\w z^2\w z^3+F_0\w \omega_2\w z^3\w z^1+F_0\w \omega_3\w z^1\w z^2
+\left(\sum_{i=1}^3\omega_i\w v_i\right)\w z^1 \w z^2\w z^3\,,
\end{aligned}
$$
whence
$$
F\wedge \psi=0 \iff 
\begin{cases}
F_{0}\w\omega_i=0\,,\qquad \forall i\in\{1,2,3\}\\
\sum_{i=1}^3v_i\w\omega_i=0
\end{cases}
$$
which implies the claim. 
\end{proof}

\begin{prop}\label{prop:FHsystem}
Assume that the $\mathrm{G}_2$-structure $\varphi$ given by \eqref{eq:G2n23} has $\tau_2=0$ (and thus $\tau_1=0$ by Lemma \ref{coclosed2}). 
Consider $\epsilon_1,\dots,\epsilon_k \in \R\smallsetminus\{0\}$. Then the $2$-forms $F^r=F_0^r+\sum_{i=1}^3v_i^r\w z^i \in \Lambda^2\n^* $, $1\leq r \leq k$, satisfy    
\begin{equation}
\label{Fsystem}
\begin{cases}
dF^r=0\,,\\
F^r\wedge \psi=0\,,\\
dH_{\varphi}=\sum_{r=1}^k\epsilon_{r}\, F^{r}\wedge F^{r}\,,
\end{cases}
\end{equation}
if and only if the following  equations are satisfied 
\begin{eqnarray}
&\label{1} F_{0}^r\w\omega_i =0\,,& \forall i\in\{1,2,3\},~\forall r\in\{1,\ldots, k\}\,,\\
&\label{5} \sum_{r=1}^k \epsilon_r\, F_0^r\w v_i^r =0\,,& \forall i\in\{1,2,3\}\,,\\
&\label{2} \sum_{i=1}^3v_i^r\w \alpha_i =0\,,& \forall r\in\{1,\ldots, k\},\\
&\label{3} \sum_{i=1}^3 v_i^r\w\omega_i=0\,,& \forall r\in\{1,\ldots, k\},\\
&\label{4} \sum_{r=1}^k\epsilon_r\,|F_0^r|^2 = \sum_{i=1}^3 |\alpha_i|^2 -12 \lambda^2\,,\\
&\label{6} \sum_{r=1}^k \epsilon_r\,  v_1^r\w v_2^r =2\lambda\, \alpha_3\,,\\
&\label{7} \sum_{r=1}^k \epsilon_r\,  v_3^r\w v_1^r =2\lambda\,  \alpha_2\,,\\
&\label{8}\sum_{r=1}^k \epsilon_r\,  v_2^r\w v_3^r =2\lambda\,  \alpha_1\,.
\end{eqnarray}
\end{prop}
\begin{proof}
Equations \eqref{1}, \eqref{2}, \eqref{3}  come  from Lemma \ref{instaton2}. 
The first one shows that $F_0^r\in\Lambda^2\r^*$ is anti-self-dual, so $F_0^r\w F_0^r = -|F_0^r|^2 *_\r1$, for all $1\leq r \leq k$. 
We then have 
\begin{equation}
\begin{split}
F^r\w F^r=&-|F_0^r|^2 *_\r1-2 v_1^r\w v_2^r\w z^1\w z^2-2 v_1^r\w v_3^r\w z^1\w z^3\\&-2 v_2^r\w v_3^r\w z^2\w z^3+\sum_{i=1}^3 F_0^r\w v_i^r\w z^i.
\end{split}
\end{equation}
Taking into account Lemma \ref{dH2}, we obtain that the equation  $dH_{\varphi}=\sum_{r=1}^k\epsilon_{r}\, F^{r}\wedge F^{r}$ is equivalent to \eqref{5} and \eqref{4}--\eqref{8}.

\end{proof}
\begin{rmk}\label{rem} \ 
\begin{itemize}
    \item Equation \eqref{3} implies that, for every $r$, $v_3^r$ is determined by $v_1^r$ and $v_2^r$. Indeed, if we denote by  
$\{J_1,J_2,J_3\}$ the hypercomplex structure on $\r$ induced by $\{\omega_1,\omega_2,\omega_3\}$, then \eqref{3} is equivalent to
\[
v_3^r=-J_2 v_1^r+J_1 v_2^r\,.
\]
\item When $\lambda\neq0$, Equations \eqref{6}--\eqref{8} together with \eqref{3} imply \eqref{a=a}. Indeed, we have
\[
\begin{split}
\alpha_1\w \omega_2 &= \frac1{2\lambda}\left(\sum_{r=1}^k \epsilon_r\,  v_2^r\w v_3^r\right)\w\omega_2 
                    = - \frac1{2\lambda}\sum_{r=1}^k \epsilon_r\,  v_3^r \w v_2^r\w\omega_2\\
                    &= - \frac1{2\lambda}\sum_{r=1}^k \epsilon_r\,  v_3^r \w (-v_1^r\w\omega_1-v_3^r\w\omega_3)
                    = \frac1{2\lambda}\sum_{r=1}^k \epsilon_r\,  v_3^r \w v_1^r\w\omega_1 = \alpha_2\w\omega_1,  
\end{split}
\]
and, similarly, $\alpha_1\w\omega_3=\alpha_3\w\omega_1$ and $\alpha_2\w\omega_3=\alpha_3\w\omega_2$. 
\end{itemize}
\end{rmk}

In summary, if a $2$-step nilmanifold $M=\Gamma\backslash N$ admits an invariant solution $(\varphi,P,\theta,\langle \cdot,\cdot\rangle_{\mathfrak{k}})$
to the $\G_2$-system \eqref{eq:HetG2sys} and either $\dim(\n')= 2$ or $\dim(\n')= 3$ and $\f$ calibrates $\n'$, 
then equations \eqref{lambda3}, \eqref{a=a} and \eqref{1}-\eqref{8} must hold. 
This will allow us to determine certain constraints on the Lie algebra $\n$ or on the signature of the bilinear form 
$\langle\cdot,\cdot\rangle_{\mathfrak{k}}$ imposed by the existence of solutions. 
Conversely, we can state the following existence result, which is analogous to Proposition \ref{R6}.

\begin{prop}\label{prop:exconstr}
Consider the vector space $\R^4$ endowed with the standard basis $\{e_1,e_2,e_3,e_4\}$ and the standard $\mathrm{SU}(2)$-structure 
$(\omega_1 = e^{13}-e^{24},~\omega_2= -e^{14}-e^{23},~\omega_3= e^{12}+e^{34})$.  
Then, any solution of the equations \eqref{lambda3}, \eqref{a=a} and \eqref{1}-\eqref{8}
such that the $1$-forms $v_1,v_2,v_3\in(\R^4)^*$, the $2$-forms $\alpha_1, \alpha_2, \alpha_3 \in\Lambda^2(\R^4)^*$ and 
$F_0^r\in\Lambda^2(\R^4)^*$, $1\leq r \leq k$, are integral with respect to the standard basis of $\R^4$, and either $\alpha_1$, $\alpha_2$ are linearly independent and $\alpha_3=0$, or $\alpha_1,\ \alpha_2$ and $\alpha_3$ are linearly independent,
gives rise to a $2$-step nilmanifold endowed with a $k$-torus bundle 
and admitting an invariant solution to the ${\G}_2$-system \eqref{eq:HetG2sys}. 
\end{prop}
\begin{proof}
Consider the $7$-dimensional vector space $\n \coloneqq \R^4\oplus\langle z_1,z_2,z_3\rangle$ and endow it with the product metric $g$ for which the vectors 
$z_1,z_2,z_3$ are orthonormal. Then, if $\{e^1,e^2,e^3,e^4,z^1,z^2,z^3\}$ denotes the dual basis, $\n$ becomes a $2$-step nilpotent Lie algebra by imposing 
 \[
    (de^1,de^2,de^3,de^4,dz^1,dz^2,dz^3) = (0,0,0,0,\alpha_1,\alpha_2,\alpha_3), 
 \]
where the $2$-forms $\alpha_1,\alpha_2,\alpha_3$ are extended to $\n$ in the obvious way.
If $\alpha_3=0$ and $\alpha_1,\alpha_2$ are linearly independent, the derived algebra of $\n$ is $\n'=\langle z_1,z_2\rangle$. If $\alpha_1,\alpha_2,\alpha_3$ are linearly independent, the derived algebra of $\n$ is $\n'=\langle z_1,z_2,z_3\rangle$.
Since the $2$-forms $\alpha_1,\alpha_2,\alpha_3$  
are integral with respect to the basis $\mathcal{B}=\{e_1,e_2,e_3,e_4,z_1,z_2,z_3\}$ of $\n$, 
the simply connected $2$-step nilpotent Lie group $N=\exp(\n)$ 
has the cocompact lattice $\Gamma\coloneqq \exp(\mathrm{span}_\Z(6e_1,6e_2,6e_3,6e_4,z_1,z_2,z_3))$, 
and we can consider the $2$-step nilmanifold $M=\Gamma\backslash N$. 
Moreover, the hypothesis on $v_1,v_2,v_3$ and on the $F_0^r$'s ensures that the 
$2$-forms $F^r=F_0^r+\sum_{i=1}^3v_i^r\w z^i$, $1\leq r\leq k$, are integral with respect to the basis $\mathcal{B}$. 
Therefore, Theorem \ref{Theoapp1} ensures then the existence of a principal $\mathbb{T}^k$-bundle 
$P\to M=\Gamma\backslash N$ endowed with a connection $\theta$ whose curvature is $F_\theta= \sum_{r=1}^k F^rt_r$, 
where $\{t_1,\ldots,t_k\}$ is a basis of the Lie algebra of $\mathbb{T}^k$. 
The $3$-form $\f$ given by \eqref{eq:G2n23}, 
\[
\f = \omega_1\w z^1 + \omega_2\w z^2 + \omega_3\w z^3 + z^1 \w z^2 \w z^3\,, 
\]
defines a $\G_2$-structure on $\n$ inducing the product metric $g$. 
Since equations \eqref{lambda3} and \eqref{a=a} hold, $\f$ is coclosed and has $\tau_0 = \frac{12}{7}\lambda$, 
so it gives rise to a $\G_2$-structure of the same type on $M=\Gamma\backslash N$. 
If we define $\langle \cdot,\cdot\rangle_{\mathfrak{k}}\coloneqq\sum_{r=1}^k\epsilon_r t^{r}\otimes t^{r}$, 
we then obtain an invariant solution $(\varphi, P,\theta,\langle\cdot,\cdot\rangle_{\mathfrak{k}})$ 
to the $\G_2$-system \eqref{eq:HetG2sys} on $M=\Gamma\backslash N$.
\end{proof}

\subsection{Proof of Theorem \ref{main2}-(i)} \label{sect:sol023}
As we observed in the introduction,
it is sufficient to construct a solution to the $\G_2$-system \eqref{eq:HetG2sys} when $k_+=1$, $k_-=0$, 
and with $\n'$ calibrated by $\varphi$ if $\dim(\n')=3$. 
 We consider $\R^4$ endowed with the standard SU(2)-structure 
$(\omega_1 = e^{13}-e^{24},~\omega_2= -e^{14}-e^{23},~\omega_3= e^{12}+e^{34})$, 
and we apply Proposition \ref{prop:exconstr}. 
We first express $\alpha_1$, $\alpha_2$ (and $\alpha_3$) by solving  \eqref{lambda3} with $\lambda=0$ and \eqref{a=a}. 
Then, we choose $v_1=v_2=v_3=0$, so that most of the equations given in Proposition \ref{prop:FHsystem} are satisfied. 
We are left with equations \eqref{1} and \eqref{4}, which are easy to solve: 
it is sufficient to consider any non-zero anti-self-dual $2$-form $F_0^1\in\Lambda^2(\R^4)^*$ with integer coefficients and to define $\varepsilon_1:=\frac{\sum_{i=1}^3|\alpha_i|^2}{|F_0^1|^2}$ in order to satisfy \eqref{4}.  
If $\dim(\n')=2$, we obtain the following solutions to \eqref{lambda3} and \eqref{a=a}:
\begin{itemize}
    \item $\alpha_1=e^{13}$, $\alpha_2 = e^{23}$, which gives a solution for $\n\cong \n_{5,2}\oplus\R^2$;
    \item $\alpha_1=2e^{24}$, $\alpha_2 = e^{13}-e^{14}-e^{23}+e^{24}$
    which gives a solution for $\n\cong \mathfrak{h}_3\oplus\mathfrak{h}_3\oplus\R$;
    \item $\alpha_1=e^{13}-e^{24}$, $\alpha_2 = e^{14}+e^{23}$, which gives a solution for $\n\cong \mathfrak{h}_3^\C\oplus\R$;
    \item $\alpha_1= 2\,e^{13}$, $\alpha_2 = e^{14}+e^{23}$, which gives a solution for $\n\cong \n_{6,2}\oplus\R$.
\end{itemize}
If $\dim(\n')=3$, we have
\begin{itemize}
    \item $\alpha_1=e^{13}$, $\alpha_2 = 2e^{23}$, $\alpha_3 = e^{12}$, which gives a solution for $\n\cong \n_{6,3}\oplus\R$;
    \item $\alpha_1=e^{24}$, $\alpha_2 = e^{23}$, $\alpha_3 = 2e^{12}$, which gives a solution for $\n\cong \n_{7,3,A}$;
    \item $\alpha_1=e^{12}+e^{13}$, $\alpha_2 = 2e^{14}$, $\alpha_3 = -e^{24}+e^{34}$, which gives a solution for $\n\cong \n_{7,3,B}$;
    \item $\alpha_1= -e^{14}$, $\alpha_2 = 2e^{13}+e^{24}$, $\alpha_3 = e^{12}-e^{34}$, which gives a solution for $\n\cong \n_{7,3,B_1}$;
    \item $\alpha_1=e^{24}$, $\alpha_2 = e^{23}$, $\alpha_3 = e^{12}+e^{34}$, which gives a solution for $\n\cong \n_{7,3,C}$;
    \item $\alpha_1=e^{12}+e^{14}+e^{34}$, $\alpha_2 = -e^{13}$, $\alpha_3 = -2e^{24}$, which gives a solution for $\n\cong \n_{7,3,D}$;
    \item $\alpha_1=e^{12}-e^{34}$, $\alpha_2 = e^{13}+e^{24}$, $\alpha_3 = e^{14}-e^{23}$, which gives a solution for $\n\cong \n_{7,3,D_1}$.
\end{itemize}
By Proposition \ref{prop:exconstr}, the above algebraic data provide a solution to the $\G_2$-system \eqref{eq:HetG2sys} on nilmanifolds 
associated to each of the Lie algebras above and endowed with principal $S^1$-bundles.

\subsection{Proof of Theorem \ref{main2}-(ii)}  
Assume $\dim(\n')=2$. 

\subsubsection{The case $k=1$ and $\lambda\neq 0$}\label{5.2.1} 
Since $k=1,$ we can drop the superscript in $F^r_0,v^r_1,v^r_2,v^r_3$, in equations \eqref{1}--\eqref{8}, and consider them as a system in the unknowns $F_0,v_1,v_2,v_3$. 

It is easy to check that there are no solutions in this case. Indeed, since $F_0\in\Lambda^2\r^*$ is anti-self-dual, Equation \eqref{5} implies that 
either $F_0=0$ or all  $v_i$'s are $0$. 
The condition $v_i=0$ for every $i$ implies that the $\alpha_i$'s vanish, which gives a contradiction, since $\n$ is not abelian. Hence, we must have $F_0=0$.  Combining \eqref{6}--\eqref{8} with \eqref{2} we get 
$$
v_1\w v_2\w v_3=0\,.
$$
This means that $v_1,v_2,v_3$ are linearly dependent. Hence,  by \eqref{6}--\eqref{8}, $\alpha_1$, $\alpha_2$ and $\alpha_3$ are collinear, which contradicts the assumption $\dim(\n')=2$. 

\begin{rmk}\label{rem:k1}
    The same argument shows that there are no invariant solutions with $k=1$ and $\lambda\neq 0$ when $\dim(\n')=3$ and $\n'$ is calibrated by $\f$. 
\end{rmk}

\subsubsection{The case $k=2$ and $\lambda\neq 0$} We show that there are no solutions in this case either. Here we consider 
equation \eqref{lambda3} (see Remark \ref{rem}) together with the
system \eqref{1}--\eqref{8} with $\alpha_3=0$ in the variables $v_i\coloneqq v_{i}^1$ and $w_i \coloneqq v_i^2$: 
\begin{eqnarray}
&\label{1.2} F_{0}^r\w\omega_i =0\,,\qquad \forall i\in\{1,2,3\},\ \forall r\in\{1,2\}\,,\\
&\label{5.2}  \epsilon_1\, F_0^1\w v_i+\epsilon_2\,F_0^2\w w_i=0\,,\quad \forall i\in\{1,2,3\},\\
&\label{2.2} v_1\w \alpha_1+v_2\w \alpha_2= w_1\w \alpha_1+w_2\w \alpha_2=0\,,\\
&\label{3.2}  v_1\w\omega_1+v_2\w\omega_2 +v_3\w\omega_3 = w_1\w\omega_1+w_2\w\omega_2 +w_3\w\omega_3 =0\,,\\
&\label{4.2} \epsilon_1\,|F_0^1|^2 + \epsilon_2\,|F_0^2|^2=-12 \lambda^2+ |\alpha_1|^2+|\alpha_2|^2\,,\\
& \label{6.2} \epsilon_1\,  v_1\w v_2+\epsilon_2\, w_1\w w_2 = 0\,,\quad\quad \\
& \label{6.7} \epsilon_1\,  v_3\w v_1+\epsilon_2\,  w_3\w w_1=2\lambda\,  \alpha_2\,,\\
&  \label{8.2} \epsilon_1\,  v_2\w v_3+\epsilon_2\,  w_2\w w_3=2\lambda\,  \alpha_1\,.
\end{eqnarray}
Equation \eqref{6.2} implies that $v_1,v_2,w_1,w_2$ belong to the same plane, which we denote by $P$. 
Moreover, since we are assuming $\lambda\neq 0$, \eqref{6.7} and \eqref{8.2} give  
\begin{equation}\label{a12}
\alpha_2=-\frac1{2\lambda}( \epsilon_1\,  v_1\w v_3+\epsilon_2\,  w_1\w w_3)\,,\qquad   \alpha_1=\frac1{2\lambda}(\epsilon_1\,  v_2\w v_3+\epsilon_2\,  w_2\w w_3)\,.
\end{equation}
Let $\{f_1,f_2,f_3,f_4\}$ be an orthonormal basis of $\r^*$ 
such that $\{f_1,f_2\}$ is a basis of $P$ and write 
\begin{equation}\label{ab}
v_1\w  v_2=a\,f_{1}\w f_2\,,\quad v_1\w  w_2+w_1 \w   v_2=2b \, f_{1}\w f_2\,.
\end{equation}
\begin{lemma}\label{NEW}
The following relation holds
$$
\epsilon_1\, a^2 + \epsilon_2\, b^2 = 0\,. 
$$
\end{lemma}
\begin{proof}
Equations  \eqref{2.2} and \eqref{a12} imply the two relations
\begin{eqnarray}
&\label{primo}2\epsilon_1\, v_1\w  v_2\w v_3+\epsilon_2w_3\w (v_1\w  w_2+w_1 \w   v_2)=0\,,\\
&\label{secondo}2\epsilon _2\, w_1\w  w_2\w w_3+\epsilon_1v_3\w (v_1\w  w_2+w_1 \w   v_2)=0\,.
\end{eqnarray}
Using \eqref{6.2}, equation \eqref{secondo} becomes
\begin{equation}\label{terzo}
2\, v_1\w  v_2\w w_3-v_3\w (v_1\w  w_2+w_1 \w   v_2)=0\,.
\end{equation}
We can rewrite \eqref{primo} and \eqref{terzo} as  
\begin{eqnarray}
&\label{17}(\epsilon_1\,a\, v_3 + \epsilon_2\,b\,w_3)\w f_1\w  f_2 =0,\\
&\label{18} (a\,w_3-b\,v_3)\w f_1\w  f_2=0\,.
\end{eqnarray}
Moreover by setting 
\begin{equation}\label{vw}
    \begin{aligned}
v_3=v_3'+xf_3+x'f_4\,,\\
w_3=w_3'+yf_3+y'f_4\,,
\end{aligned}
\end{equation}
with $v_3',w_3'\in P$ and $x,x',y,y'\in\R$, system \eqref{17}--\eqref{18} becomes 
\begin{equation}\label{vw1}
\begin{cases}
\epsilon_1\,a\,x+\epsilon_2\,b\,y=0\\
b\,x-a\,y=0\\
\epsilon_1\,a\,x'+\epsilon_2\,b\,y'=0\\
b\,x'-a\,y'=0
\end{cases}
\end{equation}
which implies 
$$
x=y=x'=y'=0\quad \mbox{ or }\quad \epsilon_1a^2+\epsilon_2b^2=0\,.
$$
If $x=y=x'=y'=0$, then both $v_3$ and $w_3$ belong to $P$, so by \eqref{a12} the structure 2-forms $\alpha_1$ and $\alpha_2$ are proportional, 
which contradicts $\dim(\n')=2$.  
Hence, $\epsilon_1a^2+\epsilon_2b^2=0$, as claimed. 
\end{proof}

\begin{lemma}\label{NEW2}
The forms $v_1\w v_2$ and $v_1\w  w_2+w_1 \w v_2$ are nonzero, namely $ab\neq0$ in Equation \eqref{ab}.
\end{lemma}
\begin{proof}
Assume by contradiction that $ab = 0$. Then, Lemma \ref{NEW} implies $a=b=0$. Thus
\begin{equation}\label{vw0}
    v_1\w v_2=w_1\w w_2=v_1\w  w_2+w_1 \w   v_2=0\,.
\end{equation}
Therefore, $v_1$ is collinear with $v_2$ and $w_1$ is collinear with $w_2$. 

We claim that in this case the four vectors $v_1,v_2,w_1,w_2$ are all collinear. 
Assume first that $v_1$ and $w_1$ are nonzero. Then, we can write 
$$
v_2=sv_1\,,\quad w_2=tw_1,
$$
for some $s,t\in \R$. Hence
$$
0=v_1\w  w_2+w_1 \w   v_2=(t-s) v_1\w  w_1.
$$
Since $\alpha_1$ and $\alpha_2$ are not proportional, $t\neq s$ and then $v_1$ and $w_1$ have to be collinear.  
Hence, the four vectors $v_1,v_2,w_1,w_2$ are collinear. 
The case $v_1=w_1=0$ is impossible since it would imply $\alpha_2=0$. If $v_1=0$, $w_1\ne 0$, we get $w_1\w w_2=w_1 \w   v_2=0$ from \eqref{vw0}, 
which implies that $w_2,v_2\in \langle w_1\rangle$. The case $v_1\ne 0$, $w_1= 0$ is similar. This proves the claim.

Using the $\mathrm{SU}(2)$-freedom on $\r^*$, we can choose an orthonormal basis $\{e_1,e_2,e_3,e_4\}$ of $\r^*$ in order to have 
$\omega_1 = e^{13}-e^{24},~\omega_2= -e^{14}-e^{23},~\omega_3= e^{12}+e^{34}$, and
$$
\begin{array}{lll}
v_1=a_1e_1\,, & v_2=a_2e_1\,, & v_3=a_1e_4+a_2e_3\,,\\
w_1=b_1e_1\,, & w_2=b_2e_1\,, & w_3=b_1e_4+b_2e_3\,,
\end{array}
$$
where $a_1,a_2,b_1,b_2\in\R$ with $a_1b_1\neq0\neq a_2b_2$. The expressions of $v_3$ and $w_3$ follow from Remark \ref{rem}. 

Now, the second equation of \eqref{a12} gives   
$$
\begin{aligned}
\alpha_1=&\frac{1}{2\lambda} (\epsilon_1\,  v_2 \w v_3 +\epsilon_2\,  w_2 \w w_3)=
\frac{1}{2\lambda} (\epsilon_1a_2\, e_1\w (a_1e_4+a_2e_3)  +\epsilon_2b_2\,  e_1\w (b_1e_4+b_2e_3) )\\
=&\frac{1}{2\lambda}\left((a_2^2 \epsilon_1+b_2^2 \epsilon_2)\, e_{1}\w e_3+(a_1a_2\epsilon_1+b_1b_2\epsilon_2) \, e_{1}\w e_4\right)\,,
\end{aligned}
$$
and, analogously, the first equation in \eqref{a12} implies 
$$
\alpha_2=-\frac{1}{2\lambda}\left((a_1^2 \epsilon_1+b_1^2 \epsilon_2)\, e_{1}\w e_4+(a_1a_2\epsilon_1+b_1b_2\epsilon_2) \, e_{1}\w e_3\right)\,.
$$
Hence, from \eqref{lambda3} we obtain 
$$
6\lambda=\langle \omega_1,\alpha_1\rangle+\langle \omega_2,\alpha_2\rangle=\frac{1}{2\lambda}(\epsilon_1 a_2^2+\epsilon_2 b_2^2)+\frac{1}{2\lambda}(\epsilon_1 a_1^2+\epsilon_2 b_1^2)\,,
$$
whence it follows that
\begin{equation}\label{12l}
    12\lambda^2=\epsilon_1 (a_1^2+a_2^2)+\epsilon_2 (b_1^2+b_2^2)\,.
\end{equation}

We now consider Equation \eqref{5.2}, which gives
$$e_1\w (a_1F_0^1+b_1F_0^2)=0,\qquad e_1\w (a_2F_0^1+b_2F_0^2)=0\,.$$
Since the wedge product with a non-zero vector maps the space $\Lambda^-\r^*$ of anti-self-dual forms on $\r$ injectively into $\Lambda^3\r^*$, we get $$a_1F_0^1+b_1F_0^2 = 0 = a_2F_0^1+b_2F_0^2\,.$$ 
The determinant of the matrix $\begin{pmatrix}
    a_1&a_2\\b_1&b_2
\end{pmatrix}$ is non-zero, as otherwise $\alpha_1$ and $\alpha_2$ would be proportional. 
Therefore, we obtain $F_0^1=F_0^2=0$, so \eqref{4.2} becomes
\begin{equation}\label{a123}
    |\alpha_1|^2+|\alpha_2|^2=12 \lambda^2\,.
\end{equation}
From \eqref{12l}, \eqref{a123}, and the explicit formulas of $\alpha_1$ and $\alpha_2$ above, it follows that 
$$
\begin{aligned}
12\lambda^2=|\alpha_1|^2+|\alpha_2|^2&=\frac{1}{4\lambda^2}\left((a_1^2 \epsilon_1+b_1^2 \epsilon_2)^2+(a_2^2 \epsilon_1+b_2^2 \epsilon_2)^2+2(a_1a_2\epsilon_1+b_1b_2\epsilon_2)^2\right)\\
&= \frac{1}{4\lambda^2}\left(\left((a_1^2 \epsilon_1+b_1^2 \epsilon_2)+(a_2^2 \epsilon_1+b_2^2 \epsilon_2)\right)^2-2\epsilon_1\epsilon_2(a_1b_2-a_2b_1)^2\right)\\
&= \frac{1}{4\lambda^2}\left(\left(\epsilon_1 (a_1^2+a_2^2)+\epsilon_2 (b_1^2+b_2^2)\right)^2-2\epsilon_1\epsilon_2(a_1b_2-a_2b_1)^2\right)\\
&= \frac{1}{4\lambda^2}\left(144 \lambda^4-2\epsilon_1\epsilon_2(a_1b_2-a_2b_1)^2\right)\\
&= 36\lambda^2-\frac{1}{2\lambda^2}\epsilon_1\epsilon_2(a_1b_2-a_2b_1)^2\,,
\end{aligned}
$$
thus  
\begin{equation}\label{e48}
    \epsilon_1\epsilon_2(a_1b_2-a_2b_1)^2=48\lambda^4\,.
\end{equation}
In particular, \eqref{a123} and \eqref{e48} show that $\epsilon_1\epsilon_2>0$. 
Using this and \eqref{12l}, we get 
$$|\epsilon_1\epsilon_2(a_1b_2-a_2b_1)^2|\le |\epsilon_1 (a_1^2+a_2^2)|\cdot|\epsilon_2 (b_1^2+b_2^2)|\le \frac14(\epsilon_1 (a_1^2+a_2^2)+\epsilon_2 (b_1^2+b_2^2))^2=36\lambda^4,$$ 
which contradicts \eqref{e48}. This finishes the proof.
\end{proof}

\begin{cor}
    If $\dim(\n')=2$, there are no invariant solutions to the $\G_2$-system \eqref{eq:HetG2sys} with $k=2$, $\lambda\neq0$ and definite $\langle\cdot,\cdot\rangle_{\mathfrak{k}}$. 
\end{cor}
\begin{proof}
    If $\langle\cdot,\cdot\rangle_{\mathfrak{k}}$ is definite, Lemma \ref{NEW} implies $a=b=0$ which contradicts Lemma \ref{NEW2}.
\end{proof}

\medskip 
We now claim that having a solution to the system \eqref{1.2}--\eqref{8.2} implies 
$$
\alpha_1\w \alpha_1=\alpha_2\w \alpha_2=\alpha_1\w\alpha_2=0,
$$
so that $\n$ is isomorphic to $ \n_{5,2}\oplus \R^2$ (cf.~Appendix \ref{2stepnilclass}). 
Since $\lambda\neq0$, from the expressions of $\alpha_1$ and $\alpha_2$ given in \eqref{a12} we obtain 
\begin{eqnarray}
&&\label{alpha11}\alpha_1\w\alpha_1=\frac{\epsilon_1\epsilon_2}{2\lambda^2}\, v_2\w v_3\w  w_2\w w_3\,,\\
&&\label{alpha22}\alpha_2\w\alpha_2=\frac{\epsilon_1\epsilon_2}{2\lambda^2}\, v_1\w v_3\w  w_1\w w_3\,,\\
&&\label{alpha12}\alpha_1\w\alpha_2=\frac{\epsilon_1\epsilon_2}{4\lambda^2}\, v_3\w w_3\w  (v_2\w w_1+v_1\w w_2)\,.
\end{eqnarray}
Equation \eqref{2.2} implies then
\[
\begin{split}
0   &=  w_3\w (v_1\w\alpha_1+v_2\w \alpha_2)=\frac{ \epsilon_1}{2\lambda}w_3\w v_3\w ( v_1\w v_2- v_2\w  v_1) 
    =   \frac{ \epsilon_1}{\lambda} w_3\w v_3\w v_1\w v_2 \,,\\
0&= w_3\w (w_1\w\alpha_1+w_2\w\alpha_2) = \frac{\epsilon_1}{2\lambda} w_3 \w v_3 \w ( v_1 \w w_2 + w_1 \w v_2 ).
\end{split}
\]
By Lemma \ref{NEW2} we know that $v_1\w v_2$ and $v_1 \w w_2 + w_1 \w v_2$ are non-zero multiples of $f_1\w f_2$, so we have $w_3 \w v_3 \w f_1\w f_2=0$. 
Since $v_1$, $v_2$, $w_1$, $w_2$ belong to $P=\langle f_1,f_2\rangle$, we deduce that $\alpha_1 \w \alpha_1 = \alpha_2 \w \alpha_2 = \alpha_1\w\alpha_2=0,$ 
whence $\n\cong\n_{5,2}\oplus\R^2$. 

\medskip 
To conclude the proof, we show that there are no solutions in the only remaining case: $\n\cong \n_{5,2}\oplus\R^2$ and $\epsilon_1\epsilon_2<0$. 
We consider the decomposition of $\alpha_1$ and $\alpha_2$ into self-dual and anti-self-dual parts:
$$
\alpha_1=\alpha_1^++\alpha_1^-,\qquad \alpha_2=\alpha_2^++\alpha_2^-.
$$
Then 
$$
0=\alpha_1\wedge\alpha_1=(|\alpha_1^+|^2-|\alpha_1^-|^2)*_\r1,
$$
implies $|\alpha_1^+|^2=|\alpha_1^-|^2$. Similarly, $|\alpha_2^+|^2=|\alpha_2^-|^2$. From \eqref{lambda3} it follows that
\begin{align*}
    |\alpha_1|^2+|\alpha_2|^2&=2(|\alpha_1^+|^2+|\alpha_2^+|^2)\geq\langle\alpha_1^+,\omega_1\rangle^2+\langle\alpha_2^+,\omega_2\rangle^2=\langle\alpha_1,\omega_1\rangle^2+\langle\alpha_2,\omega_2\rangle^2\\&\geq\frac12(\langle\alpha_1,\omega_1\rangle+\langle\alpha_2,\omega_2\rangle)^2=18\lambda^2.
    \end{align*}
    Using \eqref{4.2} we then have
    \[
    \epsilon_1\,|F_0^1|^2 + \epsilon_2\,|F_0^2|^2=-12 \lambda^2+ |\alpha_1|^2+|\alpha_2|^2 \geq 6\lambda^2,
    \]
    whence we deduce that at least one of  the anti-self-dual forms
    $F_0^1$ and $F_0^2$ is non-vanishing. We assume $F_0^2\neq 0$. Taking the Hodge dual of \eqref{5.2} yields 
    $$\epsilon_1\, F_0^1(v_i)+\epsilon_2\,F_0^2(w_i)=0\,,\qquad \forall i\in\{1,2,3\}\,,$$
    showing that there exists an endomorphism $A \coloneqq -\frac{\varepsilon_1}{\varepsilon_2}(F_0^2)^{-1}\circ F_0^1 \in\mathrm{End}(\r^*)$
    such that $w_i=A(v_i)$ for all $i\in\{1,2,3\}$. Moreover, since $F_0^1$ and $F_0^2$ belong to $\Lambda^-\r^*$, 
    the endomorphism $A$ is of the form $\mu I_4+B$, with $\mu\in\R$ and $B\in\Lambda^-\r^*$.

    In view of Lemma \ref{NEW2} $v_1\wedge v_2\ne 0$, so 
    from \eqref{6.2} it follows that $A$ preserves the plane $P$ and its orthogonal complement $P^\perp$.  
    Therefore, using the fact that $w_3=A(v_3)$ together with \eqref{vw}, we obtain $A(xf_3+x'f_4)=yf_3+y'f_4$. 
    On the other hand, the second and fourth equations in the system \eqref{vw1} together with Lemma \ref{NEW2} show that the couples 
    $(x,x')$ and $(y,y')$ are proportional: $(x,x')=\frac ab(y,y')$. 
    Moreover, we must have $(y,y')\ne (0,0)$, since otherwise $v_3,w_3$ would belong to $P$ and $\alpha_1$, $\alpha_2$ would be proportional. 
    Therefore, $xf_3+x'f_4 \in P^\perp$ is an eigenvector for $A$, and thus also for $B=A-\mu I_4$. 
    Since $B$ is skew-symmetric, its only possible real eigenvalue is 0, so $B$ has non-trivial kernel. 
    However, the only element in $\Lambda^-\r^*$ with non-trivial kernel is 0. 
    This shows that $B=0$ and thus $w_i=\mu v_i$ for all $i\in\{1,2,3\}$. By \eqref{6.2} we thus get $\varepsilon_1+\mu^2\varepsilon_2=0$, so \eqref{6.7} gives $\alpha_2=0$, which is a contradiction. 

\smallskip

\noindent
This concludes the proof of Theorem \ref{main2}-(ii): there are no invariant solutions to the $\G_2$-system \eqref{eq:HetG2sys} with $k=2$ and $\lambda\ne 0$ when $\dim(\n')=2$.

\subsubsection{The case $k=3$ and $\lambda\neq 0$}\label{sect:exn52} 
In this case, we show that there is a solution when $\n\cong\n_{5,2}\oplus\R^2$. We may assume $\alpha_3=0$ and rewrite system \eqref{1}--\eqref{8} in the variables 
$v_i\coloneqq v_{i}^1$, $w_{i}\coloneqq v_{i}^2$, and $u_{i}\coloneqq v_{i}^3$,  $1\leq i \leq 3$ as follows: 
\begin{eqnarray}
&\ F_{0}^r\w\omega_i =0\,,\qquad \forall i,r\in\{1,2,3\}\,, \label{T3n'2first}\\
&\  \epsilon_1\, F_0^1\w v_i+\epsilon_2\,F_0^2\w w_i +\epsilon_3\,F_0^3\w u_i =0\,,\qquad \forall i\in\{1,2,3\}\,,\\
&\ v_1\w \alpha_1+v_2\w \alpha_2= w_1\w \alpha_1+w_2\w \alpha_2 = u_1\w \alpha_1+u_2\w \alpha_2 =0\,,\\
&\ \sum_{i=1}^3 v_i\w\omega_i = \sum_{i=1}^3 w_i\w\omega_i = \sum_{i=1}^3 u_i\w\omega_i =0\,,\\
&\ \epsilon_1\,|F_0^1|^2 + \epsilon_2\,|F_0^2|^2 + \epsilon_3\,|F_0^3|^2 =-12 \lambda^2+ |\alpha_1|^2+|\alpha_2|^2\,,\\
&\  \epsilon_1\, v_1\w v_2+\epsilon_2\, w_1\w w_2+\epsilon_3\, u_1\w u_2=0\,,\quad\\
&\  \epsilon_1\,  v_3\w v_1+\epsilon_2\,  w_3\w w_1 +\epsilon_3\,  u_3\w u_1 =2\lambda  \alpha_2\,,\\
&\   \epsilon_1\,  v_2\w v_3+\epsilon_2\,  w_2\w w_3 +\epsilon_3\,  u_2\w u_3 =2\lambda  \alpha_1\,. \label{T3n'2last}
\end{eqnarray}

    According to Proposition \ref{prop:exconstr} and Remark \ref{rem}, it is sufficient to provide a suitable solution to equations 
    \eqref{lambda3} and \eqref{T3n'2first}-\eqref{T3n'2last} on the vector space 
    $\R^4 = \langle e_1,e_2,e_3,e_4\rangle$ endowed with the standard SU(2)-structure 
    $(\omega_1 = e^{13}-e^{24},~\omega_2= -e^{14}-e^{23},~\omega_3= e^{12}+e^{34})$. 
    One such solution is the following:
    \[
    v_1= e^1,~ v_2= 0,~ v_3 = e^4,\qquad w_1 = 0,~ w_2 = e^1,~ w_3 = e^3,\qquad u_1=u_2=u_3=0, 
    \]
    \[
    \epsilon_1 = 12\lambda^2-t,\quad \epsilon_2 = t,\quad \epsilon_3 = 12\lambda^2-3t + \frac{1}{4\lambda^2}\,t^2,\quad (t\in\R),
    \]
    \[
    F^1_0= F^2_0 = 0,\qquad F^3_0 = e^{13} + e^{24}.
    \]
    On the $7$-dimensional vector space $\n = \langle e_1,e_2,e_3,e_4\rangle\oplus\langle z_1,z_2,z_3\rangle$ we then have
    \[
    \alpha_1 = \frac{t}{2\lambda}\, e^{13},\quad \alpha_2 = \frac{t-12\lambda^2}{2\lambda}\,e^{14},
    \]
    and
    \[
    F^1 = e^{1}\w z^1 + e^{4} \w z^3 ,\quad F^2 = e^{1}\w z^2 + e^{3}\w z^3, \quad F^3 = e^{13} + e^{24}.
    \]
    We note that the 2-forms $F^1,F^2,F^3$ are integral with respect to the considered basis of $\n$. 
    Since we also want the $2$-forms $\alpha_1$ and $\alpha_2$ to be integral with respect to the same basis, 
    we let $t=2\lambda$, with $\lambda\in\Z\smallsetminus\{0\}$. 
    The 2-step nilpotent Lie algebra $\n$ has then structure equations 
    \[
    (de^1,de^2,de^3,de^4,dz^1,dz^2,dz^3) = (0,0,0,0,e^{13},(1-6\lambda)\,e^{14},0), 
    \]
    $2$-dimensional derived algebra $\n'$, and it is easily seen to be isomorphic to $\n_{5,2}\oplus\R^2$. 
    We also note that the choice $t=2\lambda$ implies
    \[
    \epsilon_1 = 12\lambda^2-2\lambda,\quad \epsilon_2 = 2\lambda,\quad \epsilon_3 = 12\lambda^2-6\lambda + 1,
    \]
    so that the possible bilinear forms $\langle\cdot,\cdot\rangle_{\mathfrak{k}}$ on $\mathfrak{k} = \mathrm{Lie}(\mathbb{T}^3)$ corresponding to this solution have signature $(k_+,k_-)=(3,0)$ when $\lambda\geq1$ 
    and $(k_+,k_-)=(2,1)$ when $\lambda\leq -1$. 

\medskip 

\subsection{Proof of  Theorem \ref{main2}-(iii)} Finally we assume $\dim(\n')=3$ and that $\varphi$ calibrates $\n'$. 

\smallskip

By Remark \ref{rem:k1}, there are no invariant solutions with $\lambda\neq 0$ and $k=1$, so we directly focus on the case $k=2$ and $\lambda\neq 0$. 
In this case, the system \eqref{1}--\eqref{8} can be written in the variables $v_i\coloneqq v_{i}^1$ and $w_i\coloneqq v_i^2$ as follows
 \begin{eqnarray}
&\label{71} F_{0}^r\w\omega_i =0\,,\qquad \forall i\in\{1,2,3\},\ \forall r\in\{1,2\}\,,\\
&  \epsilon_1\, F_0^1\w v_i+\epsilon_2\,F_0^2\w w_i=0\,,\qquad \forall i\in\{1,2,3\}\\
&\label{72} v_1\w \alpha_1+v_2\w \alpha_2 + v_3\w\alpha_3= w_1\w \alpha_1+w_2\w \alpha_2 + w_3\w\alpha_3=0\,,\\
&  v_1\w\omega_1+v_2\w\omega_2 +v_3\w\omega_3 = w_1\w\omega_1+w_2\w\omega_2 +w_3\w\omega_3 =0\,,\\
& \epsilon_1\,|F_0^1|^2 + \epsilon_2\,|F_0^2|^2=-12 \lambda^2+ |\alpha_1|^2+|\alpha_2|^2+|\alpha_3|^2\,,\\
&\label{76} \epsilon_1\,  v_1\w v_2+\epsilon_2\, w_1\w w_2=2\lambda\alpha_3\,,\\
&\label{77}  \epsilon_1\,  v_3\w v_1+\epsilon_2\,  w_3\w w_1=2\lambda  \alpha_2\,,\\
&\label{78}  \epsilon_1\,  v_2\w v_3+\epsilon_2\,  w_2\w w_3=2\lambda  \alpha_1\,.
\end{eqnarray}

We first prove that the existence of a solution of the system with $\lambda\neq0$ forces the $7$-dimensional Lie algebra $\n$ to be isomorphic either 
to $\n_{6,3}\oplus\R$ or to $\n_{7,3,A}$. Then, we provide an example for the first possibility and show that the second one cannot occur. 

\smallskip 
Since we assume $\lambda\neq0$, we can write from \eqref{76}-\eqref{78}:
$$
\alpha_1=\frac{1}{2\lambda}(\epsilon_1\,  v_2\w v_3+\epsilon_2\,  w_2\w w_3),~
\alpha_2=\frac{1}{2\lambda}(\epsilon_1\,  v_3\w v_1+\epsilon_2\,  w_3\w w_1),~
\alpha_3=\frac{1}{2\lambda}(\epsilon_1\,  v_1\w v_2+\epsilon_2\, w_1\w w_2).
$$
Using these expressions in \eqref{72}, we obtain  
\begin{eqnarray}
3\epsilon_1\,  v_1\w v_2\w v_3+\epsilon_2\,  v_1\w w_2\w w_3+\epsilon_2\,  v_2\w w_3\w w_1+\epsilon_2\, v_3\w w_1\w w_2&=&0, \label{I}\\
3\epsilon_2\, w_1\w w_2\w w_3+\epsilon_1\,  w_1\w v_2\w v_3+\epsilon_1\,  w_2\w v_3\w v_1 + \epsilon_1\, w_3\w v_1\w v_2&=&0. \label{II}
\end{eqnarray}

Next we show that $\alpha_i\w\alpha_i=0$ for $1\leq i \leq 3$. Assume by contradiction that 
    \[
    \alpha_3\w\alpha_3 = \frac{\epsilon_1\epsilon_2}{2\lambda^2} v_1  \w w_1 \w w_2 \w v_2 \neq0.
    \]
    Then $\{v_1,v_2,w_1,w_2\}$ is a basis of $\r^*$, and we can write
    \[
    v_3 = a_1 v_1+a_2v_2+a_3w_1+a_4w_2,\quad w_3 = b_1 v_1+b_2v_2+b_3w_1+b_4w_2,
    \]
    for some real numbers $a_k,b_k$, $1\leq k\leq 4$. Substituting these expressions into \eqref{I} and \eqref{II}, we obtain  
    \[
    \begin{split}
    0   &=  (3a_3\epsilon_1 - b_1\epsilon_2)\, v_1\w v_2\w w_1 + (3a_4\epsilon_1 - b_2\epsilon_2)\, v_1\w v_2\w w_2\\
        &\quad    + (b_3-a_1 )\epsilon_2\, v_1\w w_2\w w_1 + (b_4-a_2)\epsilon_2\,v_2\w w_2\w w_1,\\
    0   &=  (3b_1\epsilon_2-a_3\epsilon_1 )\,w_1\w w_2\w v_1 + (3b_2\epsilon_2-a_4\epsilon_1 )\,w_1\w w_2\w v_2\\
        &\quad    +(a_1 - b_3)\epsilon_1\,w_1\w v_2\w v_1  + (a_2 - b_4)\epsilon_1\, w_2\w v_2 \w v_1. 
    \end{split}
    \]
    These equations give 
    \[
    b_4 = a_2,\quad b_3 = a_1,\quad b_1=b_2=a_3=a_4=0,
    \]
    and thus 
    \[
    v_3 = a_1 v_1+a_2v_2,\qquad w_3 = a_1 w_1 + a_2 w_2. 
    \]
From \eqref{77} and \eqref{78} we then get 
\[
\begin{split} 
2\lambda  \alpha_2 &=\epsilon_1\,  v_3\w v_1+\epsilon_2\,  w_3\w w_1=a_2(\epsilon_1\,  v_2\w v_1+\epsilon_2\,  w_2\w w_1)\,,\\
2\lambda  \alpha_1 &= \epsilon_1\,  v_2\w v_3+\epsilon_2\,  w_2\w w_3=a_1(\epsilon_1\,  v_2\w v_1+\epsilon_2\,  w_2\w w_1)\,.
\end{split}
\]
Since $\lambda\ne 0$, this contradicts the fact that $\alpha_1$ and $\alpha_2$ are linearly independent. 
Consequently, $\alpha_3\w\alpha_3 =0$. By similar arguments, we also obtain $\alpha_2\w\alpha_2 =\alpha_1\w\alpha_1 =0$. 

Replacing the orthonormal basis $\{z_1,z_2,z_3\}$ of $\n'$ with $\left\{\frac1{\sqrt2}(z_1+z_2),\frac1{\sqrt2}(z_1-z_2),z_3\right\}$, the structure form $\alpha_1$ is replaced by $\frac1{\sqrt2}(\alpha_1+\alpha_2)$, so the same argument as before shows that $0=(\alpha_1+\alpha_2)\wedge(\alpha_1+\alpha_2)=2\alpha_1\w \alpha_2$, and similarly $0=\alpha_1\w \alpha_3=\alpha_2\w \alpha_3$. 
This shows that either $\n\cong \n_{6,3}\oplus\R$ or $\n\cong\n_{7,3,A}$.

\smallskip

To conclude the discussion, we provide an invariant solution  to the $\G_2$-system \eqref{eq:HetG2sys} with $\lambda\neq0$ and $k=2$ for $\n\cong \n_{6,3}\oplus\R$ and show that there are no such solutions when $\n\cong\n_{7,3,A}$.

\subsubsection{The case $\n\cong \n_{6,3}\oplus\R$}\label{EX5.12}
    We describe here an invariant solution to the $\G_2$-system \eqref{eq:HetG2sys} with $\lambda\neq0$ on a $2$-step nilmanifold $M=\Gamma\backslash N$
    corresponding to a Lie algebra $\n\cong\n_{6,3}\oplus\R$ and endowed with a $\mathbb{T}^2$-bundle. 
    By Proposition \ref{prop:exconstr} and Remark \ref{rem}, it is sufficient to provide a suitable solution to equations 
    \eqref{lambda3} and \eqref{71}-\eqref{78} on the vector space $\R^4=\langle e_1,e_2,e_3,e_4\rangle$ endowed with the standard 
    SU(2)-structure $(\omega_1,\omega_2,\omega_3)$. 
    We consider the following solution:
    \[
     v_1 = 2e^1,\quad v_2 = e^2,\quad v_3 = e^4, \qquad w_1 = 0,\quad w_2 = 2e^4,\quad w_3 = 2e^2,\\
    \]
    \[
    \epsilon_1 =2\lambda^2, \quad \epsilon_2 = \frac32\,\lambda^2,\qquad F^1_0 = F^2_0 = 0, 
    \]
    so that on the $7$-dimensional vector space $\n  = \langle e_1,e_2,e_3,e_4\rangle\oplus\langle z_1,z_2,z_3\rangle$ we have
    \[
    \alpha_1 = -2\lambda\,e^{24},\quad \alpha_2 = -2\lambda\,e^{14},\quad \alpha_3 = 2\lambda\,e^{12},
    \]
    \[
    F^1 = 2e^{1} \w z^1 + e^{2} \w z^2 +e^{4} \w z^3,\qquad F^2 = 2e^{4} \w z^2 + 2e^{2} \w z^3. 
    \]
    
    Now, $\n$ is a $7$-dimensional 2-step nilpotent Lie algebra with structure equations 
    \[
    (de^1,de^2,de^3,de^4,dz^1,dz^2,dz^3) = (0,0,0,0,-2\lambda\,e^{24}, -2\lambda\,e^{14}, 2\lambda\,e^{12}),
    \]
    and $3$-dimensional derived algebra $\n'=\langle e_5,e_6,e_7\rangle$;  it is isomorphic to the decomposable Lie algebra $\n_{6,3}\oplus\R$. 
    The choice $\lambda\in\Z\smallsetminus\{0\}$ ensures that the $2$-forms $\alpha_1$, $\alpha_2$ and $\alpha_3$ are integral with respect to the basis $\{e_1,e_2,e_3,e_4\}$ of $\R^4$. 
    Since also the $1$-forms $v_i,w_i$, $1\leq i \leq3$, and the  2-forms $F_0^1$ and $F_0^2$ are integral with respect to the same basis, the existence of the solution follows from Proposition \ref{prop:exconstr}.  
    Note that the bilinear form $\langle\cdot,\cdot\rangle_{\mathfrak{k}}$ on $\mathfrak{k}=\mathrm{Lie}(\mathbb{T}^2)$ 
    has signature $(2,0)$.

\subsubsection{The case $\n\cong\n_{7,3,A}$}

We finally show that there are no solutions when $\n\cong\n_{7,3,A}$, $\lambda\neq0$ and $k=2$.

\begin{prop}\label{prop:n73A}
    Every invariant solution $(\varphi,P,\theta,\langle \cdot,\cdot\rangle_{\mathfrak{k}})$ to the $\G_2$-system \eqref{eq:HetG2sys}  
    on a $2$-step nilmanifold $M=\Gamma\backslash N$,  
    where $P$ is a principal $\mathbb{T}^2$-bundle, $\n\cong\n_{7,3,A}$ and $\f$ calibrates $\n'$, has $\lambda=0$. 
\end{prop} 
\begin{proof}
Assume that there exists a solution to the system \eqref{71}--\eqref{78} with $\lambda\ne 0$. Since $\n\cong\n_{7,3,A}$, there exists a basis $\{f_0,f_1,f_2,f_3\}$ of $\r^*$ such that the structure forms $\alpha_i$ are given by $\alpha_i=f_0\w f_i$ for all $i\in\{1,2,3\}$. 
We decompose 
$$v_i=a_if_0+v'_i,\qquad w_i=b_if_0+w'_i,$$
with $a_i,b_i\in\R$ and $v'_i,w'_i\in \langle f_1,f_2,f_3\rangle$. 
Then the system \eqref{76}--\eqref{78} decouples into 
\begin{eqnarray}
&\label{76'} \epsilon_1\, ( a_1 v'_2-a_2v'_1)+\epsilon_2\, (b_1 w'_2-b_2w'_1)=2\lambda f_3\,,\\
&\label{77'} \epsilon_1\, ( a_3 v'_1-a_1v'_3)+\epsilon_2\, (b_3 w'_1-b_1w'_3)=2\lambda  f_2\,,\\
&\label{78'}  \epsilon_1\, ( a_2 v'_3-a_3v'_2)+\epsilon_2\, (b_2 w'_3-b_3w'_2)=2\lambda  f_1\,,
\end{eqnarray}
and 
 \begin{eqnarray}
&\label{76''} \epsilon_1\,  v'_1\w v'_2+\epsilon_2\, w'_1\w w'_2=0\,,\\
&\label{77''}  \epsilon_1\,  v'_3\w v'_1+\epsilon_2\,  w'_3\w w'_1=0\,,\\
&\label{78''}  \epsilon_1\,  v'_2\w v'_3+\epsilon_2\,  w'_2\w w'_3=0\,.
\end{eqnarray} 

Assume that $v'_1$ and $w'_1$ are linearly independent, spanning a plane $P\subset \r^*$. Then by \eqref{76''}, $v'_2$ and $w'_2$ belong to $P$, 
and by \eqref{77''} $v'_3$ and $w'_3$ belong to $P$ as well. Then by \eqref{76'}--\eqref{78'}, $f_1,f_2,f_3$ belong to $P$, a contradiction. 
Consequently, there exists a nonzero form $x_1\in\r^*$ and real numbers $t_1, s_1\in\R$ such that $v'_1=t_1x_1$ and $w'_1=s_1x_1$. 
A similar discussion shows the existence of nonzero forms $x_2, x_3\in\r^*$ and real numbers $t_2,t_3,s_2,s_3\in\R$ 
such that $v'_i=t_ix_i$ and $w'_i=s_ix_i$ for $i\in\{2,3\}$. 

Note that $x_1,x_2,x_3$ are linearly independent, as otherwise \eqref{76'}--\eqref{78'} would again contradict the linear independence of $f_1,f_2,f_3$.
Substituting the new expressions of the $v_i'$'s and $w_i'$'s in \eqref{76''}--\eqref{78''}, we get 
\begin{equation}\label{titj}
    \varepsilon_1t_it_j+\varepsilon_2s_is_j=0\,,\qquad\forall i\ne j\in\{1,2,3\}\,. 
\end{equation}
We also remark that $(t_i,s_i)\ne(0,0)$ for every $i\in\{1,2,3\}$. Indeed, otherwise we would get from \eqref{76'}--\eqref{78'} that $f_1,f_2,f_3$ belong to the plane generated by $x_j,x_k$, where $\{j,k\}\coloneqq\{1,2,3\}\smallsetminus\{i\}$. 

We claim that $s_it_i\ne 0$ for every $i\in\{1,2,3\}$. 
Assume, for instance, that $s_1=0$. Then we must have $t_1\ne 0$, so from \eqref{titj} we get $t_2=t_3=0$, and thus $s_2s_3=0$, which is impossible since $(t_2,s_2)\ne(0,0)\ne(t_3,s_3)$. 
This proves our claim for $i=1$; the proof in the remaining cases is analogous.  

Dividing each equation of \eqref{titj} by $t_it_j$ immediately shows that $\frac{s_i}{t_i}=\frac{s_j}{t_j}$ for every $i,j\in\{1,2,3\}$, 
so there exists a real number $c$ such that $s_i=ct_i$, and consequently $w'_i=cv'_i$ for every $i\in\{1,2,3\}$. 

 Using this, the system \eqref{76'}--\eqref{78'} becomes
\begin{eqnarray}
&\label{763} c_1 v'_2-c_2v'_1=2\lambda f_3\,,\\
&\label{773} c_3 v'_1-c_1v'_3=2\lambda  f_2\,,\\
&\label{783} c_2 v'_3-c_3v'_2=2\lambda  f_1\,,
\end{eqnarray}
where $c_i\coloneqq \varepsilon_1a_i+c\varepsilon_2b_i$ for $i\in\{1,2,3\}$. 
Multiplying \eqref{763} by $c_3$, \eqref{773} by $c_2$, \eqref{783} by $c_1$ and summing the results, gives $0=2\lambda(c_1f_1+c_2f_2+c_3f_3)$. 
Since we assumed $\lambda\neq 0$, by the linear independence of $f_1,f_2,f_3$ we must have $c_1=c_2=c_3=0,$ which by \eqref{763}--\eqref{783} again gives $f_1=f_2=f_3=0$, a contradiction. Thus, $\lambda=0$. 
\end{proof}

\begin{rmk}
    We notice that in fact we only used the equations \eqref{76}--\eqref{78} in order to prove Proposition \ref{prop:n73A}. 
\end{rmk}

\section{Open questions}\label{open}
In this final section, we collect some open problems related to our results: 
\begin{enumerate}[$\bullet$]
    \item {\bf Theorem \ref{main2}-ii}. 
    When $\n\cong\n_{5,2}\oplus\R^2$, we prove the existence of invariant solutions to the ${\rm G}_2$-system \eqref{eq:HetG2sys} with $\lambda \neq 0$ and $k=3$ (see Subsection \ref{sect:exn52}). 
    We suspect that the existence of invariant solutions to the system when $\dim (\n')=2$, $\lambda \neq 0$ and $k=3$ might force  
    the Lie algebra $\n$ to be isomorphic to $\n_{5,2}\oplus\R^2$.
    Moreover, the solution described in Subsection \ref{sect:exn52} depends on one integer parameter which determines the signature of 
    $\langle \cdot,\cdot\rangle_{\mathfrak{k}}$: it can be $(2,1)$ or $(3,0)$. 
    The existence of invariant solutions for which $\langle \cdot,\cdot\rangle_{\mathfrak{k}}$ has signature $(1,2)$ is an open problem. 
    \item {\bf Theorem \ref{main2}-iii}. The ${\rm G}_2$-system \eqref{eq:HetG2sys} admits invariant solutions on a $7$-dimensional $2$-step nilmanifold 
    $M=\Gamma\backslash N$ with $\lambda \neq 0$, $k=2$ and $\varphi$ calibrating $\n'$ if and only if $\n\cong\n_{6,3}\oplus\R$. 
    In Subsection \ref{EX5.12} we described an invariant solution such that $\langle \cdot,\cdot\rangle_{\mathfrak{k}}$ has signature $(2,0)$.  
    The existence of invariant solutions for which $\langle \cdot,\cdot\rangle_{\mathfrak{k}}$ has signature $(1,1)$  is an open problem.  
    \item {$\mathbf{ \dim(\n')=3}$}. In this case we only considered invariant solutions with $\f$ calibrating $\n'$. 
    It might be interesting to investigate whether solutions occur only in this case or whether there exist invariant solutions such that $\n'$ 
    is not calibrated by $\f$. 
    Moreover, the existence of invariant solutions with a principal $\mathbb{T}^k$-bundle is an open problem for  $k\geq 3$. 
\end{enumerate}

\appendix

\section{A construction of principal torus bundles over $2$-step nilmanifolds}\label{sect:Tkbdl}

The aim of this section is to show how to construct a principal torus bundle over a $2$-step nilmanifold  
starting with a suitable 2-form. We first introduce the setting, and then we state and prove the result. 

\smallskip

Let $N$ be a simply connected $n$-dimensional $2$-step nilpotent Lie group with Lie algebra $\n$. 
Consider a complementary subspace $\v$ of the derived algebra $\n'\coloneqq [\n,\n]$, so that $\n = \v\oplus \n'$.
Assume that $\n$ has integer structure constants with respect to a 
basis $\mathcal B=\{\ee_1,\ldots,\ee_{n-n'},z_1,\ldots,z_{n'}\}$ 
adapted to the splitting $\n = \v \oplus \n'$, where $n'=\dim(\n')$. Then 
\begin{equation}\label{eq:latticeB}
\Gamma\coloneqq\exp\left(\mathrm{span}_\Z(6\ee_1,\ldots,6\ee_{n-n'},z_1,\ldots,z_{n'})\right)
\end{equation}
is a cocompact lattice in $N$, and $M=\Gamma\backslash N$ is a $2$-step nilmanifold. 
We shall refer to $\Gamma$ as the cocompact lattice {\em determined} by the basis $\mathcal{B}$.

\begin{rmk}\label{rem:6}
Clearly, the cocompact lattice $\Gamma$ defined in \eqref{eq:latticeB} is not the only one contained in $N$ that 
one can construct starting from the basis $\mathcal{B}$. 
It is however the simplest one for which the construction discussed in this appendix works, as we will see in the proof of Lemma \ref{lem:F2}. 
\end{rmk}

\begin{defn}\label{integralForms}
An $\R^k$-valued $p$-form $F\in \Lambda^p \n^*\otimes\R^k$ is said to be {\em integral} with respect to the basis $\mathcal{B}$ of $\n$ 
if there exists a basis $\{t_1,\ldots,t_k\}$ of $\R^k$ 
such that $F = \sum_{r=1}^kF^r\otimes t_r$, where $F^r\in\Lambda^p\n^*$, and $F^r(x_1,\ldots,x_p)\in \mathbb Z$ for every 
$x_1,\ldots,x_p\in \mathcal B$ and $1\leq r\leq k$. 
\end{defn}

The main result of this appendix is the following. 

\begin{thm}\label{Theoapp1}
Let $\n$ be a $2$-step nilpotent Lie algebra having integer structure constants with respect to a basis $\mathcal{B}$. 
Let $N$ be the simply connected nilpotent Lie group with Lie algebra $\n$, and let $\Gamma\subseteq N$ be the cocompact lattice determined by $\mathcal{B}$. 
Then, for every integral closed $2$-form $F\in \Lambda^2\n^*\otimes \R^k$ there exists a principal $\mathbb{T}^k$-bundle over $M=\Gamma\backslash N$  
admitting a connection $1$-form $\theta$ with curvature $F_\theta=F$.  
\end{thm}

\medskip 
It is sufficient to prove Theorem \ref{Theoapp1} for principal $S^1$-bundles, namely for $F\in \Lambda^2\n^*$. 
In the following, we shall denote a $k$-form or a vector on the Lie algebra $\n$ and the corresponding left-invariant form  or left-invariant vector field on $N$ using the same symbol. 

\smallskip
Since $N$ is contractible, we have $F=d\sigma$ for some (not necessarily left-invariant) $1$-form $\sigma$ on $N$. 
Let $\gamma \in \Gamma$. Since 
$$
d\sigma=\gamma^*(d\sigma)=d(\gamma^*(\sigma)),
$$
we have $d(\gamma^*(\sigma)-\sigma)=0$. Hence there exists a function $f_{\gamma}\in C^{\infty}(N)$ such that 
\begin{equation}\label{eq:fgamma}
df_{\gamma}=\gamma^*(\sigma)-\sigma.
\end{equation}
Moreover, we can take $f_{1_N}\equiv0$. 

For $\gamma_1,\gamma_2\in \Gamma$, we have both 
$$
(\gamma_1\gamma_2)^*(\sigma)=\sigma+df_{\gamma_1\gamma_2}
$$
and 
$$
(\gamma_1\gamma_2)^*(\sigma)=(\gamma_2)^*(\sigma+df_{\gamma_1})=\sigma+df_{\gamma_2}+d\gamma_2^*f_{\gamma_1}\,,
$$
whence it follows that  
$$
df_{\gamma_2}+d\gamma_2^*f_{\gamma_1}-df_{\gamma_1\gamma_2}=0.
$$
Consequently,  
\begin{equation}\label{eq:idfgammas}
f_{\gamma_2}+\gamma_2^*f_{\gamma_1}-f_{\gamma_1\gamma_2} = c_{\gamma_1 \gamma_2} \in \R
\end{equation}
is constant. 

In summary: for every $\gamma\in\Gamma$, the potential $\sigma\in\Omega^1(N)$ gives rise to a function $f_\gamma$ through the identity \eqref{eq:fgamma}. Moreover, the functions corresponding to different elements of $\Gamma$ are related by the identity \eqref{eq:idfgammas}. 

\medskip 
Next consider $N\times S^1$. For $\gamma\in \Gamma$ let $\tilde\gamma\colon N\times S^1\to N\times S^1$ be the map 
\begin{equation}\label{GammaAct}
\tilde\gamma(x,{\rm e}^{2\pi it})\coloneqq  \left(\gamma x,{\rm e}^{2\pi i(t-f_\gamma(x))}\right)\,.
\end{equation}

\begin{lemma}\label{lem:cint}
    If $c_{\gamma_1 \gamma_2}$ is an integer for all $\gamma_1,\gamma_2\in\Gamma$, then \eqref{GammaAct} defines an action of $\Gamma$ on $N\times S^1$. 
\end{lemma}
\begin{proof}
Let $\gamma_1,\gamma_2\in\Gamma$. 
We have 
$$
\widetilde{\gamma_1\gamma_2} (x,{\rm e}^{2\pi it})=\left(\gamma_1\gamma_2x,{\rm e}^{2\pi i(t-f_{\gamma_1\gamma_2}(x))}\right)=
\left(\gamma_1\gamma_2x,{\rm e}^{2\pi i(t-f_{\gamma_2}(x)-f_{\gamma_1}(\gamma_2x)+c_{\gamma_1 \gamma_2})}\right)
$$   
and 
$$
\tilde\gamma_1\left(\tilde\gamma_2 (x,{\rm e}^{2\pi it})\right)  = \tilde\gamma_1\left(\gamma_2 x,{\rm e}^{2\pi i(t-f_{\gamma_2}(x))}\right)=
\left(\gamma_1\gamma_2 x,{\rm e}^{2\pi i(t-f_{\gamma_2}(x)-f_{\gamma_1}(\gamma_2x))}\right),
$$
whence it follows that
$$
\widetilde{\gamma_1\gamma_2} (x,{\rm e}^{2\pi it})=\tilde\gamma_1\left(\tilde\gamma_2 (x,{\rm e}^{2\pi it})\right) \,\,\iff   \,\,c_{\gamma_1 \gamma_2}\in \mathbb Z\,. 
$$ 
\end{proof}

\begin{prop}
If $c_{\gamma_1 \gamma_2}$ is an integer for all $\gamma_1,\gamma_2\in\Gamma$, then there exists a principal $S^1$-bundle $\Gamma\backslash (N\times S^1)\to\Gamma\backslash N$ endowed with a connection form $\theta$ with curvature $d\theta = F$. 
\end{prop}
\begin{proof}
    We know that $\Gamma$ acts on $N\times S^1$ as in \eqref{GammaAct} and it is clear that this action commutes with the action of $S^1$ on itself by left multiplication. Thus, $\Gamma\backslash (N\times S^1)$ is the total space of a principal $S^1$ bundle over $\Gamma \backslash N$. 
    The 1-form $\theta = \sigma + dt$ on $N\times S^1$, with $\sigma\in \Omega^1(N)$ such that $F=d\sigma$, is $\Gamma$-invariant. Indeed, for all $\gamma\in \Gamma$
\[
\gamma^*(\theta)=\gamma^*(\sigma) + d\gamma^*(t)=\sigma +df_{\gamma} +d(t-f_{\gamma})=dt+\sigma=\theta\,. 
\]
This shows that $\theta$ defines a connection 1-form on $\Gamma\backslash (N\times S^1)$ with curvature $d\theta = d\sigma = F$. 
\end{proof}

We now show that the assumption of the previous result holds when the closed $2$-form $F$ is integral with respect to the basis 
$\mathcal{B} = \{e_1,\ldots,e_{n-n'},z_1,\ldots,z_{n'}\}$ of $\n=\v\oplus \n'$. 
This will prove Theorem \ref{Theoapp1} for $S^1$ bundles. 

We denote by $\{e^1,\ldots,e^{n-n'},z^1,\ldots,z^{n'}\}$ the dual basis of $\mathcal{B}$,
so that $\{\zz^1,\ldots,\zz^{n'}\}$ is a basis of $(\n')^*$ and $\{\ee^1,\ldots,\ee^{n-n'}\}$ is a basis of $\v^*$. Since $\n$ is $2$-step nilpotent, we have $d\zz^r\in\Lambda^2\v^*$ and $d\ee^i=0$.

\smallskip

Let us consider the decomposition $\Lambda^2\n^*\cong \Lambda^2(\n')^*\oplus \Lambda^2\v^*\oplus (\v^*\otimes (\n')^*)$. By the above remark we get $d(\Lambda^2\v^*)=0$, $d(\v^*\otimes (\n')^*)\subset \Lambda^3\v^*$ and $d(\Lambda^2(\n')^*)\subset \Lambda^2\v^*\otimes (\n')^*$. Moreover, since $dz^1,\ldots,dz^{n'}$ are linearly independent, it is easy to see that the restriction of $d$ to 
$\Lambda^2(\n')$ is injective. Consequently, every closed 2-form is contained in $\Lambda^2\v^*\oplus \v^*\otimes (\n')^*$. 
 We can then write an integral closed 2-form $F$ as follows 
\[
F = F_1 + F_2,
\]
where 
\[
F_1 = \sum_{1\leq r < s \leq n-n'} F_{rs}\ee^r\w\ee^s \in \Lambda^2\v^*,\qquad 
F_2 = \sum_{1\leq i \leq n-n'} \sum_{1\leq r \leq n'} \tilde{F}_{ir}\ee^i\w \zz^r 
\in \v^*\otimes (\n')^*, 
\]
and $F_{rs},\tilde{F}_{ir} \in \Z$.

We claim that the left-invariant 2-form $F$ has a potential $\sigma\in\Omega^1(N)$ giving rise to functions $f_\gamma$ for which the constants 
$c_{\gamma_1\gamma_2}$ given by \eqref{eq:idfgammas} are integers for all $\gamma_1,\gamma_2\in\Gamma$. 
By linearity, it is sufficient to prove this for $F_1$ and $F_2$ separately. 
This will be done in Lemma \ref{lem:F1} and Lemma \ref{lem:F2}, respectively, 
and will complete the proof of Theorem \ref{Theoapp1} for $S^1$ bundles.  

We begin with a preliminary result. 
\begin{lemma}\label{lem:falpha}
Let $\beta\in\n^*$ and define the scalar function $u_\beta\in C^\infty(N)$ via the identity 
\[
u_\beta(\mathrm{e}^X) = \beta(X),
\] 
for all $X\in \n$. 
Then $d\beta = \frac12\sum_{1\leq i,j\leq n-n'} b_{ij}\ee^i\w\ee^j = \sum_{1\leq i,j\leq n-n'} b_{ij}\ee^i\otimes\ee^j$, 
for certain $b_{ij} = -b_{ji}\in\R$, 
and 
\[
\beta = du_\beta +\frac12\sum_{1\leq i,j\leq n-n'} b_{ij}u_{\ee^i}\ee^j.
\]
In particular, if $\beta\in\v^*$, so that $d\beta=0$, then 
\[
\beta = du_\beta.
\] 
\end{lemma}
\begin{proof}
Since $\n$ is 2-step nilpotent, for every $A,B\in \n$ the following identity holds 
$$
{\rm e}^{A}{\rm e}^{B}={\rm e}^{A+B+\tfrac12 [A,B]}\,,
$$
by the Baker-Campbell-Hausdorff formula.

Let $Y$ be a left-invariant vector field on $N$ and $p={\rm e}^Z\in N$. Then, we have 
\[
\begin{split}
Y_{p}(u_{\beta}) &= \left.\frac{d}{dt}\right|_{t=0}u_{\beta}\left(p\,{\rm e}^{tY}\right)
=\left.\frac{d}{dt}\right|_{t=0}u_{\beta}\left({\rm e}^{Z+tY+\frac12 t[Z,Y]}\right) = 
\left.\frac{d}{dt}\right|_{t=0}\, \beta\left(Z+tY+\frac12 t[Z,Y]\right) \\ 
&= \beta(Y)+\frac12 \beta([Z,Y]) =  \beta(Y)-\frac12 d\beta(Z,Y)  =  \beta(Y)-\frac12\sum_{1\leq i,j\leq n-n'} b_{ij}\ee^i(Z)\ee^j(Y)\\
&= \beta(Y)-\frac12\sum_{1\leq i,j\leq n-n'} b_{ij}u_{\ee^i}(p)\ee^j(Y) = \left(\beta-\frac12\sum_{1\leq i,j\leq n-n'} b_{ij}u_{\ee^i}\ee^j\right)_p(Y).
\end{split}
\]
The thesis then follows. 
\end{proof}

\begin{lemma}\label{lem:F1}
    The left-invariant integral closed $2$-form $F_1$ has the following potential 
    \[
    \sigma_1 = \sum_{1\leq r < s \leq n-n'} F_{rs} u_{\ee^r}\,\ee^s \in\Omega^1(N). 
    \]
    It gives rise to functions $f_\gamma$ for which $c_{\gamma_1\gamma_2}\in\Z$, for all $\gamma_1,\gamma_2\in\Gamma$.
\end{lemma}
\begin{proof}
    By Lemma \ref{lem:falpha} we have $\ee^r=du_{\ee^r}$, so we can write
    \[
    F_1 = \sum_{1\leq r < s \leq n-n'} F_{rs}\ee^r\w\ee^s = 
    \sum_{1\leq r < s \leq n-n'} F_{rs} d(u_{\ee^r}\,\ee^s) = d\sigma_1,
    \]
    where
    \[
    \sigma_1 \coloneqq \sum_{1\leq r < s \leq n-n'} F_{rs} u_{\ee^r}\,\ee^s.
    \]
    Let us focus on the 1-form $\sigma_{rs}\coloneqq u_{\ee^r}\,\ee^s$ and let us consider $\gamma = {\rm e}^C\in\Gamma$. We claim that
    \begin{equation}\label{eq:gammaubeta}
    \gamma^*u_{\ee^r}=u_{\ee^r}+\ee^r(C).
    \end{equation}
    Indeed, for all $p={\rm e}^Y\in N$, we have
    \[
    \gamma^*u_{\ee^r}(p) = u_{\ee^r}(\gamma p) = u_{\ee^r}({\rm e}^C{\rm e}^Y) = u_{\ee^r}\left({\rm e}^{C+Y+\tfrac12[C,Y]}\right) 
    = \ee^r(C) + \ee^r(Y)
    = u_{\ee^r}(p) + \ee^r(C),
    \]
    since $\ee^r|_{\n'}=0$. 
    We now compute
    \[
    \gamma^*\sigma_{rs}= \gamma^*(u_{\ee^r}\,\ee^s)  = (\gamma^*u_{\ee^r})\,\ee^s  = 
    (u_{\ee^r}+\ee^r(C))\,\ee^s  = 
    \sigma_{rs} + \ee^r(C)\ee^s = \sigma_{rs} + \ee^r(C)du_{\ee^s}\,.
    \]
    Using this, we obtain
    \[
    \gamma^*\sigma_1-\sigma_1 = \sum_{1\leq r < s \leq n-n'} F_{rs}\ee^r(C)du_{\ee^s} = df_\gamma,
    \]
    where
    \[
    f_\gamma \coloneqq \sum_{1\leq r < s \leq n-n'} F_{rs}\ee^r(C)u_{\ee^s}\,.
    \]
    For every $\gamma_1={\rm e}^{C_1},\gamma_2={\rm e}^{C_2}\in\Gamma$, we have $\gamma_1\gamma_2={\rm e}^{C_1+C_2+\frac12[C_1,C_2]}$, 
    and we then obtain 
    \[
    \begin{split}
    c_{\gamma_1\gamma_2} &= f_{\gamma_2}+\gamma_2^*f_{\gamma_1}-f_{\gamma_1\gamma_2} = 
    \sum_{1\leq r < s \leq n-n'} F_{rs}\left(\ee^r(C_2)u_{\ee^s} + \ee^r(C_1)\gamma_2^*u_{\ee^s}-\ee^r(C_1+C_2)u_{\ee^s}\right)\\
    &= \sum_{1\leq r < s \leq n-n'} F_{rs}\ee^r(C_1)\ee^s(C_2).
    \end{split}
    \]
    Since $F_{rs}\in\Z$ and $\ee^r(C)\in\Z$ for all $C\in \mathrm{span}_\Z(6\ee_1,\ldots,6\ee_{n-n'},z_1,\ldots,z_{n'})$, the thesis follows. 
\end{proof}

We now focus on the summand $F_2\in \v^*\otimes(\n')^*$. First, we rewrite it as follows 
\[
F_2 = \sum_{1\leq i \leq n-n'} \sum_{1\leq r \leq n'} \tilde{F}_{ir}\ee^i\w \zz^r 
= \sum_{1\leq i \leq n-n'} \ee^i \w \eta^i,
\]
where 
\[
\eta^i \coloneqq \sum_{1\leq r \leq n'} \tilde{F}_{ir} \zz^r \in (\n')^*. 
\]
For $1\leq i \leq n-n'$, we let
\begin{equation}\label{eq:alphai}
\alpha^i \coloneqq d\eta^i = \frac12 \sum_{1\leq j,k\leq n-n'} c_{ijk}\ee^j\w\ee^k = \sum_{1\leq j,k\leq n-n'} c_{ijk}\ee^j\otimes\ee^k,    
\end{equation}
where the components $c_{ijk}\in\Z$ are skew symmetric in the last two indices: $c_{ijk}=-c_{ikj}$. 

\begin{rmk}\label{rem:cijk}
Since $d\zz^r = - \frac12 \sum_{1\leq j,k\leq n-n'} c^r_{jk}\ee^j\wedge\ee^k$, where $c^r_{jk}\in\Z$ are the structure constants of $\n$ 
with respect to the basis $\mathcal{B}$, we have
\[
c_{ijk} = -\sum_{r=1}^{n'} \tilde{F}_{ir}c^r_{jk}.
\]    
\end{rmk}

The closure of $F_2$ gives 
\[
0 = dF_2 = - \sum_{1\leq i \leq n-n'} \ee^i \w \alpha^i
= -\frac12 \sum_{1\leq i,j,k \leq n-n'} c_{ijk}\ee^i\w \ee^j\w\ee^k,
\]
hence it is equivalent to the cocycle condition 
\begin{equation}\label{eq:cocycle}
0 = c_{[ijk]} = c_{ijk}+c_{jki}+c_{kij}, 
\end{equation}
for all $1\leq i,j,k\leq n-n'$. 

The next lemma will be useful in the proof of the last result. 
\begin{lemma}\label{lem:gammaeta}
    Let $\gamma=e^C\in\Gamma$, then for all $1\leq i\leq n-n'$ the following identity holds
    \[
    \gamma^*u_{\eta^i} = u_{\eta^i} + \eta^i(C) -\frac12 \sum_{j,k}c_{ijk}\ee^j(C)u_{\ee^k}. 
    \]
\end{lemma}
\begin{proof}
Let $p=e^Y\in N$, then 
\[
\begin{split}
\gamma^*u_{\eta^i}(p)   &= u_{\eta^i}(\gamma p) = u_{\eta^i}(e^Ce^Y) = u_{\eta^i}\left(e^{C+Y+\frac12[C,Y]}\right) 
                            = \eta^i(C) + \eta^i(Y) +\frac12\eta^i([C,Y])      \\
                        &= u_{\eta^i}(p) + \eta^i(C) -\frac12 d\eta^i(C,Y) 
                            = u_{\eta^i}(p) + \eta^i(C) -\frac12 \sum_{1\leq j,k\leq n-n'} c_{ijk}\ee^j\otimes\ee^k(C,Y)\\
                        &= u_{\eta^i}(p) + \eta^i(C) -\frac12 \sum_{1\leq j,k\leq n-n'} c_{ijk}\ee^j(C) u_{\ee^k}(p). 
\end{split}
\]
\end{proof}

\begin{lemma}\label{lem:F2}
    The left-invariant integral closed $2$-form $F_2$ has the following potential \[
    \sigma_2 = \sum_{i} u_{\ee^i}\eta^i-\frac13\sum_{i,j,k} c_{ijk}u_{\ee^i}u_{\ee^j}\ee^k.
    \]
    It gives rise to functions $f_\gamma$ for which $c_{\gamma_1\gamma_2}\in\Z$, for all $\gamma_1,\gamma_2\in\Gamma$.
\end{lemma}
\begin{proof}
    Using the definition of $\alpha^i$ and the cocycle condition \eqref{eq:cocycle}, we obtain
    \[
    \sum_{i}u_{\ee^i}\alpha^i = \frac12\sum_{i,j,k} c_{ijk}u_{\ee^i}\ee^j\w\ee^k = \frac12\sum_{i,j,k} (c_{jik}+c_{kji})u_{\ee^i}\ee^j\w\ee^k = \sum_{i,j,k} c_{ijk}u_{\ee^j}\ee^i\w\ee^k.
    \]
    This identity and equation \eqref{eq:alphai} imply that
    \[
    d\left(\sum_{i,j,k} c_{ijk}u_{\ee^i}u_{\ee^j}\ee^k\right) = \sum_{i,j,k} c_{ijk}(u_{\ee^j} \ee^i\w\ee^k + u_{\ee^i} \ee^j\w\ee^k ) = 3\sum_{i}u_{\ee^i}\alpha^i.
    \]
    Therefore, since $\alpha_i=d\eta^i$, we have
    \[
    d\sigma_2 = \sum_{i} \ee^i\w\eta^i + \sum_{i} u_{\ee^i} d\eta^i-\sum_{i}u_{\ee^i}\alpha^i = F_2,
    \]
    which shows the first assertion. 

\smallskip

Let $\gamma=e^C\in\Gamma$. Then, using \eqref{eq:gammaubeta} we obtain
\[
\begin{split}
    \gamma^*\sigma_2 - \sigma_2 &= \sum_{i} (u_{\ee^i}+\ee^i(C))\eta^i
                                    -\frac13\sum_{i,j,k} c_{ijk}(u_{\ee^i}+\ee^i(C))(u_{\ee^j}+\ee^j(C))\ee^k -\sigma_2\\
                    &=  \sum_{i} \ee^i(C)\eta^i-\frac13\sum_{i,j,k} c_{ijk}(u_{\ee^i}\ee^j(C)+u_{\ee^j}\ee^i(C)+\ee^i(C)\ee^j(C))\ee^k.
\end{split}
\]
We claim that $\gamma^*\sigma_2 - \sigma_2 = df_\gamma$, where 
\[
f_\gamma \coloneqq  \sum_{i} \ee^i(C)u_{\eta^i} -\frac16\sum_{i,j,k} c_{ijk}\ee^j(C)u_{\ee^i}u_{\ee^k} 
                    -\frac13\sum_{i,j,k} c_{ijk}\ee^i(C)\ee^j(C)u_{\ee^k}. 
\]
First, applying Lemma \ref{lem:falpha} to $\eta^i\in(\n')^*$ and using \eqref{eq:alphai}, we obtain
\[
du_{\eta^i} = \eta^i - \frac12 \sum_{1\leq j,k\leq n-n'} c_{ijk}u_{\ee^j}\ee^k.
\]
Using this identity, we get
\[
d\left(\sum_{i} \ee^i(C)u_{\eta^i}\right) = \sum_{i} \ee^i(C)\eta^i - \frac12 \sum_{i,j,k} c_{ijk} \ee^i(C) u_{\ee^j}\ee^k.
\]
Moreover, recalling that $du_{\ee^i} = \ee^i$ and using the cocycle condition $c_{ijk}+c_{jki}+c_{kij}=0$, we have
\[
\begin{split}
d\left(\sum_{i,j,k} c_{ijk}\ee^j(C)u_{\ee^i}u_{\ee^k} \right) &= \sum_{i,j,k} c_{ijk}\ee^j(C)(u_{\ee^k}\ee^i+u_{\ee^i}\ee^k)\\
                                                                    &= \sum_{i,j,k} c_{kji}\ee^j(C)u_{\ee^i}\ee^k 
                                                                        + \sum_{i,j,k}c_{ijk}\ee^j(C)u_{\ee^i}\ee^k\\
                                                                    &= \sum_{i,j,k} (-c_{jik}-c_{ikj})\ee^j(C)u_{\ee^i}\ee^k 
                                                                        + \sum_{i,j,k}c_{ijk}\ee^j(C)u_{\ee^i}\ee^k\\
                                                                    &= -\sum_{i,j,k} c_{ijk}\ee^i(C)u_{\ee^j}\ee^k 
                                                                        + 2\sum_{i,j,k}c_{ijk}\ee^j(C)u_{\ee^i}\ee^k.
\end{split}
\]
An easy computation using the expression of $f_\gamma$ and the last two identities proves the claim: $df_\gamma = \gamma^*\sigma_2-\sigma_2$. 

\smallskip

Now, consider $\gamma_1={\rm e}^{C_1},\gamma_2={\rm e}^{C_2}\in\Gamma$.  
Using equation \eqref{eq:gammaubeta} and Lemma \ref{lem:gammaeta}, we obtain 
\[
\begin{split}
\gamma_2^*f_{\gamma_1} &= \sum_{i} \ee^i(C_1)\gamma_2^*u_{\eta^i} -\frac16\sum_{i,j,k} c_{ijk}\ee^j(C_1)\gamma_2^*u_{\ee^i}\gamma_2^*u_{\ee^k} 
                            -\frac13\sum_{i,j,k} c_{ijk}\ee^i(C_1)\ee^j(C_1)\gamma_2^*u_{\ee^k}\\
                    &= \sum_{i} \ee^i(C_1)u_{\eta^i} + \sum_{i} \ee^i(C_1)\eta^i(C_2)-\frac12 \sum_{i,j,k} c_{ijk} \ee^i(C_1)\ee^j(C_2)u_{\ee^k}\\
                    &\quad    -\frac16\sum_{i,j,k} c_{ijk}\ee^j(C_1)(u_{\ee^i}+\ee^i(C_2))(u_{\ee^k}+\ee^k(C_2))
                        -\frac13\sum_{i,j,k} c_{ijk}\ee^i(C_1)\ee^j(C_1)(u_{\ee^k}+\ee^k(C_2))\\
                    &= \sum_{i} \ee^i(C_1)u_{\eta^i} + \sum_{i} \ee^i(C_1)\eta^i(C_2)-\frac12 \sum_{i,j,k} c_{ijk} \ee^i(C_1)\ee^j(C_2)u_{\ee^k}\\
                    &\quad    -\frac16\sum_{i,j,k} c_{ijk}\ee^j(C_1)u_{\ee^i}u_{\ee^k}
                                -\frac16\sum_{i,j,k} c_{ijk}\ee^j(C_1)\ee^k(C_2)u_{\ee^i}
                                -\frac16\sum_{i,j,k} c_{ijk}\ee^j(C_1)\ee^i(C_2)u_{\ee^k}\\
                    &\quad            -\frac16\sum_{i,j,k} c_{ijk}\ee^j(C_1)\ee^i(C_2)\ee^k(C_2)
                        -\frac13\sum_{i,j,k} c_{ijk}\ee^i(C_1)\ee^j(C_1)u_{\ee^k}
                        -\frac13\sum_{i,j,k} c_{ijk}\ee^i(C_1)\ee^j(C_1)\ee^k(C_2).
\end{split}
\]
Moreover, since $\gamma_1\gamma_2={\rm e}^{C_1+C_2+\frac12[C_1,C_2]}$ and $\ee^i|_{\n'}=0$, we obtain
\[
\begin{split}
f_{\gamma_1\gamma_2} &= \sum_{i} \ee^i(C_1+C_2)u_{\eta^i} -\frac16\sum_{i,j,k} c_{ijk}\ee^j(C_1+C_2)u_{\ee^i}u_{\ee^k} 
                           -\frac13\sum_{i,j,k} c_{ijk}\ee^i(C_1+C_2)\ee^j(C_1+C_2)u_{\ee^k}\\
                    &= \sum_{i} \ee^i(C_1)u_{\eta^i} + \sum_{i} \ee^i(C_2)u_{\eta^i}  -\frac16\sum_{i,j,k} c_{ijk}\ee^j(C_1)u_{\ee^i}u_{\ee^k} 
                         -\frac16\sum_{i,j,k} c_{ijk}\ee^j(C_2)u_{\ee^i}u_{\ee^k}\\
                    &\quad -\frac13\sum_{i,j,k} c_{ijk}\left(\ee^i(C_1)\ee^j(C_1)
                            +\ee^i(C_1)\ee^j(C_2)+\ee^i(C_2)\ee^j(C_1)+\ee^i(C_2)\ee^j(C_2)\right)u_{\ee^k}.
\end{split}
\]
Also, by definition
\[
f_{\gamma_2} = \sum_{i} \ee^i(C_2)u_{\eta^i} -\frac16\sum_{i,j,k} c_{ijk}\ee^j(C_2)u_{\ee^i}u_{\ee^k} 
                    -\frac13\sum_{i,j,k} c_{ijk}\ee^i(C_2)\ee^j(C_2)u_{\ee^k}
\]
Combining all these expressions, we finally get
\[
\begin{split}
c_{\gamma_1\gamma_2} &= f_{\gamma_2}+\gamma_2^*f_{\gamma_1}-f_{\gamma_1\gamma_2} \\
                        &= \sum_{i} \ee^i(C_1)\eta^i(C_2) -\frac16\sum_{i,j,k} c_{ijk}\ee^j(C_1)\ee^i(C_2)\ee^k(C_2)
                        -\frac13\sum_{i,j,k} c_{ijk}\ee^i(C_1)\ee^j(C_1)\ee^k(C_2)\\
                        &\quad-\frac12 \sum_{i,j,k} c_{ijk} \ee^i(C_1)\ee^j(C_2)u_{\ee^k} 
                        -\frac16\sum_{i,j,k} c_{ijk}\ee^j(C_1)\ee^k(C_2)u_{\ee^i}  -\frac16\sum_{i,j,k} c_{ijk}\ee^j(C_1)\ee^i(C_2)u_{\ee^k} \\
                         &\quad +\frac13\sum_{i,j,k} c_{ijk}\ee^i(C_1)\ee^j(C_2)u_{\ee^k}
                            +\frac13\sum_{i,j,k} c_{ijk}\ee^i(C_2)\ee^j(C_1)u_{\ee^k}\\ 
                        &= \sum_{i} \ee^i(C_1)\eta^i(C_2) -\frac16\sum_{i,j,k} c_{ijk}\ee^j(C_1)\ee^i(C_2)\ee^k(C_2)
                        -\frac13\sum_{i,j,k} c_{ijk}\ee^i(C_1)\ee^j(C_1)\ee^k(C_2)\\
                        &\quad-\frac16 \sum_{i,j,k} c_{ijk} \ee^i(C_1)\ee^j(C_2)u_{\ee^k} 
                        -\frac16\sum_{i,j,k} c_{ijk}\ee^j(C_1)\ee^k(C_2)u_{\ee^i}  +\frac16\sum_{i,j,k} c_{ijk}\ee^j(C_1)\ee^i(C_2)u_{\ee^k}\\
                        &= \sum_{i} \ee^i(C_1)\eta^i(C_2) -\frac16\sum_{i,j,k} c_{ijk}\ee^j(C_1)\ee^i(C_2)\ee^k(C_2)
                        -\frac13\sum_{i,j,k} c_{ijk}\ee^i(C_1)\ee^j(C_1)\ee^k(C_2)\\
                        &\quad-\frac16 \sum_{i,j,k} c_{ijk} \ee^i(C_1)\ee^j(C_2)u_{\ee^k} 
                        -\frac16\sum_{i,j,k} c_{kij}\ee^i(C_1)\ee^j(C_2)u_{\ee^k} -\frac16\sum_{i,j,k} c_{jki}\ee^i(C_1)\ee^j(C_2)u_{\ee^k}\\
                        &= \sum_{i} \ee^i(C_1)\eta^i(C_2) -\frac16\sum_{i,j,k} c_{ijk}\ee^j(C_1)\ee^i(C_2)\ee^k(C_2)
                        -\frac13\sum_{i,j,k} c_{ijk}\ee^i(C_1)\ee^j(C_1)\ee^k(C_2)\\
                        &\quad
                        -\frac16 \sum_{i,j,k} (c_{ijk} + c_{kij}+c_{jki} ) \ee^i(C_1)\ee^j(C_2)u_{\ee^k} \\
                        &= \sum_{i} \ee^i(C_1)\eta^i(C_2) -\frac16\sum_{i,j,k} c_{ijk}\ee^j(C_1)\ee^i(C_2)\ee^k(C_2)
                        -\frac13\sum_{i,j,k} c_{ijk}\ee^i(C_1)\ee^j(C_1)\ee^k(C_2). 
\end{split}
\]
Since $c_{ijk}\in\Z$ and $C_1,C_2\in\mathrm{span}_\Z(6\ee_1,\ldots,6\ee_{n-n'},z_1,\ldots,z_{n'})$, the thesis follows.
\end{proof}

\section{The classification of $7$-dimensional $2$-step nilpotent Lie algebras}\label{2stepnilclass} 

In this appendix, we recall the classification of real 7-dimensional 2-step nilpotent Lie algebras from \cite{Gon}. 
For each Lie algebra $\n$, the structure equations are written with respect to a basis $\{e^1,\ldots,e^7\}$ of the dual Lie algebra $\n^*$. 

The notation $\n_{n,n'}$ or $\n_{n,n',\bullet}$ means that the Lie algebra has dimension $n$ and derived algebra of dimension $n'$, 
while different capital letters in the third argument are used to distinguish non-isomorphic Lie algebras whose derived algebras have the same dimension. 
Moreover, $\mathfrak{h}_n$ denotes the Heisenberg Lie algebra of dimension $n$ (for odd $n$), and $\mathfrak{h}_3^{\C}$ denotes the real Lie algebra underlying the complex 3-dimensional Heisenberg Lie algebra.

\begin{itemize}
\item 7-dimensional 2-step nilpotent Lie algebras $\n$ with $\dim(\n')=1$:
\begin{eqnarray*}
\mathfrak{h}_3\oplus\R^4 			&=& \left(0,0,0,0,0,0,e^{12}\right),\\
\mathfrak{h}_5\oplus\R^2 			&=& \left(0,0,0,0,0,0,e^{12}+e^{34}\right),\\
\mathfrak{h}_7					&=& \left(0,0,0,0,0,0,e^{12}+e^{34}+e^{56}\right).
\end{eqnarray*}

\item  7-dimensional 2-step nilpotent Lie algebras $\n$ with $\dim(\n')=2$:
\begin{eqnarray*}
\n_{5,2}\oplus\R^2 		                       &=& \left(0,0,0,0,e^{12},e^{13},0\right),\\
\mathfrak{h}_3\oplus\mathfrak{h}_3\oplus\R 	   &=& \left(0,0,0,0,e^{12},e^{34},0\right),\\
\mathfrak{h}_3^{\C}\oplus\R 		           &=& \left(0,0,0,0,e^{13}-e^{24},e^{14}+e^{23},0\right),\\
\n_{6,2}\oplus\R		                       &=& \left(0,0,0,0,e^{12},e^{14}+e^{23},0\right),\\
\n_{7,2,A}				                       &=& \left(0,0,0,0,0,e^{12},e^{14}+e^{35}\right),\\
\n_{7,2,B}				                       &=& \left(0,0,0,0,0,e^{12}+e^{34},e^{15}+e^{23}\right).
\end{eqnarray*}

\item 7-dimensional 2-step nilpotent Lie algebras $\n$ with $\dim(\n')=3$:
\begin{eqnarray*}
\n_{6,3}\oplus\R	&=&\left(0,0,0,0,e^{12},e^{13},e^{23}\right),\\
\n_{7,3,A}		&=& \left(0,0,0,0,e^{12},e^{23},e^{24}\right),\\
\n_{7,3,B}		&=& \left(0,0,0,0,e^{12},e^{23},e^{34}\right),\\
\n_{7,3,B_1}		&=& \left(0,0,0,0,e^{12}-e^{34},e^{13}+e^{24},e^{14}\right)\\
\n_{7,3,C}		&=& \left(0,0,0,0,e^{12}+e^{34},e^{23},e^{24}\right),\\
\n_{7,3,D}		&=& \left(0,0,0,0,e^{12}+e^{34},e^{13},e^{24}\right),\\
\n_{7,3,D_1}		&=& \left(0,0,0,0,e^{12}-e^{34},e^{13}+e^{24},e^{14}-e^{23}\right). 
\end{eqnarray*}
\end{itemize}


\begin{thebibliography}{[99]}

\bibitem{BFF}
{\sc L.~Bagaglini, M.~Fern\'andez, A.~Fino}. 
\newblock Coclosed G$_2$-structures inducing nilsolitons. 
\newblock {\em Forum Math.}, {\bf30}, 109--128, 2018.

\bibitem{Bry}
{\sc R.~L.~Bryant}.
\newblock Some remarks on {G$_2$}-structures.
\newblock In {\em Proceedings of {G}{\"o}kova {G}eometry-{T}opology {C}onference 2005}, pages 75--109. 
G{\"o}kova Geometry/Topology Conference (GGT), G{\"o}kova, 2006.

\bibitem{ChSw} 
{\sc S.~G. Chiossi, A.~Swann}.
\newblock G{$_2$}-structures with torsion from half-integrable nilmanifolds.
\newblock {\em J.~Geom.~Phys.} {\bf 54}, 262--285, 2005.

\bibitem{CFT} 
{\sc A. Clarke, M. Garcia-Fern\'andez, C. Tipler}.
\newblock T-Dual solutions and infinitesimal moduli of the $G_2$-Strominger system.
\newblock {\em Adv. Theor. Math. Phys.} {\bf 26}, 1669--1704, 2022.  

\bibitem{dSGFLSE}
{\sc A.A. da Silva Jr., M. Garcia-Fern\'andez, J.D. Lotay, H.N. S\'a Earp}. 
\newblock Coupled G$_2$-instantons. 
\newblock {\em Internat. J. Math.} \href{https://doi.org/10.1142/S0129167X25420029}{doi:10.1142/S0129167X25420029}, 2025. 

\bibitem{dG}
{\sc X. de la Ossa and M. Galdeano}, Families of solutions of the heterotic $\G_2$ system.
 \href{https://arxiv.org/abs/2111.13221}{\tt arXiv:2111.13221}.



\bibitem{X}
{\sc X. de La Ossa, M. Larfors, M. Magill and E. E. Svanes.} 
Superpotential of three dimensional $\mathcal{N} = 1$ heterotic supergravity. {\em Journal of High Energy Physics} {\bf 195}, 2020.

\bibitem{X1}
{\sc X. de La Ossa, M. Larfors, M. Magill and E. E. Svanes.} 
Quantum aspects of heterotic $\G_2$ systems. {\em 2024 MATRIX annals}, Part II, 253--293.

\bibitem{DLS}
{\sc X. de La Ossa, M. Larfors, and E. E. Svanes.} 
The Infinitesimal Moduli Space of Heterotic $\G_2$ Systems. {\em Commun. Math. Phys.} {\bf360} 727, 2018.


\bibitem{X2}
{\sc X. de La Ossa, M. Larfors, and E. E. Svanes.} 
Restrictions of heterotic $\G_2$  structures and instanton connections.
{\em Oxford University Press, Oxford}, 503--517, 2018. 

\bibitem{DMR}
{\sc V.~del Barco, A.~Moroianu, A.~Raffero}. 
\newblock Purely coclosed G$_2$-structures on 2-step nilpotent Lie groups. 
\newblock {\em Rev.~Mat.~Complut.}~\textbf{35}, 323--359, 2022.

\bibitem{DT}
{\sc S. K. Donaldson and R. P. Thomas}, Gauge Theory in Higher Dimensions, in  {\em The Geometric Universe: Science, Geometry, and the Work of Roger Penrose} (Oxford, 1998).


\bibitem{FIUVa}
{\sc M. Fernandez, S. Ivanov, L. Ugarte and D. Vassilev}, Quaternionic Heisenberg Group and
Heterotic String Solutions with Non-Constant Dilaton in Dimensions 7 and 5.
{\em Commun. Math. Phys.} {\bf 339}, 199--219, 2015.

\bibitem{FIUV}
{\sc M. Fernandez, S. Ivanov, L. Ugarte and R. Villacampa}, Compact supersymmetric solutions of
the heterotic equations of motion in dimensions $7$ and $8$.
{\em Adv. Theor. Math. Phys.} {\bf 15}, 2011. 

\bibitem{fist}
{\sc M. A. Fiset, C. Quigley, E. E. Svanes.}
Marginal deformations of heterotic  G$_2$  sigma models
{\em J. High Energy Phys.} no. 2, 052, front matter+19 pp., 2018. 

\bibitem{FMR}
{\sc A.~Fino, L.~Mart\'in Merch\'an, A.~Raf{}fero}. 
\newblock The twisted G$_2$ equation for strong G$_2$-structures with torsion. 
\newblock {\em Pure Appl.~Math.~Q.}~\textbf{20}, 2711--2767, 2024.

\bibitem{FI}
{\sc T. Friedrich and S.  Ivanov}. 
\newblock Parallel spinors and connections with skew-symmetric torsion in string theory.
\newblock {\em Asian J.~Math.}  \textbf{6}, 303--335, 2002.

\bibitem{FI2}
{\sc T. Friedrich and S. Ivanov}, Killing spinor equations in dimension 7 and geometry of
integrable G$_2$ manifolds. {\em J. Geom. Phys.} {\bf 48}, 1--11, 2003.

\bibitem{GS}
{\sc M. Galdeano, L. Stecker}. 
\newblock The heterotic ${\rm G}_2$ system with reducible characteristic holonomy. 
\newblock \href{https://arxiv.org/abs/2403.00084}{\tt arXiv:2403.00084}


\bibitem{GMPW}
{\sc J. P. Gauntlett, D. Martelli, S. Pakis,  D. Waldram}. 
\newblock G-Structures and Wrapped NS5-Branes. 
\newblock {\em Commun.~Math.~Phys.} {\bf247}, 421--445, 2004. 

\bibitem{GMPW2}
{\sc J.P. Gauntlett, D. Martelli and D. Waldram}, Superstrings with intrinsic torsion.
{\em Phys. Rev. D} {\bf 69} (2004). 

\bibitem{Gon}
{\sc M.-P. Gong}.  
\newblock Classification of Nilpotent {L}ie Algebras of Dimension 7 (Over Algebraically Closed Fields and $\mathbb{R}$).
\newblock PhD thesis, University of Waterloo (Canada), 1998.

\bibitem{REF1}
{\sc M. Gunaydin and H. Nicolai}, Seven-dimensional octonionic Yang-Mills instanton and its
extension to an heterotic string soliton. {\em Phys. Lett. B351} {\bf 169}, 1995.

\bibitem{LS}
{\sc J.D. Lotay and H.N.S. Earp}, The heterotic $\G_2$ system on contact Calabi–Yau $7$-manifolds.
{\em Trans. Amer. Math. Soc.} {\bf 10}, 907--943, 2023. 


\bibitem{mac1}
{\sc J. McOrist, M. Sticka, E. E. Svanes.}
The physical moduli of heterotic  $\G_2$  string compactifications.
{\em J. High Energy Phys.}, no. 5, Paper No. 219, 30 pp., 2025. 

\bibitem{mac2}
{\sc J. McOrist, M. Sticka, E. E. Svanes.}
The heterotic  $\G_2$   moduli space metric.
{\em J. High Energy Phys.}, no. 11, Paper No. 16, 17 pp, 2025. 




\bibitem{Mil}
{\sc J.~Milnor.}
\newblock Curvatures of left invariant metrics on {L}ie groups.
\newblock {\em Adv.~Math.} {\bf21}, 293--329, 1976.

\bibitem{MMS22} {\sc A. Moroianu, Á Murcia, C. S. Shahbazi.}
Heterotic solitons on four-manifolds.
{\em New York J. Math.} {\bf 28}, 1463--1497, 2022.

\bibitem{MMS24} {\sc A. Moroianu, Á Murcia, C. S. Shahbazi.}
The Heterotic-Ricci flow and its three-dimensional solitons. 
{\em J. Geom. Anal.} {\bf 34}, 122, 2024. 

\bibitem{Nolle}
{\sc C. Nolle}, Homogeneous heterotic supergravity solutions with linear dilaton.
{\em J. Phys. A} {\bf 45}, 2012.

\bibitem{Nomizu} 
{\sc K. Nomizu} On the cohomology of compact homogeneous spaces of nilpotent Lie groups.
{\em Ann. of Math.} (2) {\bf 59}, 531–-538, 1954.


\end{thebibliography}
\end{document}